\documentclass[12pt]{amsart}
\usepackage{amssymb,latexsym,amsmath,amsthm,amscd}
\usepackage{setspace}
\usepackage{array}
\usepackage{float}
\usepackage[all]{xy} \xyoption{arc}
\usepackage[left=3cm,top=2cm,right=3cm,bottom = 2cm]{geometry}
\usepackage{graphicx}
\usepackage[usenames,dvipsnames]{xcolor}
\usepackage{subfig}
\usepackage{charter}
\usepackage{hyperref}

\theoremstyle{plain}
\newtheorem{thm}{Theorem}
\newtheorem{lem}[thm]{Lemma}
\newtheorem{prop}[thm]{Proposition}
\newtheorem{cor}[thm]{Corollary}

\theoremstyle{definition}
\newtheorem{dfn}[thm]{Definition}
\newtheorem{ex}[thm]{Example}

\theoremstyle{remark}
\newtheorem{rmk}[thm]{Remark}

\DeclareMathOperator{\uhp}{\mathcal{H}}

\DeclareMathOperator{\Tr}{Tr}

\DeclareMathOperator{\Aut}{Aut}

\DeclareMathOperator{\GL}{GL}
\DeclareMathOperator{\PGL}{PGL}
\DeclareMathOperator{\PSL}{PSL}
\DeclareMathOperator{\SL}{SL}

\DeclareMathOperator{\Ind}{Ind}

\DeclareMathOperator{\diag}{diag}
\DeclareMathOperator{\Li}{Li}

\newcommand*{\df}{\mathrel{\vcenter{\baselineskip0.5ex \lineskiplimit0pt
                     \hbox{\scriptsize.}\hbox{\scriptsize.}}} =}

\providecommand{\abs}[1]{\left\lvert#1\right\rvert}

\providecommand{\twomat}[4]{\left(\begin{array}{cc}#1&#2\\#3&#4\end{array}\right)}
\providecommand{\stwomat}[4]{\left(\begin{smallmatrix}#1&#2\\#3&#4\end{smallmatrix}\right)}

\newcommand{\QQ}{\mathbf{Q}}

\newcommand{\CC}{\mathbf{C}}

\newcommand{\ZZ}{\mathbf{Z}}
\newcommand{\PP}{\mathbf{P}}
\newcommand{\RR}{\mathbf{R}}

\DeclareMathOperator{\diagg}{diag}
\DeclareMathOperator{\Vir}{Vir}
\DeclareMathOperator{\qchar}{q-char}
\DeclareMathOperator{\ch}{\mathfrak{ch}}
\DeclareMathOperator{\Rad}{Rad}
\DeclareMathOperator{\VB}{VB}

\begin{document}
\title{Classification of some vertex operator algebras of rank $3$}
\author{Cameron Franc}
\email{franc@math.usask.ca}
\thanks{The first listed author was supported by an NSERC Discovery Grant.}
\author{Geoffrey Mason}
\email{gem@ucsc.edu}
\thanks{The second listed author was supported by grant  $\# 427007$ from the Simons Foundation.}
\date{}

\begin{abstract}
  We discuss the classification of strongly regular vertex operator algebras (VOAs) with exactly three simple modules whose character vector satisfies a monic modular linear differential equation with irreducible monodromy. Our Main Theorem \ref{thmmain} provides a classification of all such VOAs in the form of one infinite family of affine VOAs, one individual affine algebra and two Virasoro algebras, together with a family of eleven exceptional character vectors and associated data that we call the $U$-series. We prove that there are at least $15$ VOAs in the $U$-series occurring as commutants in a Schellekens list holomorphic VOA. These include the affine algebra $E_{8,2}$ and H\"ohn's Baby Monster VOA $\VB^\natural_{(0)}$ but the other $13$ seem to be new. The idea in the proof of our Main Theorem is to exploit properties of a family of vector-valued modular forms with rational functions as Fourier coefficients, which solves a family of modular linear differential equations in terms of generalized hypergeometric series.
\end{abstract}
\maketitle
\tableofcontents

\section{Introduction and statement of the Main Theorem}\label{SI}
It is a natural problem to classify ($2$-dimensional) rational conformal field theories, which we  conflate with the classification of (chiral) rational vertex operator algebras $V$ (VOAs).\ In order to do this one needs some \emph{invariants} of $V$. They should be computable and yet capable of reflecting enough of the structure of $V$ so that they can distinguish between isomorphism classes of VOAs, or at least come close to this ideal. In fact we work with \emph{strongly regular} VOAs $V$ \cite{MasonLattice}. Among other properties, these are simple VOAs of CFT-type which are also rational and $C_2$-cofinite.\ In particular, they have only finitely many (isomorphism classes of) simple modules.

Before continuing, let us develop some notation. If $V$ has $n$  simple modules $V\df M_0$, $M_1$, $\ldots$, $M_{n-1}$ it is convenient to say that \emph{$V$ has rank $n$}. The $q$-character of $M_i$ is defined in the usual way, namely
\[
f_i(\tau)= \Tr_{M_i} q^{L(0)-c/24}.
\]
Notation here is standard, and in particular $V$ has central charge $c$, $\tau$ lies in the complex upper half-plane $\uhp$, and $q\df e^{2\pi i \tau}$. The \emph{character vector} of $V$ is the $n$-vector
\begin{eqnarray*}
F(\tau)\df(f_0, f_1, \hdots, f_{n-1})^T,
\end{eqnarray*}
and we let $\ch_V$ denote the \emph{span} of the $f_i(\tau)$. By Zhu's Theorem \cite{Zhu}, $\ch_V$ is a right $\Gamma$-module, where $\Gamma\df \SL_2(\ZZ)$ and the action is induced by $\gamma : f_i(\tau)\mapsto f_i(\gamma\tau)$  ($\gamma \in \Gamma$).

Another way to state these facts is in the language of \emph{vector-valued modular forms} (VVMFs): there is representation $\rho : \Gamma\rightarrow \GL_n(\CC)$ such that
\[
F(\gamma\tau)=\rho(\gamma)F(\tau),
\]
which says that $F(\tau)$ is a VVMF of weight zero on $\Gamma$. For a survey of VVMFs, including their connections to Riemann-Hilbert type problems (which we consider below) but \emph{not} their applications to VOAs, we may refer the reader to \cite{FM2}.

A striking  property of the character vector is its \emph{modularity} \cite{Huang}, which may be stated as follows: the kernel of $\rho$ is a \emph{congruence subgroup} of $\Gamma$. This entails that each $q$-character $f_i(\tau)$ is a modular function of weight zero on some congruence subgroup of $\Gamma$. One might therefore think that the character vector could serve as a good invariant for $V$ of the type we are seeking. In fact experience shows that there is a more useful and more subtle invariant that we will explain here: it is a \emph{modular linear differential equation} (MLDE) cf. \cite{FM2}. For the case at hand this may be taken to be a linear differential equation of \emph{weight $k$} with modular coefficients of the form
\begin{equation}
  \label{MLDE}
(P_0D^n+P_1D^{n-1}+\cdots+P_{n-1}D+ P_n)u=0.
\end{equation}
Here, each $P_{\ell}\in \CC[E_4, E_6]$ is a holomorphic modular form of weight $k+2\ell-2n$ and $D$ is the \emph{modular derivative} defined on modular forms of weight $k$ by the formula $D_k = q\tfrac{d}{dq} - \tfrac{k}{12}E_2$. In this paper, since the character vector of a VOA is of weight zero, we require the case where $D = D_0$ and so
\[
 D^n = D_{2n-2} \circ \cdots \circ D_2 \circ D_0.
\]
Then one knows that \emph{there is an MLDE of some weight $k$ whose solution space is $\ch_V$}.

The MLDE \eqref{MLDE} may be taken as the desired invariant of $V$. Not only does it implicitly include $\ch_V$ as the space of solutions of \eqref{MLDE}, but in addition it carries a \emph{monodromy representation} defined by analytic continuation of the solutions around the singularities. Because of the special nature of the differential equation \eqref{MLDE} this monodromy is essentially the representation $\rho$ of $\Gamma$ acting on $\ch_V$.

The purpose of the present paper is to prove the analog of the Mathur-Mukhi-Sen Theorem \cite{MMS}, \cite{MNS} in rank $3$. The extra dimension gives rise to a great deal of additional complication and difficulties. Some of these were discussed in \cite{MasonFive} where our Main Theorem appeared as Problem 4. In particular, while it has long been recognized that VOAs have a strong arithmetic vein, the current proof of Main Theorem \ref{thmmain} includes an unprecedented amount of number theoretic complications.

We shall now state our main result precisely and outline its proof: we characterize strongly regular VOAs $V$ of rank $3$ whose associated MLDE \eqref{MLDE} has weight $0$ so that it takes the form
\[
(D^3+aE_4D+bE_6)u=0,\quad (a, b \in \CC).
\]
An MLDE of weight zero such as this is said to be \emph{monic}. Additionally, we assume that the monodromy $\rho$ is an \emph{irreducible} representation of $\Gamma$. With these conditions and definitions we establish the following main result:
\begin{thm}[Rank $3$ Mathur-Mukhi-Sen]
  \label{thmmain}
  Let $V$ be a strongly regular VOA with exactly three simple modules. Suppose that the $q$-characters of the simple $V$-modules furnish a fundamental system of solutions for an MLDE of order $3$ that is (i) monic, and (ii) has irreducible monodromy. Then one of the following holds:
\begin{enumerate}
\item[(a)]$V$ is isomorphic to one of the following:
\begin{eqnarray*}
&&B_{\ell,1} \quad\quad (\ell \geq 2),\\
&&A_{1,2},\\
&&\Vir(c_{3, 4}),~ \Vir(c_{2,7}).
\end{eqnarray*}
\item[(b)] $V$ lies in the $U$-series (cf. Remark \ref{r:useries}).
\end{enumerate}
(Here, and below, $\mathcal{G}_{\ell, k}$ denotes an affine algebra of type $\mathcal{G}$, rank $\ell$, and level $k$; $\Vir(c)$ is a Virasoro VOA of central charge $c$.)
\end{thm}

\begin{rmk}
  \label{r:useries}
  The $U$-series\footnote{In an earlier preprint $U$ stood for `unknown' or 'undecided'. Although the question of existence is now decided in many cases -- subject to a standard conjecture -- it is still a useful mnemonic.} refers both to $11$ sets of datum indexed by an integer $k$ in the range $0\leq k\leq 10$, and to a family of VOAs uniformly described by the data, each of which satisfies the hypotheses of Theorem \ref{thmmain}. The data arises from the residual cases in our approach to the proof of Theorem \ref{thmmain}. 

  Two VOAs in the $U$-series are well-known. These are the affine algebra $E_{8,2}$ and the baby Monster VOA $\VB^{\natural}_{(0)}$ \cite{Hoehn1}. These two VOAs correspond to $k=8$ and $k=0$ respectively.

  We will show that $13$ additional VOAs in the $U$-series, corresponding to $k=1,2,\ldots, 6$ may also be constructed as commutants in a Schellekens list VOA. This is strongly suggested by, and depends on, the work of Gaberdiel-Hampapura-Mukhi \cite{GHM} and Lin \cite{Lin}. For further details we refer the reader to Section \ref{SU}.
\end{rmk}

\begin{rmk}
\label{r:ghm}
  Since the original submission of this paper we have been able to prove that Theorem \ref{thmmain} remains true without the irreducibility assumption (ii). Were we to include details, however, it would significantly add to the length of the present paper, so we skip them here.
\end{rmk}

The idea of classifying $2$-dimensional conformal field theories is an old dream of physicists, dating from the late 1980s, and the influential paper of Moore and Seiberg \cite{MooreSeiberg} is often cited in this regard. The idea of attacking the problem based on the method of MLDEs as we have explained it was propounded by Mathur, Mukhi and Sen \cite{MMS} in 1988, where they discussed the classification of rank $2$ VOAs at the level of physical rigor. Until recently mathematicians have hesitated to get on this bandwagon, perhaps because of the lack of a sufficiently solid theory of MLDEs and VVMFs, however that trend has now reversed itself. The rank $2$ theory of Mathur-Mukhi-Sen  was put on a solid mathematical foundation in \cite{MNS}, and Tener and O'Grady have extended this in developing the theory of rank $2$ \emph{extremal} VOAs \cite{TG}. 

As for the rank $3$ theory treated here, our Main Theorem \ref{thmmain} subsumes a number of results in both the mathematical and physical literature. Hampapura and Mukhi treated the Baby Monster VOA from the MLDE perspective in \cite{HM}. This example together with $E_{8,2}$ was considered by Gerald H\"{o}hn \cite{Hoehn1}. Gaberdiel, Hampapura and Mukhi also constructed several VOAs related to, and conjecturally equal to, some VOAs in the $U$-series in their work \cite{GHM}, and in Appendix C of \cite{MMS2} one finds a discussion of the infinite series of affine algebras intervening in Theorem \ref{thmmain}. Arike, Nagatomo, Kaneko and Sakai discussed the MLDEs satisfied by these and many other affine algebras in a very useful paper \cite{AKNS}. Arike, Nagatomo and Sakai characterized some low-dimensional Virasoro algebras according to their MLDEs \cite{ANS1}, and the results of both this and a preprint of  Mason, Nagatomo and Sakai \cite{MNS}  characterizing some VOAs with $c=8$ or $16$ are special cases of Theorem \ref{thmmain}.

Theorem \ref{thmmain} is proved by exploiting the fact, proved in \cite{FM2}, that a monic MLDE of degree three can be solved in terms of generalized hypergeometric series. This solution describes an algebraic family of modular forms that vary according to choices of local exponents at the cusp for the MLDE. The important point for our analysis is that this family of modular forms has Fourier coefficients that are \emph{rational functions} of the local exponents. Since our goal is to classify specializations of the family that have Fourier coefficients that are nonnegative integers, we proceed as follows:
\begin{enumerate}
\item It is known that the monodromy representation is congruence, and in Section \ref{s:monodromy} we give a direct proof of this fact (cf. Theorem \ref{t:repclass}). Indeed, together with the results of \cite{FM3}, our results establish the \emph{unbounded denominator conjecture} for $3$-dimensional irreducible representations of the modular group (whereas \cite{FM3} treated the case of imprimitive representations). The $2$-dimensional case was proved in \cite{FM1}. The main result of Section \ref{s:monodromy} details the $3$-dimensional irreducible representations of $\Gamma$ and makes precise some computations from \cite{BH}.
\item Next in Section \ref{s:positivity} we study the divisors of the first nontrivial Fourier coefficients of the character vector. The signs of the coefficients are constant on the connected components of the complement of the divisors, so that we may restrict our search to a reasonably small and manageable subset of all possible parameters. This is explained in Theorem \ref{t:positivity} and it is displayed graphically in Figure \ref{f:positiveregion} on page \pageref{f:positiveregion}.
\item  The remaining characters are tested for integrality in Section \ref{s:fibers}, where we use arithmetic properties of hypergeometric series discussed in \cite{FGM}. The output is one infinite family of possible character vectors, in addition to a finite list of additional exceptional possibilities tabulated in Figures \ref{f:full57} and \ref{f:full2} on pages \pageref{f:full57} and \pageref{f:full2}.
\item Next in Section \ref{s:sieve} we apply further tests arising from the theory of VOAs, namely symmetry of the $S$-matrix and the Verlinde formula \cite{Huang}, to whittle the remaining examples down to the statement of our Main Theorem \ref{thmmain}.
  \item In Section \ref{s:quadraticfamily} we complete the proof of Theorem \ref{thmmain} by discussing the infinite family of possible character vectors, and we explain how they are in fact realized by VOAs.
  \end{enumerate}
  Finally, Section \ref{SU} discusses the $U$-series.

 It is worth noting that a significant feature of our proof, indeed, of the general approach to VOA classification through VVMFs and MLDEs, is the difficulty in distinguishing VOAs that have \emph{more} than three simple modules but which satisfy $\dim \ch_V = 3$. A good part of our proof goes through under the weaker assumption that  $\dim \ch_V=3$. But in order to readily apply the symmetry of the $S$-matrix we must assume that $V$ has rank $3$. A similar circumstance already revealed itself in \cite{MNS} in the rank $2$ case.

It is well-known that the VOAs listed in Theorem \ref{thmmain} are strongly regular, have exactly three simple modules, and satisfy the other conditions of Theorem \ref{thmmain}. For the case of the affine algebras this is easily deduced from \cite{AKNS} and for the Virasoro algebras, see e.g., \cite{LL}. In Table \ref{Table1} on page \pageref{Table1} we have collected some relevant data for these VOAs.
 
  \emph{Acknowledgements}. We are indebted to the following individuals for helpful discussions, supplying references, and for answering our questions: Chongying Dong, Gerald H\"{o}hn, Ching Hung Lam, Sunil Mukhi, Kiyokazu Nagatomo, and Ivan Penkov. We also thank the referee for their detailed comments.



\begin{table}
\renewcommand{\arraystretch}{1.4}
\begin{tabular}{|c|c|c|}   \hline
VOA & $c$& $h_1, h_2$\\ \hline
$A_{1, 2}$&$\tfrac{3}{2}$& $\tfrac{3}{16}$, $\tfrac{1}{2}$\\
$B_{\ell,1},~ \ell \geq 2$&$\tfrac{2\ell+1}{2}$& $\tfrac{2\ell+1}{16}$, $\tfrac{1}{2}$\\
$E_{8,2}$&$\tfrac{31}{2}$& $\tfrac{3}{2}$, $\tfrac{15}{16}$\\
$\Vir(c_{2, 7})$& $-\tfrac{68}{7}$ & $-\tfrac{2}{7}$, $-\tfrac{3}{7}$\\
$\Vir(c_{3, 4})$&$\tfrac{1}{2}$& $\tfrac{1}{16}$, $\tfrac{1}{2}$\\ \hline
\end{tabular}
\caption{Some VOAs with three simple modules}\label{Table1}
\end{table}

\section{Background on VOAs}
\label{s:voas}
\subsection{The invariants $c$, $\tilde{c}$, $\ell$}
\label{SScl}
In this Subsection we discuss the numerical  invariants $c$, $\tilde{c}$ and $\ell$  associated with a strongly regular VOA $V{=}(V, Y, \mathbf{1}, \omega)$ that we will use in the following Sections. For additional background and discussion we refer the reader to \cite{MasonLattice}. We note that one of our results, Theorem \ref{thmctildec}, is new and improves upon an inequality of Dong and Mason \cite{DMRational}. In this Subsection we do \emph{not} make any assumptions about the number of irreducible modules that $V$ may have, merely that they are finite in number.

The invariant $c$, the \emph{central charge} of $V$, is of course well-known and a standard invariant that is part of the definition of $V$. We sometimes write $c=c_V$. Because $V$ is strongly regular then it has only finitely many (isomorphism classes of) irreducible modules, which we label as $M_0$, $M_1$, $\hdots$, $M_{n-1}$. And because $V$ is necessarily simple then one of the $M_i$ is isomorphic to $V$, and we will always choose notation so that $V=M_0$. Each $M_i$ has a \emph{conformal weight} $h_i$ defined to be the least nonvanishing eigenvalue of the $L(0)$-operator. Thus $M_i$ has (conformal) grading $M_i=\oplus_{n\geq 0}M_{n+h_i}$, and the \emph{$q$-character} of $M_i$ is defined by
\begin{equation}
\label{qchardef}
\qchar M_i\df \Tr_{M_i}q^{L(0)-c/24}=q^{h_i-c/24}\sum_{n\geq 0} \dim (M_i)_{n+h_i}q^n.
\end{equation}
Throughout this paper we use the notation
\[
m\df \dim V_1.
\]
In particular, and as part of the definition of a strongly regular VOA, we have
\[
\mbox{q-char}\ V\df \Tr_{V}q^{L(0)-c/24}=q^{-c/24}\sum_{n\geq0} \dim V_{n}q^n=q^{-c/24}(1+mq+\hdots)
\]
We note that $c$ and each $h_i$ lies in $\QQ$, the field of rational numbers \cite{DLMModular}.

The \emph{effective central charge} $\tilde{c}=\tilde{c}_V$ is defined as
\begin{equation*}
\tilde{c}\df c-24h_{\textrm{min}}
\end{equation*}
where $h_{\textrm{min}}$ is the \emph{least} of the rational numbers $h_i$. Note that $h_0=0$ by our convention, in particular we always have $c \leq \tilde{c}$, and of course $\tilde{c} \in \QQ$. The effective central charge will play an important r\^{o}le in our efforts to characterize certain VOAs. Its relevance is partially explained by noticing that among the set of $q$-characters \eqref{qchardef}, the \emph{least} of the leading $q$-powers is precisely $q^{-\tilde{c}/24}$.

The invariant $\ell$ is defined to be the \emph{Lie rank} of $V_1$. It is well-known that the  homogeneous space $V_1$ of a strongly regular VOA carries the structure of a Lie algebra with respect to the bracket $[ab] \df a(0)b$. Indeed, $V_1$ is a \emph{reductive Lie algebra} \cite{DMRational}. Then $\ell$ is the dimension of a Cartan subalgebra of $V_1$. The following equality involving $\ell$ and $\tilde{c}$ is known (loc. cit.)
\begin{eqnarray*}
\tilde{c} \geq \ell,\quad \textrm{and} \quad \tilde{c}=0\quad \textrm{only if}\quad V=\CC.
\end{eqnarray*}
In particular, if $V\neq\CC$ then at least one of the $q$-characters \eqref{qchardef} has a \emph{pole} at $q=0$.

In \cite{DMRational} it was shown that the simultaneous equalities $c=\tilde{c}=\ell$ characterize lattice VOAs $V_{\Lambda}$ (some positive-definite even lattice $\Lambda$), and the authors expected that the equality $\tilde{c}=\ell$ should suffice to characterize this class of VOAs. Here, we shall prove this and more.
\begin{thm}
  \label{thmctildec}
  Suppose that $V$ is a strongly regular VOA satisfying $\tilde{c}<\ell+1$. Then $c=\tilde{c}$. In particular, if $\tilde{c}=\ell$ then $V$ is isomorphic to a lattice theory $V_{\Lambda}$ for some even lattice $\Lambda$.
\end{thm}
\begin{proof}
  We shall do this by modifying the proof of Theorem 7 of \cite{MasonLattice}. Theorem 1 of \cite{MasonLattice} says that $V$ contains a subVOA $T \subseteq V$ with the following  properties:
\begin{enumerate}
  \item[(a)] $T$ is a \emph{conformal subalgebra} of $V$, i.e., $V$ and $T$ have the \emph{same} Virasoro element, and in particular $c_V =c_T$;
  \item[(b)] $T$ is a tensor product $T \cong W \otimes C$ of a pair of subVOAs $W$ isomorphic to a lattice theory $V_\Lambda$ of rank $\ell$, and $C$ isomorphic to a discrete series Virasoro VOA $\Vir(c_{p,q})$.
 \end{enumerate}
    Actually, in this set-up we have $0\leq c_{p,q}<1$, so that $C$ is in the \emph{unitary} discrete series. We have the following series of inequalities that proves what is needed:
 \begin{eqnarray*}
 c_V \leq \tilde{c}_V \leq \tilde{c}_T = \tilde{c}_{V_{\Lambda}}\tilde{c}_{\Vir(c_{p,q})} = c_{V_{\Lambda}}c_{\Vir(c_{p,q})} = c_T = c_V.
\end{eqnarray*}
Here, the first inequality was pointed out before; the second inequality holds because $T$ is a conformal subalgebra of $V$; the first equality holds because effective central charge is multiplicative over tenor products; the second equality holds because central charge and effective central charge \emph{coincide} for both lattice theories and unitary discrete series of Virasoro VOAs; and finally the third equality holds because central charge is also multiplicative over tensor products.
\end{proof}

As a corollary of this proof, we have:
\begin{cor}
  \label{corc/2}
  Suppose $2c \in \ZZ$. Then one of the following holds:
  \begin{enumerate}
\item[(a)] $\tilde{c}-\ell \geq 1$;
\item[(b)] $\tilde{c}-\ell =\tfrac{1}{2}$ and $\tilde{c}=c$;
\item[(c)] $\tilde{c}=\ell =c$.
\end{enumerate}
\end{cor}
\begin{proof} If (a) is false then $\tilde{c}-\ell{<}1$ and Theorem \ref{thmctildec} tells us that $\tilde{c}=c$. Moreover, as the proof shows, $V$ contains a conformal subVOA isomorphic to $V_{\Lambda} \otimes \Vir(c_{p,q})$ where $\Lambda$ is an even lattice of rank $\ell$. The Virasoro tensor factor lies in the unitary discrete series because its central charge is less than $1$.\ It follows that there is an integer $z \geq 2$ such that 
\[
c=\ell{+}1{-}\tfrac{6}{z(z+1)}.
\]
Because $2c \in\ZZ$, this can only happen if $z = 2$ or $3$. These two possibilities correspond to (c) and (b) respectively. This completes the proof.
\end{proof}

\subsection{The space $\ch_V$ of $q$-characters }
\label{SSchV}
We retain the notation of the previous Subsection and in particular $V$ denotes a strongly regular VOA. For the rest of this Subsection we assume that $\dim \ch_V=3$ and that $\ch_V$ is the solution space of a \emph{monic} MLDE that has an \emph{irreducible} monodromy representation $\rho:\SL_2(\ZZ) \rightarrow \GL(\ch_V)$ cf. \cite{FM2,FM3}. In particular, the MLDE in question must look like
\begin{equation}
  \label{MLDE1}
(D_0^3+aE_4 D_0+b E_6)f=0,\quad a, b \in\QQ.
\end{equation}
Here $E_4$ and $E_6$ are the holomorphic Eisenstein series of level one and weights $4$ and $6$ respectively, normalized so that the constant terms are $1$. In \cite{FM2,FM3} this MLDE arose as the differential equation satisfied by forms of minimal weight for $\rho$. It is worth noting that the form of minimal weight for a given representation (and choice of exponents for $\rho(T)$) is rarely $0$, so that the modular forms arising as character vectors of VOAs are almost never of minimal weight. Nevertheless, the computations of \cite{FM2,FM3} may be used to study the solutions of equation \eqref{MLDE1}, and we discuss this next.

Because $\rho$ is irreducible it is easy to see, and it is a special case of a result of Tuba-Wenzl \cite{TW}, that the $T$-matrix $\rho(T)$ has \emph{distinct eigenvalues}. A general result \cite{DLMModular} says that $\rho(T)$ has \emph{finite order} (although in the present context this can be seen more directly), and in any case there are \emph{distinct} $r_0, r_1, r_2 \in \QQ \cap[0,1)$ and a basis of $\ch_V$ such that if we assume that $\rho$ is written with respect to this choice of basis then
\begin{equation}
  \label{Texps}
\rho(T)= \left(\begin{matrix}e^{2\pi ir_0} & 0 & 0 \\0 & e^{2\pi ir_1} & 0 \\0 & 0 & e^{2\pi ir_2}\end{matrix}\right)
\end{equation}

Because $\ch_V$ \emph{spans} the solution space of the MLDE \eqref{MLDE1} then it is easy to see that the three eigenfunctions for $\rho(T)$ may be taken to be the $q$-characters of three irreducible $V$-modules, and that moreover we may take the first of these $V$-modules to be $V=M_0$. Let $M_1, M_2$ be the other two irreducible $V$-modules. The \emph{character} vector of $V$ is thus the vector-valued modular form
\[
F(\tau)\df\left(\begin{matrix}f_0(\tau) \\ f_1(\tau) \\ f_2(\tau)\end{matrix}\right)
\]
where
\begin{eqnarray*}
f_0(\tau)&\df&\Tr_{V}q^{L(0)-c/24}= q^{-c/24}+O(q^{1-c/24}),\\
f_i(\tau)&\df&\Tr_{M_i}q^{L(0)-c/24}= \dim (M_i)_{h_i}q^{h_i-c/24}+O(q^{1+h_i-c/24}), \quad i=1, 2,
\end{eqnarray*}
and furthermore
\begin{eqnarray*}
r_0 &\equiv& -\tfrac{c}{24} \pmod{\ZZ},\\
r_i&\equiv& h_i-\tfrac{c}{24}\pmod{\ZZ},\quad i=1, 2.
\end{eqnarray*}

There is an important identity that accrues from the special shape of the MLDE \eqref{MLDE1}, namely:
\begin{lem}
  \label{lemc}
  The following hold:
\begin{enumerate}
\item[(a)] $c=8(h_1+h_2-\tfrac{1}{2})$,
\item[(b)] $\det\rho(T)=-1$.
\end{enumerate}
\end{lem}
\begin{proof}
  (a) The \emph{indicial equation} (at $\infty$) for \eqref{MLDE1} is readily found to be 
\[
x^3-\tfrac{1}{2}x^2+(a+\tfrac{1}{18})x+b=0,
\]
and in particular the corresponding indicial roots sum to $\tfrac{1}{2}$. However these roots are the leading exponents of $q$ for the functions  $f_i(\tau)\ (i{=}0, 1, 2)$, namely $-\tfrac{c}{24}$, $h_1-\tfrac{c}{24}$ and $h_2-\tfrac{c}{24}$. Part (a) of the Lemma follows immediately.

As for (b), using (a) we have $\det\rho(T)=e^{2\pi i(r_0+r_1+r_2)}=e^{2\pi i(h_1+h_2-c/8)}=-1$.
\end{proof}

\subsection{Things  hypergeometric}
\label{SShypergeo}
It is fundamental for this paper that with a suitable change of variables the MLDE \eqref{MLDE1} becomes a generalized hypergeometric differential equation that is solved by generalized hypergeometric functions ${_3}F_2$. This circumstance is explained in \cite{FM2,FM3} where, in particular, motivation for using the level $1$ hauptmodul $K:\uhp\cup \{\infty\} \longrightarrow \PP^1(\CC)$ defined by
\[
K=\frac{E_4^3}{E_4^3-E_6^2} = \frac{1728}{j}=1728q+\cdots
\]
is provided. The well-known paper of Beukers and Heckman \cite{BH}, which describes the monodromy of all generalized hypergeometric differential equations of all orders, may also be referenced here. We shall only need the case of order $3$. In terms of the differential operator $\theta_K \df K\tfrac{d}{dK}$, the MLDE \eqref{MLDE1} becomes (cf. \cite{FM2}, Example 15)
\begin{eqnarray*}
\left(\theta_K^3 -\tfrac{2K+1}{2(1-K)}\theta_K^2+\tfrac{18a+1-4K}{18(1-K)}\theta_K+\tfrac{b}{1-K}\right)f=0.
\end{eqnarray*}
Following \cite{BH}, Section 2, upon multiplying the previous differential operator by $1-K$ we obtain the following alternate formulation:
\[
\left\{(\theta_K+\beta_1-1)(\theta_K+\beta_2-1)(\theta_K+\beta_3-1)-K(\theta_K+\alpha_1)(\theta_K+\alpha_2)(\theta_K+\alpha_3)\right\}f=0
\]
for scalars $\alpha_1, \hdots, \beta_3$ satisfying
\begin{eqnarray}
  \label{localindices1}
&&\alpha_1+\alpha_2+\alpha_3=1,\notag\\
&&\alpha_1\alpha_2+\alpha_1\alpha_3+\alpha_2\alpha_3=\tfrac{2}{9}, \notag \\
&&\alpha_1\alpha_2\alpha_3=0,\\
&&\beta_1+\beta_2+\beta_3 -3=-\tfrac{1}{2},\notag\\
&&(\beta_1-1)(\beta_2-1)+(\beta_1-1)(\beta_3-1)+(\beta_2-1)(\beta_3-1)=\tfrac{1}{18}+a,\notag\\
&&(\beta_1-1)(\beta_2-1)(\beta_3-1)=b.\notag
\end{eqnarray}
The local indices at the three singularities $K=0, 1, \infty$ are as follows
\begin{eqnarray}
  \label{localindices2}
&1-\beta_1,~ 1-\beta_2,~ 1-\beta_3 &\textrm{at } K=0,\notag\\
&\alpha_1=0,~ \alpha_2=\tfrac{1}{3},~ \alpha_3=\tfrac{2}{3}&\textrm{at } K=\infty,\\
&0,~ 1,~ \tfrac 12& \textrm{at } K=1.\notag
\end{eqnarray}

Inasmuch as $K(\infty)=0$, $K(e^{2\pi i/3})=\infty$ and  $K(i)=1$, these sets of indices correspond to the local monodromies $\rho(T)$, $\rho(R)$, $\rho(S)$ respectively (where $R = -ST$ -- see Section \ref{ss:generalities} below for the notation). For example, we see that 
\begin{align*}
\det\rho(T)&=-1, & \det\rho(R)&=1, & \det\rho(S)&=-1.
\end{align*}

The generalized hypergeometric function ${_3}F_2$ is defined by
\begin{eqnarray*}
{_3}F_2(a_1, a_2, a_3; b_1, b_2;z)\df 1+\sum_{n\geq 1} \frac{(a_1)_n(a_2)_n(a_3)_n}{(b_1)_n(b_2)_n} \frac{z^n}{n!},
\end{eqnarray*}
where $(t)_n\df t(t+1)\hdots(t+n-1)$ is the rising factorial. Here, $a_1, a_2, a_3, b_1, b_2$ are arbitrary scalars subject to the exclusion that $b_1, b_2$ are \emph{not} nonpositive integers. With this convention, ${_3}F_2$ converges for $\abs{z}<1$, has singularities at $z=0, 1, \infty$, and is defined by analytic continuation elsewhere.

With the assumption that \emph{no two of the  $\beta_i$ differ by an integer}, a  fundamental system of solutions near $K=0$ of our hypergeometric differential equation is given as in equation (2.9) of \cite{BH} by
\begin{eqnarray}
  \label{Hypergeosolns}
&&K^{1-\beta_1}{_3}F_2(1+\alpha_1-\beta_1, 1+\alpha_3-\beta_1, 1+\alpha_1-\beta_1;1+\beta_2-\beta_1, 1+\beta_3-\beta_1;K)\notag\\
&&K^{1-\beta_2}{_3}F_2(1+\alpha_1-\beta_2, 1+\alpha_3-\beta_2, 1+\alpha_1-\beta_2;1+\beta_1-\beta_2, 1+\beta_3-\beta_2;K)\\
&&K^{1-\beta_3}{_3}F_2(1+\alpha_1-\beta_3, 1+\alpha_3-\beta_3, 1+\alpha_1-\beta_3;1+\beta_1-\beta_3, 1+\beta_2-\beta_3;K)\notag
\end{eqnarray}
In this way one obtains explicit and useful formulas for the character vector $F(\tau)$ of Subsection \ref{SSchV}. We shall exploit this hypergeometric formula, which describes a family of vector-valued modular forms varying over a space of indices for the differential equation \eqref{MLDE1}, to classify possible character vectors of VOAs having exactly $3$ irreducible modules and irreducible monic monodromy. The key points are that the Fourier coefficients of this family are rational functions in the local indices, and that the arithmetic behaviour of these coefficients are very well-studied, cf. \cite{Dwork}, \cite{FGM}.


\section{Classification of the monodromy}
\label{s:monodromy}
The purpose of this Section is to enumerate the possible monodromies $\rho$ of the MLDE attached to $\ch_V$ (cf. Subsection \ref{SSchV}). Essentially, this amounts to cataloguing certain equivalence classes of $3$-dimensional irreducible representations of $\SL_2(\ZZ)$. We shall do this, and in particular we will calculate the possible sets of exponents $r_i$ of the $T$-matrix \eqref{Texps}. These rational numbers (and in particular their \emph{denominators}) will play an important r\^{o}le in the arithmetic analysis in later Sections.

In \cite{BH} Beukers and Heckman described the monodromy of all hypergeometric functions ${}_nF_{n-1}$, so in principle they already solved the problem that concerns us in this Section because, as we have explained, our MLDE is hypergeometric. However there are several reasons why we prefer to develop our results from first principles. Firstly, the results of Beukers and Heckman are couched indirectly in terms of what they refer to as \emph{scalar shifts}, making their general answer that applies to all ranks too imprecise for our specific purpose. Secondly, they work with representations of the free group of rank $2$ whereas our monodromy groups factor through the modular group $\SL_2(\ZZ)$. So the question of the \emph{modularity} of $\rho$ does not arise in \cite{BH}. Finally, we anticipate that the details of our explicit enumeration will be useful in further work involving MLDEs of order $3$.
 
 Some of the main arithmetic results are summarized in the following:
 \begin{thm}
   \label{t:repclass}
   Let $V$ be a strongly regular VOA $V$ and suppose that the third order MLDE \eqref{MLDE1} associated with $\ch_V$ is monic with \emph{irreducible} monodromy representation $\rho$. Then $\rho$ is a \emph{congruence representation}, and one of the following holds:
   \begin{enumerate}
   \item $\rho$ is \emph{imprimitive} and both $h_1$ and $h_2$ are rational with denominators dividing $16$. Moreover, either
     \begin{enumerate}
     \item[(a)] one of $h_1$ or $h_2$ lies in  $\tfrac{1}{2}\ZZ$ or
     \item[(b)] the denominators of $h_1$ and $h_2$  are equal to each other
     \end{enumerate}
   \item $\rho$ is \emph{primitive} and the denominators of $h_1$ and $h_2$ are both equal to each other and to one of $5$ or $7$.
   \end{enumerate}
  \end{thm}
  
  We describe how to classify the representations of Theorem \ref{t:repclass}, and give more detailed information about them, in the following sections.
  
  \subsection{Some generalities}
  \label{ss:generalities}
  We begin with some general facts about $\Gamma$ and the representation $\rho$ that we shall need. 

Let $\Gamma\df \SL_2(\ZZ)$ and let $U$ be the left $\CC[\Gamma]$-module furnished by the representation $\rho$ of $\Gamma$ associated to our MLDE \eqref{MLDE1}.In effect, $U=\ch_V$, though this particular realization of $U$ will be unhelpful in this Subsection. We use the following notation for elements in $\Gamma$:
\begin{align*}
R&\df \twomat 01{-1}{-1},&  S&\df \twomat 0{-1}10, & T&\df \twomat 1101. 
\end{align*}

\begin{lem}
  \label{lemS2}
  The following hold:
\begin{enumerate}
\item[(a)] If $\gamma\in\Gamma$ then $\det\rho(\gamma)=\pm1$.
\item[(b)] $\rho(S^2)=I$.
\end{enumerate}
\end{lem}
\begin{proof}
  Because $\rho$ is irreducible then $\rho(R)$ has the $3$ cube roots of unity as eigenvalues, and in particular $\det\rho(R)=1$. However $\Gamma=\langle R, T\rangle$, and we have seen in Lemma \ref{lemc} part (b) that $\det\rho(T)=-1$. Now part (a) of the present Lemma follows.

To prove part (b) assume that it is false. Then $\rho(S^2)=-I$, and it follows from (a) that there is a subgroup $G\unlhd\Gamma$ of index $2$ such that $\Gamma=G\times\langle S^2\rangle$. But this is impossible, because $G$ must contain the congruence subgroup $\Gamma(2)$, whereas $S^2\in\Gamma(2)$. This completes the proof of the Lemma.
\end{proof}

Part (b) informs us that $\rho$ is an \emph{even} representation, i.e., it factors through the quotient $\PSL_2(\ZZ)\df \Gamma/\langle \pm I\rangle$. Furthermore, we have
\begin{cor}
  \label{cordet1}
  The subgroup of $\rho(\Gamma)$ that acts on $W$ with determinant $1$ has index $2$.
\end{cor}
\begin{proof}
  This follows from Lemmas \ref{lemS2}(a) and \ref{lemc}(b).
\end{proof}

The next result is well-known. We give a proof for completeness.

\begin{lem}\label{thmprimcong}
  The following hold:
\begin{enumerate}
\item[(a)] Suppose that $N\unlhd\Gamma$ and that $\Gamma/N\cong L_2(7)$. Then  $N=\Gamma(7)\langle S^2\rangle$.
\item[(b)] $A_6\cong L_2(9)$ is not a quotient of $\Gamma$.
\end{enumerate}
\end{lem}
\begin{proof} The proofs of each of these assertions are essentially the same. We deal with (a) and skip the proof of (b). We may, and shall, calculate in the group $\Gamma/\{\pm I\}$.

Part (a) is essentially explained by the  \emph{automorphism group} $\PGL_2(7)$ of $L_2(7)$, which has order $336$.

Count ordered pairs of elements of orders $2$ and $3$ that generate the abstract group $L_2(7)$: if this set is denoted by $X$, we claim that $X$ is a \emph{$\PGL_2(7)$-torsor}, i.e., $\PGL_2(7)$ acts transitively (by conjugation) on $X$ and $\abs{X}=\abs{\PGL_2(7)}$. The action is evident, so it suffices to check the cardinality of $X$.

For example, the total number of pairs of elements of order $2$ and $3$ respectively equal $21\cdot 56$, whereas the number of $S_3$-pairs is $6\cdot 28$, the number of $A_4$-pairs is $2\cdot 7\cdot 24$, and the number of $S_4$-pairs is $2\cdot 7\cdot 24$. Therefore we find that the number of $L_2(7)$-pairs is $21\cdot 56- 12(14+28+28)=336$.

Finally, let $\nu:\Gamma/\{\pm I\}\rightarrow L_2(7)$ be reduction mod $7$, and let $\varphi:\Gamma/\{\pm I\}\rightarrow L_2(7)$ be any surjection.
\[
\xymatrix{
\Gamma/\{\pm I\}\ar[rr]^{\nu}\ar[rrd]_{\varphi}&&L_2(7)\ar[d]^{\alpha}\\
&&L_2(7)}
\]
Because $X$ is a $\PGL_2(7)$-torsor, there is $\alpha \in \PGL_2(7)$ that makes the diagram commute. Therefore, $\varphi = \alpha \circ \nu$ has kernel $\Gamma(7)\langle S^2\rangle /\langle S^2\rangle$. This completes the proof of part (a) of the Lemma.
\end{proof}

\subsection{The imprimitive case}
\label{Simprim}
Suppose that $N \unlhd \Gamma$ is a normal subgroup. Suppose further that the \emph{restriction} $U|_N$ of $U$ to $N$ is \emph{not} irreducible. Then there is a direct sum decomposition into $1$-dimensional $N$-submodules
\[
U|_N \cong U_0\oplus U_1\oplus U_2
\]
and there are  just two possibilities for the \emph{Wedderburn structure}, namely
\begin{enumerate}
\item[(i)] (One Wedderburn component) the $W_j$ are pairwise \emph{isomorphic} as $N$-modules;
\item[(ii)] (Three Wedderburn components) the $W_j$ are pairwise \emph{nonisomorphic} as $N$-modules, and they are transitively permuted among themselves by the action of $\Gamma$.
\end{enumerate}
Care is warranted because the $U_j$ may \emph{not} be the three $T$-eigenspaces. If case (ii) pertains, the representation $\rho$ is called \emph{imprimitive}. Otherwise, it is \emph{primitive}.

\begin{lem}
  \label{lemZ}
  Suppose that $N$ has one Wedderburn component. Then $\rho(N) \subseteq Z(\rho(\Gamma))$ and $\rho(N)$ is isomorphic to a subgroup of $\ZZ/6\ZZ$.
 \end{lem}
 \begin{proof}
   By hypothesis, each element $\gamma \in N$ is such that $\rho(\gamma)$ acts on each $W_j$ as multiplication by the \emph{same} scalar. In other words, $\rho(\gamma)$ is a scalar matrix. As such it lies in the center $Z(\rho(\Gamma))$. This proves the first assertion of the Lemma. Suppose that $\lambda$ is the eigenvalue for such a $\rho(\gamma)$. Then we must have $\lambda^6=1$ by Corollary \ref{cordet1}, and the second assertion of the Lemma follows.
\end{proof}

We now assume that $\rho$ is imprimitive, and choose a maximal element $K$ in the poset of normal subgroups $K_1 \unlhd \Gamma$ with the property that $W|_{K_1}$ is not irreducible. Let the Wedderburn decomposition be
\begin{eqnarray*}
U|_K= W_0\oplus W_1 \oplus W_2.
\end{eqnarray*}
Note that elements of $K$ are represented by \emph{diagonal matrices}, whence $\rho(K)$ is \emph{abelian}.

By assumption, $\Gamma$ permutes the subspaces $W_j$ among themselves and acts transitively on this set. The \emph{kernel} of this action is a normal subgroup leaving each $W_j$ invariant, and by the maximality of $K$, it is none other than $K$ itself. Hence $\Gamma/K$ is isomorphic to one of $\ZZ/3\ZZ$ or $S_3$, being a transitive subgroup of $S_3$ in its action on $3$ letters.

It follows from the previous paragraph that one of the powers $T^s$ ($s=1, 2, 3$) lies in $K$. It is well-known (e.g., \cite{KLN}) that the \emph{normal closure} of $T^s$ in $\Gamma$ is the principal congruence subgroup $\Gamma(s)$. Hence $\Gamma(s) \subseteq K$. Now note that because $K \neq \Gamma$ then $s \neq 1$.

Next we show that the assumption $\Gamma/K \cong \ZZ/3\ZZ$ leads to a contradiction, so assume it is true. Then $K$ is the unique normal subgroup of index $3$, and as such it has just three classes of subgroups of order $4$ which generate $K$. It follows that $K/K'\langle S^2\rangle \cong  (\ZZ/2\ZZ)^2$. But $\rho(K)$ is abelian, hence $\rho(K) \cong  (\ZZ/2\ZZ)^2$, $K=\Gamma(3)\langle S^2\rangle$, and $\Gamma/K \cong A_4$. But then $\rho(T)$ has order $3$, contradicting Lemma \ref{lemc}(b).

This reduces us to the Case when $\Gamma/K \cong S_3$. Suppose also that $s=3$. Then $R$ and $T$ jointly generate a subgroup of index $2$ in $\Gamma$, a contradiction because they are generators of $\Gamma$. It follows that $s=2$. In this case we must have $K=\Gamma(2)$ because $\Gamma/\Gamma(2) \cong S_3$. Now $\Gamma(2)/\langle S^2\rangle$ is a free group of rank $2$. Therefore because $\rho(K)$ is abelian it is a homocyclic quotient of $\ZZ^2$ (remember that $\rho(S^2)=I$). Now because $\rho(T)$ has distinct eigenvalues, then it \emph{cannot} have order $2$. Therefore $\rho(T^2)$ is a nonidentity torsion element of $\rho(K)$. This implies that $\rho(K) \cong (\ZZ/t\ZZ)^2$ for some integer $t$, and in particular $\rho(\Gamma)$ is \emph{finite} (of order $6t^2$).
 
At this point we have maneuvered ourselves into a position where we can apply the results of \cite{FM3} concerning finite-image, imprimitive, irreducible representations of $\Gamma/\langle S^2\rangle$. Indeed, setting $H =\Gamma_0(2)$, $\rho$ is an induced representation $\rho =\Ind_H^{\Gamma} \chi$ for some linear character
\[
\chi:\Gamma_0(2) \rightarrow\CC^{\times}.
\]
of \emph{finite order}. In  the notation of \cite{FM3}, there is a positive integer $n$ and a primitive $n^{\textrm{th}}$ root of unity $\lambda$ such that
\begin{align*}
\chi(U)& =\lambda,&   \chi(V)&=1, & \chi(S^2) &=1,
\end{align*}
where the images of $U\df \stwomat 1021$ and  $V \df \stwomat {-1}{1}{-2}{1}$ generate the abelianization of $H/\langle S^2\rangle$. In \cite{FM3} $\chi$ takes the value $\epsilon =\pm 1$ on $V$, however the condition $\det\rho(T)=-1$ demands that $\epsilon=1$. Furthermore, the irreducibility of $\rho$ implies that $n\neq 1$ or $3$. 
\begin{prop}
  \label{propimprimreps}
The following hold:
\begin{enumerate}
\item[(a)] $\rho$ is a \emph{congruence representation}, i.e., $ker\rho$ is a congruence subgroup, and all elements in $\ch_V$ are modular functions of weight $0$ and level $2n$;
\item[(b)] $n \mid 24$ and $n\neq 1,3$.
\end{enumerate}
\end{prop}
\begin{proof}
  By construction, $\ch_V$ is spanned by functions having $q$-expansions with \emph{integral Fourier coefficients}. Now the  Proposition is essentially a restatement of Theorem 21 of \cite{FM3} 

The only assertion not explicitly stated in \cite{FM3} is the statement that the level is $2n$. This amounts to showing that $\rho(T)$ has order $2n$, and this is follows from a knowledge of the eigenvalues of $\rho(T)$, which are as follows (\cite{FM3}, Proposition 2): 
\begin{equation}
  \label{eigendata}
\{\lambda, \pm \sigma\} \quad \textrm{where}\quad \sigma^2=\bar{\lambda}.
\end{equation}
\end{proof}

From Proposition \ref{propimprimreps} together with \eqref{eigendata}, there is an \emph{even divisor} $n$ of $24$ and an integer $k$ coprime to $n$ such that the eigenvalues of $\rho(T)$ are $\{e^{2\pi i k/n}, e^{-2\pi i k/2n}, e^{2\pi i(n- k)/2n}\}$. The three \emph{exponents} occurring here are equal $\pmod{\ZZ}$, and in some order, to the exponents $\{r_0, r_1, r_2\}$ occurring in \eqref{Texps}. These in turn are equal $\pmod{\ZZ}$, and in the same order, to $\{-\tfrac{c}{24}, h_1{-}\tfrac{c}{24}, h_2{-}\tfrac{c}{24}\}$.

It follows that $\{h_1, h_2\}$ is congruent $\pmod{\ZZ}$ to one of $\{ -\tfrac{3k}{2n}, \tfrac{n-3k}{2n}\}$,  $\{\tfrac{3k}{2n}, \tfrac{1}{2}\}$ or  $\{\tfrac{3k+n}{2n}, \tfrac{1}{2}\}$. Because $n$ is an even divisor of $24$, all of the rational numbers involved here have denominators equal to $2$, $4$, $8$ or $16$ and in fact we obtain the following more precise result:
\begin{prop}
\label{imprimdenoms}
One of the following holds:
\begin{enumerate}
\item[(a)] One of $h_1$ or $h_2$ is an element of $\tfrac{1}{2}+\ZZ$, and the other has denominator equal to $4, 8$ or $16$;
\item[(b)] The denominators of $h_1$ and $h_2$ are equal, and both are equal to $4,8$ or $16$.
\end{enumerate}
Furthermore, we always have $2c\in\ZZ$, and in particular the conclusions of Corollary \ref{corc/2} apply.
\end{prop}
\begin{proof}
  The assertion regarding the central charge $c$  follows from (a) and (b) together with Lemma \ref{lemc}(a). The Lemma follows.
\end{proof}

\subsection{The primitive case}
\label{SSprim}
The purpose of this Section is to establish results that parallel those of Subsection \ref{Simprim} but now in the case that $\rho$ is \emph{primitive}. This means that if $N \unlhd \Gamma$ then either $U|_N$ is irreducible, or else $N$ is a central subgroup of order dividing $6$ (cf. Lemma \ref{lemZ}). We assume that this holds throughout this Subsection.

In the imprimitive case we were able to rely on the results of \cite{FM3} to restrict the possibilities for $\rho$ to a manageable list. For the case that now presents itself, we will prove
\begin{prop}
  \label{propprimreps}
Suppose that $\rho$ is \emph{primitive}. Then 
\[
\rho(\Gamma)\cong L_2(p) \times\ZZ/r\ZZ,\quad \textrm{(}p=5 \textrm{ or } 7,~ r=2 \textrm{ or } 6\textrm{).}
\]
In all cases $\rho$ is a congruence representation of level $pr$.
\end{prop}
\begin{proof}
  Let $Z\df Z(\rho(\Gamma))$ and note that $Z$ is cyclic of order dividing $6$. This holds because $U|_Z$ is necessarily reducible. In particular $\rho(\Gamma)\neq Z$, so we may choose a \emph{minimal nontrivial normal subgroup} $M/Z \unlhd\rho(\Gamma)/Z$.

\emph{Case 1: $M$ is solvable}. We will show that this Case cannot occur. Otherwise,  $M/Z\cong (\ZZ/\ell\ZZ)^d$ for some prime $\ell$ and integer $d$. Now $U|_M$ is irreducible, and this forces $\ell = 3$, moreover the Sylow $3$-subgroup of $M$, call it $P$, satisfies $P \unlhd \rho(\Gamma)$. Indeed, $d=2$ and $P$ is an extra-special group $P\cong 3^{1+2}$. Because $P$ acts irreducibly on $U$ its centralizer consists of scalar matrices which therefore lie in $Z$. As a result, it follows that $\rho(\Gamma)/Z$ is isomorphic to a subgroup of the group of automorphisms of $P$ that acts trivially on $Z(P)$. This latter group is $(\ZZ/3\ZZ)^2 \rtimes \SL_2(3)$. Because $\rho(\Gamma)$ has a subgroup of index $2$ (Corollary  \ref{cordet1}) the only possibilities are that $\rho(\Gamma)/PZ$ is isomorphic to subgroup of $\ZZ/12\ZZ$, where we use the fact that the abelianization of $\Gamma$ is \emph{cyclic} to eliminate some possibilities. Indeed, this abelianization is $\ZZ/12\ZZ$, generated by the image of $T$, and furthermore $T^6\Gamma'= S^2\Gamma'$. It follows that in fact $\rho(\Gamma)/PZ$ is isomorphic to subgroup of  $\ZZ/6\ZZ$. But in all such cases, $M=PZ$ is \emph{not} a minimal normal subgroup. This completes the proof in Case 1.

\emph{Case 2: $M$ is nonsolvable}. Here, the only quasisimple groups with a $3$-dimensional faithful projective representation are $L_2(5)$, $L_2(7)$, $3.L_2(9)$, and the latter group is excluded thanks to Lemma \ref{thmprimcong}(b). We deduce that $M\cong L_2(p)\times Z$ with $p= 5$ or $7$. Furthermore $\Aut(L_2(p))= \PGL_2(p)$ does \emph{not} have a $3$-dimensional faithful representation, so $\Gamma=M$. Let $Z \cong Z/r\ZZ$ with $r \mid 6$. Because $\rho(\Gamma)$ has a subgroup of index $2$, then $2 \mid r$, so that  $r=2$ or $6$. 

Finally, use Lemma \ref{thmprimcong}(a) and the fact that $\Gamma'\langle S^2\rangle$ is a congruence subgroup of level $6$ to see that $\ker\rho$ is also a congruence subgroup, of level $pr$. This completes the proof of the Proposition.
\end{proof}

With this result in hand we turn to a description of the possible sets of eigenvalues for $\rho(T)$. Because $T$ generates the abelianization of $\Gamma$ and the level of $\ker\rho$ is $pr$, there is a generator $z$ of $Z \cong \ZZ/r\ZZ$ and an element $x \in L_2(p)$ of order $p$ such that $\rho(T) = xz$. Noting that $L_2(p)$ has a pair of conjugate irreducible representations of dimension $3$, it follows that $\rho$ falls into one of just $12$ equivalence classes and similarly there $12$ possible sets of eigenvalues for $\rho(T)$. Thus if $p=5$ then the  eigenvalues for $\rho(T)$ are of the form $\{\mu, \mu\lambda, \mu\bar{\lambda} \}$ where $\lambda$ and $\mu$ are primitive $5^{\textrm{th}}$ and $r^{\textrm{th}}$ roots of unity, respectively. Similarly, if $p=7$ the eigenvalues for $\rho(T)$ are of the form $\{\mu\lambda, \mu\lambda^2, \mu\lambda^4 \}$ where $\lambda$ and $\mu$ are primitive $7^{\textrm{th}}$ and $r^{\textrm{th}}$ roots of unity, respectively. Hence the possible exponents $\pmod{\ZZ}$ are as follows:
\begin{eqnarray}
  \label{prtriples}
&&(p, r){=}(5, 2).\  \{ \tfrac{1}{2}, \tfrac{3}{10},  \tfrac{7}{10} \},  \{ \tfrac{1}{2}, \tfrac{1}{10}, \tfrac{9}{10}  \}\notag\\
&&(p, r){=}(5, 6).\ \{\tfrac{1}{6}, \tfrac{11}{30}, \tfrac{29}{30}\}, \{ \tfrac{1}{6}, \tfrac{17}{30}, \tfrac{23}{30}\} ,\{\tfrac{5}{6}, \tfrac{1}{30}, \tfrac{19}{30} \}, \{\tfrac{5}{6}, \tfrac{7}{30}, \tfrac{13}{30}\}  \notag       \\
&&(p, r){=}(7, 2).\  \{\tfrac{1}{14} , \tfrac{9}{14}, \tfrac{11}{14} \},  \{ \tfrac{3}{14}, \tfrac{5}{14},  \tfrac{13}{14} \}\\
&&(p, r){=}(7, 6).\    \{\tfrac{13}{42} , \tfrac{19}{42}, \tfrac{31}{42} \},  \{ \tfrac{25}{42}, \tfrac{37}{42},  \tfrac{1}{42} \}, \notag
  \{\tfrac{41}{42} , \tfrac{5}{42}, \tfrac{17}{42} \},  \{ \tfrac{11}{42}, \tfrac{23}{42},  \tfrac{29}{42} \}.
\end{eqnarray}

Finally we summarize these computations in the following:
\begin{prop}
  \label{primdenoms}
  If $\rho$ is a primitive representation then one of the following holds:
  \begin{enumerate}
  \item  If $p=5$, then the pairs of rational numbers $\{h_1, h_2\} \pmod{\ZZ}$ take \emph{all} possible  values $\{\tfrac{u}{5}, \tfrac{v}{5}\}$ with $1{\leq}u{<}v{\leq}4$.
\item If $p=7$, then the pairs of rational numbers $\{h_1, h_2\} \pmod{\ZZ}$ takes each of the $6$ values $\{\tfrac{1}{7}, \tfrac{3}{7}\}, \{\tfrac{1}{7}, \tfrac{5}{7}\}, \{\tfrac{2}{7}, \tfrac{3}{7}\}, \{\tfrac{2}{7}, \tfrac{6}{7}\}, \{\tfrac{4}{7}, \tfrac{5}{7}\},  \{\tfrac{4}{7}, \tfrac{6}{7}\}$
exactly $3$ times, and the other $9$ values are \emph{omitted}.
\end{enumerate}
\end{prop}

\begin{rmk}
\label{r:denoms}
  In what follows, the critical points to observe in Propositions \ref{imprimdenoms} and \ref{primdenoms} are that the \emph{denominators} of $h_1$ and $h_2$ are divisors of $16$ in the imprimitive cases, and they are divisors of $5$ or $7$ in the primitive cases.
\end{rmk}

\section{The elliptic surface}
\label{s:elliptic}

Thanks to the results in Sections \ref{s:voas} and \ref{s:monodromy}, we are now prepared to tackle the arithmetic classification of possible character vectors $F(\tau)$ for strongly regular VOAs $V$ with exactly $3$-simple modules and irreducible monic monodromy. It will then remain to analyze which of the possible character vectors are in fact realized by a VOA.

The next step in our classification specializes equation \eqref{Hypergeosolns} to yield the following formula for the character vector $F(\tau)$ corresponding to a VOA $V$ with simple modules $V$, $M_1$ and $M_2$: we have $F(\tau) = (f_0,f_1,f_2)^T$ where
\begin{align*}
  f_0 &= j^{\frac{2x+2y+3}{6}}{}_3F_2\left(-\tfrac{2x+2y+3}{6},-\tfrac{2x+2y+1}{6},-\tfrac{2x+2y-1}{6};-x,-y;\tfrac{1728}{j}\right),\\
  f_1 &= A_1j^{\frac{2y-4x-3}{6}}{}_3F_2\left(\tfrac{4x-2y+3}{6},\tfrac{4x-2y+5}{6},\tfrac{4x-2y+7}{6};x+1,x-y;\tfrac{1728}{j}\right),\\
  f_2 &= A_2j^{\frac{2x-4y-3}{6}}{}_3F_2\left(\tfrac{4y-2x+3}{6},\tfrac{4y-2x+5}{6},\tfrac{4y-2x+7}{6};y+1,y-x;\tfrac{1728}{j}\right),
\end{align*}
and $c = 8(x+y)+12$, $h_1 = x+1$, $h_2 = y+1$, $A_1 = \dim (M_1)_{h_1}$, $A_2 = \dim(M_2)_{h_2}$.

While Section \ref{s:monodromy} showed that we need only consider certain rational values of $x$ and $y$ whose denominators divide $16$, $5$ or $7$, it is useful to observe that $F(\tau)$ is in fact an algebraic family of vector-valued modular forms varying with the parameters $x$ and $y$, in the sense that the Fourier coefficients of this family are rational functions in $x$ and $y$. If $F(\tau)$ corresponds to a VOA, then the coefficients must in fact be nonnegative integers. Since $A_1$ and $A_2$ are unknown positive integers, in this Section we focus on $f_0$. More precisely, if we write $f_0(q) = q^{-c/24}(1 + mq + O(q^2))$ as in Section \ref{s:voas}, then the hypergeometric expression for $f_0$ above shows that $m$, $x$ and $y$ satisfy an algebraic equation that defines an elliptic surface:
\begin{equation}
  \label{eq1}
0 = (4(x+y)+6)((4(x+y)+2)(4(x+y)-2)-62xy)+mxy.
\end{equation}
As a fibration over the $m$-line, a theorem of Siegel (Theorem 7.3.9 of \cite{BG}) tells us that all of the good fibers of this surface have finitely many rational solutions subject to our restrictions on the monodromy from Section \ref{s:monodromy}. It does not appear to be easy to classify all of the relevant rational solutions directly, and so ultimately our analysis will rely on properties of this elliptic surface, in addition to properties of vector-valued modular forms and generalized hypergeometric series. Nevertheless, we shall describe some facts on the geometry and arithmetic of this surface that were crucial in our initial studies on this classification problem, but which will otherwise not be used in the sequel.

Begin by homogenizing equation \eqref{eq1}: we are interested in the curve $E/\CC(m)$ defined by $F(x,y,z) = 0$ where
\[
  F(x,y,z) = (4(x+y)+6z)((4(x+y)+2z)(4(x+y)-2z)-62xy)+mxyz.
\]
Notice that $E$ meets the line at infinity defined by $z = 0$ in three distinct points:
\begin{align*}
  P_1 & = (1:-1:0), & P_2 &= \left(15+\sqrt{-31}:16:0\right), & P_3 &= \left(15-\sqrt{-31}:16:0\right).
\end{align*}
Taking $P_1\df \infty$ for the identity of the group, the inversion for the group law on $E$ is given by swapping $x$ and $y$. At the level of VOAs this corresponds to interchanging the nontrivial modules $M_1$ and $M_2$ for $V$. The group law of \eqref{eq1} itself has a more complicated expression in terms of $m$ that we will not write down explicitly.

Consider the change of coordinates:
{\footnotesize
\[
  (U:V:W) = (x:y:z)\left(\begin{matrix}
      -24(65m^2-24552m-353648)&-6912m(m-248)(m-496)&248\\
      -24(65m^2-24552m-353648)&6912m(m-248)(m-496)&248\\
      -3(m^3-732m^2+97712m-4243776)&0&372-m
    \end{matrix}
  \right).
\]}
This change of coordinates turns equation \eqref{eq1} into the Weierstrass form $H(U,V,W) = 0$ where
{\footnotesize
\begin{align*}
  &H(U, V, W)= -V^2W+U^3-27(m^3-844m^2+210992m+1049536)(m+124)UW^2\\
  &+54(m^6-1080m^5+353904m^4-78209280m^3+16393117440m^2+465661052928m+1484665229312)W^3.
\end{align*}}
The discriminant of this elliptic curve over $\CC(m)$ is
\[
\Delta = 2^{27}\cdot 3^{13}\cdot (m+4)m^2(m-248)^2(m-496)^2\left(m^2+\frac{123}{3}m+\frac{8464}{3}\right)
\]
and the $j$-invariant is
\[
  j = \frac{(m+124)^3(m^3-844m^2+210992m+1049536)^3}{2^{15}\cdot 3 \cdot m^2(m-248)^2(m-496)^2(m+4)\left(m^2+\frac{128}{3}m+\frac{8464}{3}\right)}.
\]

Setting $y=0$ in equation \eqref{eq1} yields three rational points
\begin{align*}
  Q_1 &= (1/2:0:1),\\
  Q_2 &= (-1/2:0:1),\\
  Q_3 &= (-3/2:0:1),
\end{align*}
such that $Q_1+Q_2+Q_3 = \infty$. One can show that these points have infinite order in the fiber $E_m$ of $E/\CC(m)$ for all rational values of $m$ except when $m = -32$, $-4$, $0$, $\tfrac{633}{3}$, $248$ and $496$. Thus, the rational fibers $E_m$ typically have Mordell-Weil rank at least $2$. This might sound surprising, as the average Mordell-Weil rank of a rational elliptic curve is expected to be $\tfrac 12$. But in fact, families such as \eqref{eq1} with large rank are not so uncommon -- see for example \cite{Elkies} for an interesting discussion of such matters.

We began our study of \eqref{eq1} directly via the fibration over $\CC(m)$. It turns out that fibering over $y$ is more useful for classifying the VOAs under discussion here: indeed, all but finitely many of the infinite number of VOAs identified in Theorem \ref{thmmain} correspond to $y = -1/2$. Nevertheless, we shall record here a result that allows the effective enumeration of solutions $(m,x,y)$ to \eqref{eq1} for fixed rational $m$ and rational $x$ and $y$ with bounded denominator that was crucial in our initial studies of equation \eqref{eq1}.

The idea is to first study the rational points of the quotient surface obtained by modding out \eqref{eq1} by the inverse for the elliptic curve group law. Since inversion is given by swapping $x$ and $y$ in Equation \eqref{eq1}, we are interested in the rational solutions to the equation
\begin{equation*}
  0 = 4(2u+3)(8u^2-2-31v)+mv.
\end{equation*}
Solving for $x$ and $y$ via $x+y = u$ and $xy = v$ yields solutions of \eqref{eq1} defined over a quadratic extension of $\QQ$. It will be convenient to work with the corresponding projectivized equation
\begin{equation}
\label{eq2}
  0 = 4(2u+3w)(8u^2-2w^2-31vw)+mvw^2.
\end{equation}

Equation \eqref{eq2} defines a one-parameter family of singular cubic curves that, generically, are connected (and there are a finite number of fibers equal to a conic times a line). The rational points in the smooth locus of a connected rational singular cubic can be parameterized by linear projection from a rational singularity. The point $P = (0:1:0)$ is a rational singular point of every fiber, and this is the point that we will project from. The general line meeting $P$ is given by the equation
\[
  au+bw = 0.
\]
First suppose that $b = 0$. This means we wish to describe the solutions to \eqref{eq2} with $u=0$. These are the point $P$, along with the points
\[
 \left(0:\frac{24}{m-372}:1\right)
\]
with $m \neq 372$.

Henceforth we may assume that $b$ and $u$ are nonzero. After reparameterizing our line, we may assume $w = au$. Substituting this into equation \eqref{eq2} and using $u \neq 0$ yields
\begin{equation*}
  a(-ma+372a+248)v = -8(a-2)(a+2)(3a+2)u. 
\end{equation*}
If $a = 0$ then this equation forces $u = 0$, and we have already classified such points. We are thus now free to assume $a \neq 0$. If $-ma+372a+248 = 0$ then we must have $a = 2, -2$ or $a = -2/3$. This implies that away from the fibers for $m  = 0$, $248$ and $496$, we may assume $-ma+372a+248 \neq 0$. Therefore, away from these values of $m$ we can solve for $v$ above to obtain the family of points
\[
  \left(u:\frac{8(2u-1)(2u+1)(2u+3)}{(372-m+248u)}:1\right).
\]
Notice that if we set $u = 0$ we recover the preceding family of points.

It remains to consider whether the fibers have other rational singularities besides $P$ (as those points can't be accessed via projection), and to consider the fibers above $m = 0, 248$ and $496$.

First we treat the singularities. The $v$-partial derivative of \eqref{eq2} yields
\[
  w(-wm+248u+372w) = 0.
\]
Thus, singular solutions in a fiber of \eqref{eq2} must satisfy either $w = 0$ or $u = \frac{m-372}{248}w$. When $w = 0$ we find, by consideration of the $v$-partial, that the only possible additional rational singularity is $\left(1:\frac{12}{31}:0\right)$. The $u$-partial does not vanish at this point, and hence this is not in fact a singularity of the fibers. The other case is when $w \neq 0$ and
\[
  Q = \left(\frac{m-372}{248}:v:1\right).
\]
Substituting this into \eqref{eq2} yields $m = 0, 248$ or $496$. When $m = 0$ we obtain the unique additional singularity $(-3/2:16/31:1)$, when $m = 248$ we obtain the unique additional singularity $(-1/2:-8/31:1)$, and when $m = 496$ we obtain the unique additional singularity $(1/2:16/31:1)$. These are all the missing singularities, and all the missing points on the fibers corresponding to $m = 0, 248$ and $496$. Thus, we have described all rational solutions to \eqref{eq2}. We have nearly proven the following:
\begin{prop}
  \label{p:ratpoints}
  Suppose that $(m,x,y)$ is a rational solution to \eqref{eq1}. Then if $u =x+y$ and $v = xy$, the rational point $(u,v,m)$ is equal to 
  \[\left(u,\frac{8(2u-1)(2u+1)(2u+3)}{(372-m+248u)},m\right)\]
and $u \neq \frac{m-372}{248}$.
\end{prop}
\begin{proof}
We have seen that the only other possible rational solutions $(m,x,y)$ correspond to $(u,v,m)$ equal to one of the singular points $(-3/2,16/31,0)$, $(-1/2,-8/31,248)$ or $(1/2,16/31,496)$. But none of these correspond to \emph{rational} values of $x$ and $y$.
\end{proof}

\begin{thm}
  \label{t:linearbound}
Let $N > 0$ be an integer and let $m$ be a rational number. Then the number of solutions $a_m(N)$ to equation \eqref{eq1} with rational $x$, $y$ of denominator dividing $N$ satisfies
\[
  a_m(N) \leq 2+N\max\left(\frac{16\abs{m-372}}{31},6148\right).
\]
\end{thm}
\begin{proof}
  Let $(m,x,y)$ be a rational solution to equation \eqref{eq1}, and let $(u,v,m)$ be the corresponding solution to \eqref{eq2} with $u = x+y$, $v = xy$. Then $(u,v,m)$ is equal to one of the points in Proposition \ref{p:ratpoints}. Since the polynomial $T^2-uT+v$ has rational roots by hypothesis, it follows that the discriminant
\[
u^2-4v = u^2-\frac{32(2u-1)(2u+1)(2u+3)}{(372-m+248u)}
\]
must be a rational square. In particular,
\[
  1 \geq \frac{32}{31}\frac{(1-\frac 1{(2u)^2})(1+\frac{3}{2u})}{(1+\frac{372-m}{248u})}.
\]
As $\abs{u}$ grows, the right hand side converges to $32/31$, so that in fact, there are only finitely many solutions in each fiber. We knew this already by a result of Siegel, but we can now use the parameterization to obtain precise bounds.

  First assume that $\abs{(372-m)/248u} < 1/A$ for some big constant $A$ that we will specify later. Then for $A > 31$ we find
\[
1 > \frac{31(A+1)}{32A} \geq \left(1-\frac 1{4u^2}\right)\left(1+\frac{3}{2u}\right),
\]
and this will produce contradictions for large $\abs{u}$. Choose numbers $e_1,e_2 \in (0,1)$ with $e_1+e_2 = 1$. We will find explicit bounds on $u$ that ensure
\begin{align*}
(1-(2u)^{-2}) &> (31(A+1)/32A)^{e_1},\\
(1+3/(2u)) &> (31(A+1)/32A)^{e_2}.
\end{align*}
The first bound is equivalent with
\[
1 - \left(\frac{31(A+1)}{32A}\right)^{e_1} > \frac{1}{(2u)^2}
\]
which is equivalent with
\[
\abs{u} > \frac{1}{2}\left(1 - \left(\frac{31(A+1)}{32A}\right)^{e_1}\right)^{-1/2}
\]
The second bound is equivalent with
\[
1-\left(\frac{31(A+1)}{32A}\right)^{e_2} > -\frac{3}{2u}
\]
This is always true if $u > 0$ by choice of $A$ and $e_2$, since the left side is positive, so that the second bound will hold whenever
\[
\abs{u} > \frac{3}{2}\left(1 - \left(\frac{31(A+1)}{32A}\right)^{e_2}\right)^{-1}
\]
Thus, if $\abs{u}$ is bigger than the max of these, we have a contradiction. Therefore, we must have
\[
\abs{u} \leq \max\left(\frac{A\abs{m-372}}{248},\frac{1}{2}\left(1 - \left(\frac{31(A+1)}{32A}\right)^{e_1}\right)^{-1/2}, \frac{3}{2}\left(1 - \left(\frac{31(A+1)}{32A}\right)^{e_2}\right)^{-1} \right).
\]

Now to optimize parameters. First off, our choice of $A$ must ensure that $1 > \frac{31(A+1)}{32A}$, and we'd like it to be as small as possible. A natural choice is $A = 32$, but any $A$ satisfying $31 < A \leq 32$ would work. To be definite take $A = 32$, so that
\[
\abs{u} \leq \max\left(\frac{4\abs{m-372}}{31},\frac{1}{2}\left(1 - \left(\frac{1023}{1024}\right)^{e_1}\right)^{-1/2}, \frac{3}{2}\left(1 - \left(\frac{1023}{1024}\right)^{e_2}\right)^{-1} \right)
\]
Next we would like to optimize the choice of $e_1$ and $e_2$ so that this maximum is minimized. Computations show that the minimum of the last two values above is achieved for $e_1$ somewhere between $1/5000$ and $1/10000$. For example, using $e_1 = 1/5000$ we obtain
\[
\abs{u} \leq \max\left(\frac{4\abs{m-372}}{31},1537\right).
\]

We are only interested in the values of $u$ of the form $u = i/N$ in this range, and there are at most $2N\max\left(\frac{4\abs{m-372}}{31},1537\right)+1$ of these. For each such choice, we have at most two rational solutions $(m,x,y)$ and $(m,y,x)$ to equation \eqref{eq1}. This concludes the proof.
\end{proof}

\begin{rmk}
  In the proof above, many values of $u$ correspond to points for which the discriminant
\[
u^2-\frac{32(2u-1)(2u+1)(2u+3)}{(372-m+248t)} \geq 0
\]
is not a rational square. In such cases the corresponding pair of points $(m,x,y)$ and $(m,y,x)$ satisfying equation \eqref{eq1} have $x$ and $y$ values contained in a real quadratic extension of $\QQ$. Thus, it seems possible that the linear bound on $a_m(N)$ above could be improved by making stronger use of the discriminant condition.
\end{rmk}

\begin{rmk}
  For fixed values of $m$, the preceding proof yields an explicit and efficient algorithm for enumerating all rational solutions to equation \eqref{eq1} satisfying the divisibility conditions of Theorem \ref{t:linearbound}. The steps are as follows:
  \begin{enumerate}
\item Fix a rational value of $m$.
\item List the finite number of values $u = i/N$ satisfying the inequality
  \[
\abs{u} \leq \max\left(\frac{4\abs{m-372}}{31},1537\right).
\]
\item For each value of $u$ from the previous step, test whether the discriminant
  \[
  D(u,m) = u^2-\frac{32(2u-1)(2u+1)(2u+3)}{(372-m+248u)}
\]
is a rational square.
\item If $D(u,m)$ is a rational square, then set $x = (u+\sqrt{D})/2$ and $y = (u-\sqrt{D})/2$. This contributes solutions $(m,x,y)$ and $(m,y,x)$ to equation \eqref{eq1} (note that it's possible to have $x = y$).
  \end{enumerate}
  We have run this algorithm for $m = 0$ through $m = 20,000$, and one finds that it is most common to have $a_m(16) = 8$ and $a_m(5) = a_m(7) = 0$ in that range. Note that $a_m(16) \geq 8$ for all $m$ due to the existence of the points $\pm Q_1$, $\pm Q_2, \pm Q_3$ on the elliptic curve over $\CC(m)$ defined by \eqref{eq1}, as well as the points $\pm Q_4$ where
  \[
  Q_4 = Q_1 - Q_2 =  \left(-\frac{m}{16}-1:-\frac{m}{16}-\frac{1}{2}:1\right).
\]
Notice that the existence of this family of points shows that the bound on $\abs{u}$ used in the proof of Theorem \ref{t:linearbound} is essentially optimal, since this family of points corresponds to $u = -\frac{1}{8}m-\frac{3}{2}$.
\end{rmk}

In general, for each $m$ Equation \eqref{eq1} has many rational solutions that do not correspond to VOAs. To aid us in eliminating many of these solutions we shall next analyze all three coordinates of the corresponding (in general hypothetical) characters corresponding to a solution of Equation \eqref{eq1}.

\section{Positivity restrictions}
\label{s:positivity}
Let $(m,x,y)$ denote a solution to Equation \eqref{eq1} that corresponds to a VOA as in Theorem \ref{thmmain}, and let $F(\tau)$ be the corresponding character vector. In this Section we exploit the fact that the Fourier coefficients of $F(\tau)$ must be nonnegative. Since these coefficients are reducible rational functions, we can gain some traction by studying their divisors, as the sign of the coefficient is constant in the connected components of the complement of the divisor.
\begin{thm}
  \label{t:positivity}
 If $(m,x,y)$ denotes a solution to Equation \eqref{eq1} realized by a VOA satisfying the restrictions of Theorem \ref{thmmain}, and if $\abs{x+1}> 5/2$ or $\abs{y+1} >5/2$, then exactly one of the following holds:
\begin{enumerate}
\item $\abs{x-y} \leq 1$;
\item $-2 \leq y \leq 0$;
\item $-2 \leq x \leq 0$.  
\end{enumerate}
\end{thm}
\begin{proof}
  Begin by writing
  \[
\left(\begin{smallmatrix}
    f_0\\
    f_1\\
    f_2\\
  \end{smallmatrix}\right)=\diag\left(q^{-\frac{2x+2y+3}{6}},A_1q^{-\frac{2y-4x-3}{6}},A_2q^{-\frac{2x-4y-3}{6}}\right)\left(\begin{smallmatrix}
    1 +mq + O(q^2)\\
    1 + F_1q + O(q^2)\\
    1 + F_2q + O(q^2)
  \end{smallmatrix}
\right)
  \]
Equation \eqref{eq1} gives an explicit formula for $m$ in terms of $x$ and $y$. From the expressions for $f_1$ and $f_2$ in terms of generalized hypergeometric series, one finds similarly that
  \begin{align*}
    F_1(x,y) &=  \frac{4(2y-4x-3)(x^2-xy+8y^2+3x+14y+8)}{(x+2)(y-x-1)}
  \end{align*}
and $F_2(x,y) = F_1(y,x)$. Observe that the divisors of $F_1$ dissect the plane into a finite number of regions, and the sign of $F_1$ is constant in each region. Figure \ref{f:divisor} shows the divisors of each of $m$, $F_1$ and $F_2$.
\begin{figure}[H]
  \centering
  \subfloat[$m$]{\includegraphics[scale=0.4]{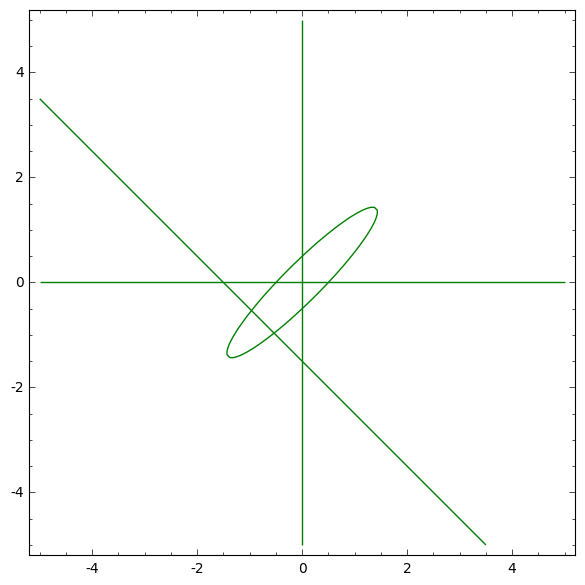}}
  \qquad
  \subfloat[$F_1$]{\includegraphics[scale=0.4]{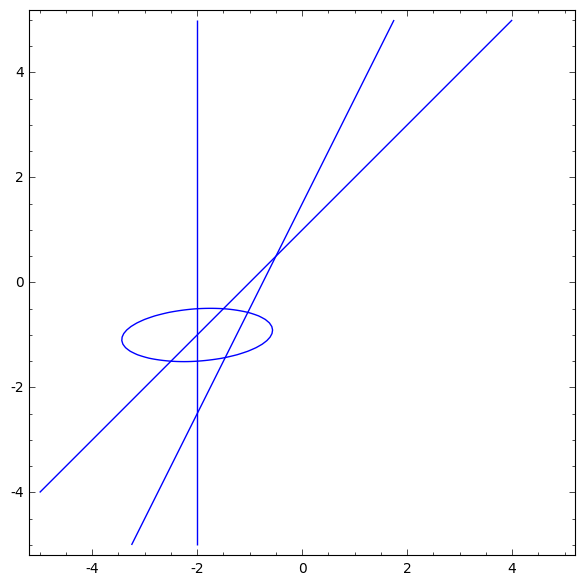}}
  \qquad
  \subfloat[$F_2$]{\includegraphics[scale=0.4]{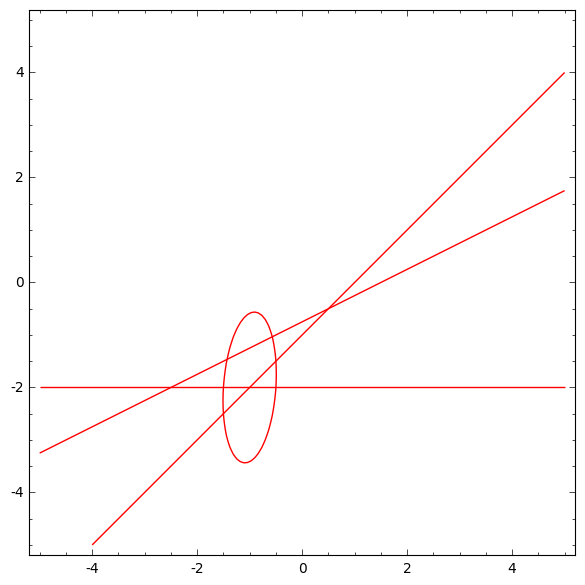}}
  \caption{The divisors of $m$, $F_1$ and $F_2$.}
  \label{f:divisor}
\end{figure}
Figure \ref{f:positiveregion} plots all three divisors. Outside of the boxed region enclosed by the dashed lines, the only regions where $m$, $F_1$ and $F_2$ are simultaneously positive are the shaded regions in Figure \ref{f:positiveregion}, and these regions correspond to the statement of the Theorem.
\begin{figure}
    \centering
  \includegraphics[scale=0.60]{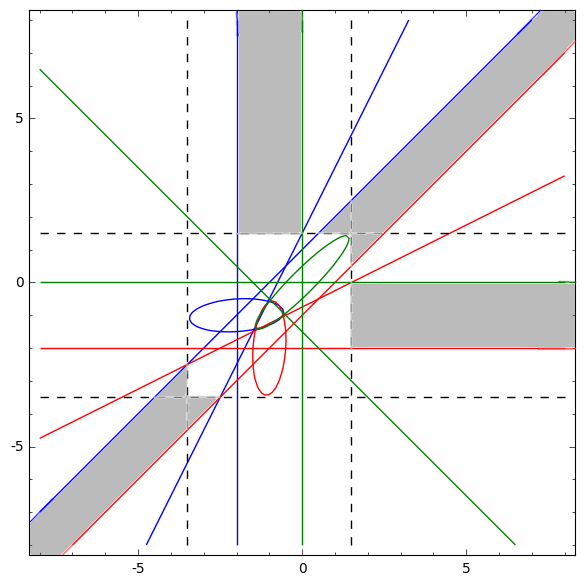}
    \caption{Regions corresponding to VOAs.}
    \label{f:positiveregion}
  \end{figure}
\end{proof}

\begin{rmk}
  Since $h_1 = x+1$ and $h_2 = y+1$, the following condition holds for a strongly regular VOA with exactly $3$ simple modules and whose character vector satisfies a monic MLDE of degree $3$ with irreducible monodromy: if $\abs{h_1} > 5/2$ or $\abs{h_2} > 5/2$, then one of the following holds:
  \begin{enumerate}
  \item $\abs{h_1-h_2} \leq 1$,
  \item $\abs{h_1} \leq 1$ or
  \item $\abs{h_2} \leq 1$.
  \end{enumerate}
This is a relatively simple consequence of the fact that the Fourier coefficients of $F(\tau)$ are rational functions of $h_1$ and $h_2$.
\end{rmk}

The region bounded by the dashed lines in Figure \ref{f:positiveregion} contains a finite number of points $(x,y)$ where $x$ and $y$ are rational numbers satisfying the restrictions of Section \ref{s:monodromy} (recall that that \ref{s:monodromy} showed that $x$ and $y$ are necessarily rational numbers with denominators that divide $5$, $7$ or $16$). It is thus a simple matter to enumerate them. Therefore, by symmetry we may now focus our attention on the shaded regions in Figure \ref{f:positiveregion} on page \pageref{f:positiveregion} below the diagonal $x = y$. The shaded regions contain a finite number of horizontal and diagonal slices of the elliptic surface defined by Equation \eqref{eq1} of relevance to our classification. These slices turn out to be singular cubic curves whose rational points are parameterized and studied in Section \ref{s:fibers} below. In the next section we exploit this geometry and the hypergeometric nature of $F$ to find all values of $x$ and $y$ where $f_0$ has positive integer coefficients, and where $f_1$ and $f_2$ have positive coefficients.
\begin{rmk}
  Due to the unknown scalars $A_j = \dim (M_j)_{h_j}$ for $j=1,2$, we cannot yet make use of the fact that $f_1$ and $f_2$ have integer coefficients.
\end{rmk}

\section{The remaining fibers}
\label{s:fibers}
\subsection{The horizontal fibers}
In this Section we regard Equation \eqref{eq1} as a fibration over $y$. In order to homogenize the equation, let $m$ be of degree $1$ and let $y$ be of degree $0$. Then the homogenized version of Equation \eqref{eq1} is
\begin{equation}
  \label{yfiber}
  0 = (4(x+yz)+6z)((4(x+yz)+2z)(4(x+yz)-2z)-62xyz)+mxyz
\end{equation}
and there is a (unique) singular point $(m:x:z) = (1:0:0)$ at infinity in every fiber. Therefore, the smooth locus of each fiber can be rationally parameterized by projection from $(1:0:0)$.

Before proceeding to this we shall classify all additional singular points in the affine patches with $z \neq 0$, as such points cannot be obtained by projection from $(1:0:0)$. First off, the vanishing of the $m$-partial of Equation \eqref{yfiber} implies that either $x = 0$ or $z = 0$ at a singularity. The vanishing of the partials at points with $x = 0$ corresponds to the polynomial equations
\begin{align*}
  0 &= z (56 y^{2} z -  m y + 180 y z + 16 z),\\
  0 &= (2 y - 1) (2 y + 1) (2 y + 3) z^{2}.
\end{align*}
It follows that if $y \neq \pm 1/2$ and $-3/2$, then the only singular point in the fiber is the point $(1:0:0)$ at infinity. Thus, the entire fiber of Equation \eqref{yfiber} can be described by projection from infinity, as long as $y \neq \pm 1/2$ and $-3/2$. In the exceptional fibers we find the following additional singular points, corresponding to a conic intersecting a line in two points:
\begin{align*}
  y=1/2: &\quad (240:0:1),\\
  y=-1/2: &\quad (120:0:1),\\
  y=-3/2: &\quad (256/3:0:1).
\end{align*}
Note that $y=1/2$ is outside of the shaded region, so there are in fact only two exceptional fibers that we must consider.

Thus, we now suppose that $-2 \leq y \leq 0$ with $y \neq -1/2, -3/2$, and we will treat these two exceptional fibers separately afterward. In order to rule out the existence of a VOA corresponding to all but (an explicitly computable) finite number of such solutions to Equation \eqref{eq1}, we will use the fact that the character of the hypothetical VOA
\[
  f_0 = j^{\frac{2x+2y+3}{6}}{}_3F_2\left(-\frac{2x+2y+3}{6},-\frac{2x+2y+1}{6},-\frac{2x+2y-1}{6};-x,-y;\frac{1728}{j}\right)
\]
must have nonnegative integers as coefficients.

Let $B_k$ denote the $k$th coefficient of the underlying hypergeometric series (without the $j$-factors taken into account) defining $f_0$. If we can show that some hypergeometric coefficient $B_k$ has a prime divisor in its denominator that does not divide the denominators of $x/6$ and $y/6$, then it will also appear in the denominator of the $k$th coefficient of $f_0$. Notice that since we are only interested in solutions $(x,y)$ to Equation \eqref{eq1} with denominators equal to $5$, $7$ or a divisor of $16$, by Section \ref{s:monodromy}, this means that only primes $p \leq 96$ could possibly divide some denominator of a coefficient $B_k$ but not divide any denominators in $f_0$. Thus, below we restrict to primes $p \geq 96$ and consider only the coefficients $B_k$, rather than the more complicated coefficients of $f_0$.

Recall from \cite{FGM} Theorem 3.4 that if $c_p(x,y)$ denotes the number of $p$-adic carries required to compute the $p$-adic addition $x+y$, and if $v_p$ denotes the $p$-adic valuation normalized so that $v_p(p) = 1$, then
\begin{align}
  \label{coeffval}
  v_p(B_k) = &c_p\left(-\tfrac{1}{3}(x+y)-\tfrac 32,k\right) + c_p\left(-\tfrac{1}{3}(x+y)-\tfrac 76,k\right)+c_p\left(-\tfrac{1}{3}(x+y)-\tfrac 56,k\right)\\
  &- c_p(-x-1,k)-c_p(-y-1,k).\notag
\end{align}
The key point here is that if there exists a prime $p \geq 96$ such that the zeroth $p$-adic digit of $-y-1$ is largest among the $5$ arguments above, say $-y-1 \equiv y_0 \pmod{p}$, then in \eqref{coeffval}, $c_p(-y-1,p-y_0) \geq 1$, while each other term $c_p(*,p-y_0)$ will be zero. Therefore, for such primes we have $v_p(A_{p-y_0}) \leq -1$ and hence $f_0$ is \emph{not} integral. The arithmetic difficulty that arises in our argument for the exceptional cases when $y = -1/2, -3/2$ is that $-y-1$ has zeroth $p$-adic digit asymptotic to $p/2$ for \emph{all} odd primes. Hence we shall treat those cases separately.

Suppose first that the denominators of $x$ and $y$ are both equal to $5$. We shall give all the details in this case and omit the details for the cases of the other possible denominators, as the arguments are identical save for adjusted constants. The exceptions are the fibers $y=-1/2$ and $y=-3/2$, which we shall also treat in detail. Note that since we are interested in irreducible monodromy representations, we may assume that $5x$ and $5y$ are both integral and relatively prime to $5$, and also $5x \not \equiv 5y \pmod{5}$. The key result in this case is the following:
\begin{prop}
  \label{p:5}
  Let $(m,x,y)$ be a solution to equation \eqref{eq1} with $\abs{y+1} < 1$, such that $5x$ and $5y$ are integers coprime to $5$, and such that $5x\not \equiv 5y\pmod{5}$. Then if $x > 18188$, the series $f_0$ does not have integral Fourier coefficients.
\end{prop}
\begin{proof}
There is a unique nonzero congruence class $p_0 \pmod{30}$ such that for all primes $p\equiv p_0\pmod{30}$ big enough (e.g. $p > 96$ suffices), the zeroth $p$-adic digit of $-y-1$ is of the form $\frac{4p+A}{5}$ where $4p+A\equiv 0\pmod{5}$, and the zeroth $p$-adic digit of $-1/3$ is $(p-1)/3$ (this second condition just forces $p_0\equiv 1\pmod{3}$). Note that $A$ depends on $y$, but there are finitely many choices for $y$, so it's bounded absolutely. For example, the following table lists the zeroth $p$-adic digits of some relevant quantities when $y = -1/5$:
\[
  \renewcommand{\arraystretch}{1.5}
  \begin{array}{|c|c|c|c|c|c|c|c|c|}\hline
    y=-\frac 15&p\equiv 1& p\equiv 7& p\equiv 11 & p\equiv 13 & p\equiv 17 & p\equiv 19 & p\equiv 23 & p\equiv 29\\
    \hline
    -\frac{y}{3}-\frac{3}{2}&\frac{13p-43}{30}&\frac{19p-43}{30}&\frac{23p-43}{30}&\frac{p-43}{30}&\frac{29p-43}{30}&\frac{7p-43}{30}&\frac{11p-43}{30}&\frac{17p-43}{30}\\
    -\frac{y}{3}-\frac{7}{6}&\frac{3p-33}{30}&\frac{9p-33}{30}&\frac{3p-33}{30}&\frac{21p-33}{30}&\frac{9p-33}{30}&\frac{27p-33}{30}&\frac{21p-33}{30}&\frac{27p-33}{30}\\
    -\frac{y}{3}-\frac{5}{6}&\frac{23p-23}{30}&\frac{29p-23}{30}&\frac{13p-23}{30}&\frac{11p-23}{30}&\frac{19p-23}{30}&\frac{17p-23}{30}&\frac{p-23}{30}&\frac{7p-23}{30}\\
    \hline
    -y-1&\frac{24p-24}{30}&\frac{12p-24}{30}&\frac{24p-24}{30}&\frac{18p-24}{30}&\frac{12p-24}{30}&\frac{6p-24}{30}&\frac{18p-24}{30}&\frac{6p-24}{30}\\
    -\frac 13 & \frac{p-1}{3} & \frac{p-1}{3} & \frac{2p-1}{3} & \frac{p-1}{3} & \frac{2p-1}{3} & \frac{p-1}{3} & \frac{2p-1}{3} & \frac{2p-1}{3}\\
    \hline
  \end{array}
\]
A similar table exists for each choice of $y$, and the important feature is that there is always a unique column where $-y-1$ has zeroth digit asymptotic to $4p/5$, and $-1/3$ has zeroth $p$-adic digit $(p-1)/3$. When $y=-1/5$ this is the column $p\equiv 1 \pmod{30}$, but in general it is some class mod $30$ such that $p\equiv 1 \pmod{3}$.

So far we have ignored the occurences of $x$ in the formula \eqref{coeffval} for $v_p(B_k)$. We incorporate this information next. Taking account of $x$ has the effect of shifting the digits in first three rows of the table above by a uniform amount (the zeroth $p$-adic digit of $-x/3$) modulo $p$. The key is to find primes $p$ such that this shift does not make one of the entries in the first three rows larger than the zeroth $p$-adic digit of $-y-1$. Therefore, given $x$, it will suffice to prove that there exists a prime $p\equiv p_0\pmod{30}$ satisfying $p > 96$ and
\begin{equation}
\label{eq:xineq}
  0 < \left[\frac{(p-1)x}{3}\right]_p < \frac{p}{30},
\end{equation}
(where $[\alpha]_p$ denotes the least nonnegative residue of an integer $\alpha$ mod $p$). This is due to the fact that $\left[(p-1)x/3\right]_p$ is the zeroth $p$-adic digit of $-x/3$, which is the amount that we are shifting $p$-adic digits by. 

Observe that if we write $x = x_0/5$ then
\[
\left[\frac{(p-1)x}{3}\right]_p = \begin{cases}
\left\{\frac{x_0}{15}\right\}p - \frac{x_0}{15}  & p\equiv 1 \pmod{30},~ p > \frac{x_0}{[x_0]_{15}},\\
\left\{\frac{x_0-3[x_0]_{5}}{15}\right\}p - \frac{x_0}{15}&p\equiv 7\pmod{30},~ p > \frac{x_0}{[x_0-3[x_0]_5]_{15}},\\
\left\{\frac{x_0-9[x_0]_{5}}{15}\right\}p - \frac{x_0}{15}&p\equiv 13\pmod{30},~p > \frac{x_0}{[x_0-9[x_0]_5]_{15}},\\
\left\{\frac{x_0-12[x_0]_{5}}{15}\right\}p - \frac{x_0}{15}&p\equiv 19\pmod{30},~p > \frac{x_0}{[x_0-12[x_0]_5]_{15}},.
\end{cases}
\]
In each case there is an integer $A$ (in fact $A = 1,3,9$ or $12$) such that we win if there exists a prime $p\equiv p_0 \pmod{30}$ with
\[
  \frac{x_0}{[x_0-A[x_0]_5]_{15}} < p
\]
and
\[
\left\{\frac{x_0-A[x_0]_{5}}{15}\right\}p - \frac{x_0}{15} < \frac{p}{30}
\]
These two inequalities are equivalent with
\[
\frac{x_0}{[x_0-A[x_0]_5]_{15}} < p < \frac{x_0}{[x_0-A[x_0]_5]_{15}-\frac 12}.
\]
If we set $X = \frac{x_0}{[x_0-A[x_0]_5]_{15}}$ then this is equivalent with
\[
  X < p < \left(\frac{[x_0-A[x_0]_5]_{15}}{[x_0-A[x_0]_5]_{15}-\frac 12}\right)X
\]
In all cases, the complicated scalar factor in the rightmost inequality above is minimized as $28/27$. Therefore, if we can show that for $X>N$ for an explicit $N$, there is always a prime $p \equiv p_0 \pmod{30}$ that satisfies $X < p < (28/27)X$, then we will be done by the discussion following equation \eqref{coeffval}.

It is a standard argument from analytic number theory that such generalizations of Bertrand's postulate (incorporating more general scalar factors, and restricting to congruence classes of primes) can be proven if one has a sufficiently good understanding of zeros of Dirichlet $L$-functions. For an explicit discussion involving effective results, see Appendix \ref{a:bertrand}. In particular, Theorem \ref{t:bertrand5} of Appendix \ref{a:bertrand} implies that $f_0$ will not be integral as long as $X > 6496$. Therefore, $f_0$ is not integral if $x_0 > 14\cdot 6496$. Since $x = x_0/5$, the Proposition follows.
\end{proof}

Proposition \ref{p:5} allows the classification of all solutions to Equation \eqref{eq1} with $\abs{y+1} < 1$ and $y$ of the form $y = y_0/5$ such that the corresponding function $f_0$ has positive integral Fourier coefficients, and such that the first two Fourier coefficients of $f_1$ and $f_2$ are nonnegative. We computed the first thousand Fourier coefficients of $f_0$, $f_1$ and $f_2$ for all solutions to Equation \eqref{eq1} as in Proposition \ref{p:5}, but with $x \leq 18188$, and tabulated which have the property that
\begin{enumerate}
\item the first thousand coefficients of $f_0$ are nonnegative integers;
\item the first thousand coefficients of $f_1$ and $f_2$ are nonnegative.
\end{enumerate}
Using only the first thousand coefficients already cut the number of possibilities for $f_0$ down dramatically. The results of this computation are in Figure \ref{f:full57}.

A similar argument works for all other $y$-fibers with $\abs{y+1} < 1$ of interest to us, save for those with $y = -1/2$ and $y = -3/2$. As mentioned above, the issue in these two cases is that the $p$-adic expansion of $-y-1$ has a zeroth coefficient asymptotic to $p/2$, so it is harder to use the technique described above to find primes such that its zeroth digit is the largest among the five hypergeometric parameters appearing in Equation \eqref{coeffval}. Thus, we treat these two cases next.

Upon specialization to these two values of $y$, Equation \eqref{eq1} factors as:
\begin{align*}
  y = -1/2: &\quad  x(128x^2+248x-m+120) = 0,\\
  y = -3/2: &\quad  x(128x^2+360x-3m+256) = 0.
\end{align*}
Therefore, among the horizontal fibers, it remains to consider solutions $(m,x,y)$ to Equation \eqref{eq1} of the form 
\begin{align*}
\left(m,0,-\tfrac 12\right), && \left(m,0,-\tfrac 32\right), && \left(\tfrac{n^2-64}{512},\tfrac{-248\pm n}{256},-\tfrac 12\right), && \left(\tfrac{n^2+1472}{1536},\tfrac{-360\pm n}{256},-\tfrac 32\right) 
\end{align*}
The first two sections of Equation \eqref{eq1} with $x = 0$ correspond to reducible monodromy representations, since $x$ is an integer, and so we can ignore them for the present classification of VOAs with \emph{irreducible} monodromy. Thus, since it remains to consider solutions to \eqref{eq1} in the horizontal region with $x > 3/2$, the other points having already been tabulated, it remains in this region to consider the two families of solutions:
\begin{align*}
  \left(\tfrac{n^2-64}{512},\tfrac{n-248}{256},-\tfrac 12\right) &\quad n> 632,&
  \left(\tfrac{n^2+1472}{1536},\tfrac{n-360}{256},-\tfrac 32\right) &\quad n > 744.
\end{align*}
Any points above corresponding to a finite monodromy representation as classified in Section \ref{s:monodromy} will necessarily correspond to imprimitive representations. In order to be irreducible, the $x$ values cannot be in $(1/2)\ZZ$, and thus by Section \ref{s:monodromy} they must have denominator equal to $4$, $8$ or $16$ when expressed in lowest terms. Hence in the first case we are only interested in values of $n$ such that $\frac{n-248}{256} = \frac{\alpha}{16}$ for an integer $\alpha\not\equiv 0\pmod{8}$, while in the second we are only interested in values of $n$ such that $\frac{n-360}{256} =\frac{\beta}{16}$ for $\beta \not \equiv 0 \pmod{8}$. Thus, taking this integrality condition into consideration, we need only consider solutions of the form
\begin{align*}
  \left(\tfrac{(\alpha+15)(\alpha+16)}{2},\tfrac{\alpha}{16},-\tfrac 12\right), &&  \left(\tfrac{\beta^2+45\beta+512}{6},\tfrac{\beta}{16},-\tfrac 32\right),
\end{align*}
where $\alpha,\beta > 24$ are integers such that $\alpha,\beta \not \equiv 0 \pmod{8}$. Notice that $m = \tfrac{(\alpha+15)(\alpha+16)}{2}$ is always a positive integer for positive integral values of $\alpha$. On the other hand, the ratio $\tfrac{\beta^2+45\beta+512}{6}$ is only a positive integer if additionally $\beta \not \equiv 0\pmod{3}$. We shall show in Section \ref{s:quadraticfamily} below that the first family of points in terms of $\alpha$ do in fact correspond to known VOAs -- all but finitely many of the examples in Theorem \ref{thmmain} correspond to points in this family! In the remainder of this section we show that the family of points defined in terms of $\beta$ does \emph{not} correspond to any VOAs (save for some small values of $\beta$).

Consider now the values $(m,x,y) = \left(\frac{\beta^2+45\beta+512}{6},\frac{\beta}{16},-\frac 32\right)$ where $\beta > 24$ is not divisible by $3$ and it is not divisible by $8$. In this case we have
\begin{align*}
  v_p(B_k) = &c_p\left(-\tfrac{\beta+48}{48},k\right) + c_p\left(-\tfrac{\beta+32}{48},k\right)+c_p\left(-\tfrac{\beta+16}{48},k\right)\\
  &- c_p\left(-\tfrac{\beta+16}{16},k\right)-c_p\left(\tfrac 12,k\right).
\end{align*}
Let $p>3$ be a prime divisor of $\beta+24$. The parameters above are congruent to the following quantities mod $p^2$:
\[
  \renewcommand{\arraystretch}{1.5}
  \begin{array}{|c|c|c|}\hline
    y=-3/2&p\equiv 1 \pmod{3} & p\equiv 2 \pmod{3}\\
    \hline
    -1-\frac{\beta}{48}&\left(\frac{p-1}{2}\right)+\left(\frac{p-1}{2}\right)p&\left(\frac{p-1}{2}\right)+\left(\frac{p-1}{2}\right)p\\
    -\frac 23-\frac{\beta}{48}&\left(\frac{p-1}{6}\right)+\left(\frac{p-1}{6}\right)p&\left(\frac{5p-1}{6}\right)+\left(\frac{p-5}{6}\right)p\\
    -\frac 13 -\frac{\beta}{48}&\left(\frac{5p+1}{6}\right)+\left(\frac{5p-5}{6}\right)p&\left(\frac{p+1}{6}\right)+\left(\frac{5p-1}{6}\right)p\\
    \hline
    -1-\frac{3\beta}{48}&\left(\frac{p+1}{2}\right)+\left(\frac{p-1}{2}\right)p&\left(\frac{p+1}{2}\right)+\left(\frac{p-1}{2}\right)p\\
    \frac 12&\left(\frac{p+1}{2}\right)+\left(\frac{p-1}{2}\right)p&\left(\frac{p+1}{2}\right)+\left(\frac{p-1}{2}\right)p\\
    \hline
  \end{array}
\]
Therefore, if $p > 3$ is a prime divisor of $\beta+24$ we find that $v_p(B_{(p-1)/2}) = -1$. Notice that since $\beta$ is coprime to $3$, $\beta+24$ is likewise coprime to $3$. Therefore, $\beta+24$ can only fail to have an odd prime divisor $p > 3$ if $\beta+24 = 2^u$ for some $u \geq 0$. If $u \geq  3$ then this violates that $8$ does not divide $\beta$. We thus see that thanks to our hypotheses, there is always a prime $p > 3$ that divides $\beta+24$.

It now remains to verify that, for such a prime $p$, the factor of $p$ in the denominator of $B_{(p-1)/2}$ is not canceled upon multiplying the hypergeometric factor by the power $j^{\beta/48}$ and substituting $1728/j$ for the argument of $_3F_2$, as in the definition of $f_0$. This is a straightforward computation using the $q$-expansions for $1728/j$ and $j^{\beta/48}$, where the latter $q$-expansion is computed via the binomial theorem. Therefore, this family of points does not contribute any series $f_0$ with nonnegative integer coefficients for parameters, and hence there is no corresponding VOA for any of these choice of parameters.

In this way one can parameterize all possible rational solutions to \eqref{eq1} in the horizontal region in Figure \ref{f:positiveregion} with $\abs{y+1} \leq 1$ where $f_0$ has positive integral Fourier coefficients, and the first two coefficients of $f_1$ and $f_2$ are positive. 

\subsection{The diagonal fibers}
It remains finally to treat the diagonal fibers in Figure \ref{f:positiveregion}. Thus suppose that $x-y = a$ for some $\abs{a} \leq 1$. In fact, we may suppose that $y= x-a$ for $0 < a < 1$, since the cases where $a=0,1$ correspond to reducible monodromy, and we may assume $a > 0$ by making use of the $(x,y)$ symmetry of Equation \eqref{eq1}. In this case,
\begin{align*}
  v_p(B_k) = &c_p\left(-\tfrac{2}{3}x+\tfrac{2a-9}{6},k\right) + c_p\left(-\tfrac{2}{3}x+\tfrac {2a-7}{6},k\right)+c_p\left(-\tfrac{2}{3}x+\tfrac{2a-5}{6},k\right)\\
  &- c_p(-x-1,k)-c_p(a-x-1,k).
\end{align*}
By the classification of the possible monodromy representations of Section \ref{s:monodromy}, we need only consider the cases where $a = b/5$, $c/7$ or $d/16$, and then $x$ must also be a rational number with denominator supported at the same prime. These three cases can be treated as we treated the horizontal fibers in the previous subsection, by choosing primes so that the zeroth $p$-adic coefficient of $-x-1$ is large relative to the other quantities appearing above. It turns out that no new solutions to equation \eqref{eq1} arise in this diagonal region (outside of the boxed area where $\abs{x+1} \leq 5/2$, $\abs{y+1}\leq 5/2$ which was treated separately by a finite computation), where $f_0$ has positive and integral Fourier coefficients. This concludes our discussion of how to describe a list, corresponding to one infinite family and a number of sporadic exceptions, of solutions to Equation \eqref{eq1} that can be used to establish Theorem \ref{thmmain}.

In Figures \ref{f:full57} and \ref{f:full2} on pages \pageref{f:full57} and \pageref{f:full2}, we list all possible solutions to equation \eqref{eq1} such that $f_0$, $f_1$ and $f_2$ satisfy:
\begin{enumerate}
\item the monodromy is irreducible with a congruence subgroup as kernel;
\item the first thousand Fourier coefficients of $f_0$, $f_1$ and $f_2$ are all nonnegative;
\item the first thousand Fourier coefficients of $f_0$ are integers.
\end{enumerate}
We believe that (3) could be easily strengthened to show that $f_0$ is in fact positive integral in each case, but we have not gone to the trouble of doing so. This is because in all of the cases of interest for this paper, namely those corresponding to VOAs, integrality follows automatically since the Fourier coefficients count dimensions of finite dimensional vector spaces.

We shall show that most of the entries in Figures \ref{f:full57} and \ref{f:full2} are \emph{not} realized by a strongly regular VOA with exactly $3$ nonisomorphic simple modules and irreducible monic monodromy. Presumably some of these sets of parameters are realized by VOAs $V$ with a $3$-dimensional space of characters $\ch_V$ but more than $3$ simple modules, and therefore we include the full dataset.

\begin{figure}[h!]
   \renewcommand{\arraystretch}{1.35}
\begin{tabular}{ccc}
\begin{tabular}[t]{|c|c|c|c|c|}
  \hline
  $m$ & $h_1$ & $h_2$ & $c$ & $\widetilde{c}$ \\
  \hline
$0$ & $-\frac{1}{5}$ & $-\frac{2}{5}$ & $-\frac{44}{5}$& $\frac{4}{5}$ \\
$0$ & $\frac{12}{5}$ & $\frac{11}{5}$ & $\frac{164}{5}$& $\frac{164}{5}$ \\
$1$ & $\frac{1}{5}$ & $-\frac{1}{5}$ & $-4$ &$\frac{4}{5}$ \\
$2$ & $\frac{2}{5}$ & $\frac{1}{5}$ & $\frac{4}{5}$ &$\frac{4}{5}$ \\
$3$ & $-\frac{2}{5}$ & $-\frac{3}{5}$ & $-12$ &$\frac{12}{5}$ \\
$3$ & $\frac{3}{5}$ & $\frac{1}{5}$ & $\frac{12}{5}$ &$\frac{12}{5}$ \\
$10$ & $\frac{1}{5}$ & $-\frac{2}{5}$ & $-\frac{28}{5}$& $4$ \\
$24$ & $\frac{3}{5}$ & $\frac{2}{5}$ & $4$ & $4$ \\
$27$ & $\frac{9}{5}$ & $\frac{7}{5}$ & $\frac{108}{5}$ &$\frac{108}{5}$ \\
$28$ & $\frac{4}{5}$ & $\frac{2}{5}$ & $\frac{28}{5}$ &$\frac{28}{5}$ \\
$58$ & $\frac{9}{5}$ & $\frac{8}{5}$ & $\frac{116}{5}$ &$\frac{116}{5}$ \\
$92$ & $\frac{8}{5}$ & $\frac{6}{5}$ & $\frac{92}{5}$ &$\frac{92}{5}$ \\
$104$ & $\frac{6}{5}$ & $\frac{3}{5}$ & $\frac{52}{5}$ &$\frac{52}{5}$ \\
$105$ & $\frac{4}{5}$ & $-\frac{3}{5}$ & $-\frac{12}{5}$& $12$ \\
$120$ & $\frac{8}{5}$ & $\frac{7}{5}$ & $20$ & $20$ \\
$136$ & $\frac{7}{5}$ & $\frac{4}{5}$ & $\frac{68}{5}$ &$\frac{68}{5}$ \\
$144$ & $\frac{4}{5}$ & $\frac{3}{5}$ & $\frac{36}{5}$ &$\frac{36}{5}$ \\
  \hline
\end{tabular} &
                \begin{tabular}[t]{|c|c|c|c|c|}
  \hline
  $m$ & $h_1$ & $h_2$ & $c$ & $\widetilde{c}$ \\
  \hline
$156$ & $\frac{6}{5}$ & $-\frac{2}{5}$ & $\frac{12}{5}$& $12$ \\
$220$ & $\frac{6}{5}$ & $\frac{2}{5}$ & $\frac{44}{5}$ &$\frac{44}{5}$ \\
$222$ & $\frac{7}{5}$ & $\frac{3}{5}$ & $12$ & $12$ \\
$253$ & $\frac{7}{5}$ & $\frac{1}{5}$ & $\frac{44}{5}$ &$\frac{44}{5}$ \\
$312$ & $\frac{11}{5}$ & $-\frac{2}{5}$ & $\frac{52}{5}$& $20$ \\
$336$ & $\frac{7}{5}$ & $\frac{6}{5}$ & $\frac{84}{5}$ &$\frac{84}{5}$ \\
$374$ & $\frac{9}{5}$ & $\frac{2}{5}$ & $\frac{68}{5}$ &$\frac{68}{5}$ \\
$380$ & $\frac{8}{5}$ & $\frac{4}{5}$ & $\frac{76}{5}$ &$\frac{76}{5}$ \\
$437$ & $\frac{9}{5}$ & $\frac{3}{5}$ & $\frac{76}{5}$ &$\frac{76}{5}$ \\
$534$ & $\frac{12}{5}$ & $\frac{1}{5}$ & $\frac{84}{5}$& $\frac{84}{5}$ \\
$690$ & $\frac{11}{5}$ & $\frac{3}{5}$ & $\frac{92}{5}$& $\frac{92}{5}$ \\
$860$ & $\frac{14}{5}$ & $\frac{2}{5}$ & $\frac{108}{5}$ & $\frac{108}{5}$ \\
$1404$ & $\frac{12}{5}$ & $\frac{4}{5}$ & $\frac{108}{5}$& $\frac{108}{5}$ \\
$1536$ & $\frac{16}{5}$ & $\frac{3}{5}$ & $\frac{132}{5}$& $\frac{132}{5}$ \\
$1711$ & $\frac{13}{5}$ & $\frac{4}{5}$ & $\frac{116}{5}$& $\frac{116}{5}$ \\
$3612$ & $\frac{18}{5}$ & $\frac{4}{5}$ & $\frac{156}{5}$& $\frac{156}{5}$ \\
$13110$ & $\frac{33}{5}$ & $\frac{4}{5}$ & $\frac{276}{5}$& $\frac{276}{5}$ \\
  \hline
\end{tabular} &
\begin{tabular}[t]{|c|c|c|c|c|}
  \hline
  $m$ & $h_1$ & $h_2$ & $c$ & $\widetilde{c}$ \\
  \hline
$0$ & $-\frac{2}{7}$ & $-\frac{3}{7}$ & $-\frac{68}{7}$& $\frac{4}{7}$ \\
$0$ & $\frac{17}{7}$ & $\frac{16}{7}$ & $\frac{236}{7}$& $\frac{236}{7}$ \\
$1$ & $\frac{2}{7}$ & $-\frac{1}{7}$ & $-\frac{20}{7}$ &$\frac{4}{7}$ \\
$1$ & $\frac{3}{7}$ & $\frac{1}{7}$ & $\frac{4}{7}$ &$\frac{4}{7}$ \\
$6$ & $\frac{3}{7}$ & $\frac{2}{7}$ & $\frac{12}{7}$ &$\frac{12}{7}$ \\
$41$ & $\frac{13}{7}$ & $\frac{11}{7}$ & $\frac{164}{7}$& $\frac{164}{7}$ \\
$78$ & $\frac{12}{7}$ & $\frac{11}{7}$ & $\frac{156}{7}$& $\frac{156}{7}$ \\
$88$ & $\frac{5}{7}$ & $\frac{4}{7}$ & $\frac{44}{7}$ &$\frac{44}{7}$ \\
$156$ & $\frac{6}{7}$ & $\frac{4}{7}$ & $\frac{52}{7}$ &$\frac{52}{7}$ \\
$210$ & $\frac{8}{7}$ & $\frac{3}{7}$ & $\frac{60}{7}$ &$\frac{60}{7}$ \\
$221$ & $\frac{9}{7}$ & $\frac{3}{7}$ & $\frac{68}{7}$ &$\frac{68}{7}$ \\
$248$ & $\frac{10}{7}$ & $\frac{9}{7}$ & $\frac{124}{7}$& $\frac{124}{7}$ \\
$325$ & $\frac{11}{7}$ & $\frac{5}{7}$ & $\frac{100}{7}$& $\frac{100}{7}$ \\
$348$ & $\frac{10}{7}$ & $\frac{8}{7}$ & $\frac{116}{7}$& $\frac{116}{7}$ \\
$378$ & $\frac{11}{7}$ & $\frac{6}{7}$ & $\frac{108}{7}$& $\frac{108}{7}$ \\
$380$ & $\frac{12}{7}$ & $\frac{4}{7}$ & $\frac{100}{7}$& $\frac{100}{7}$ \\
$456$ & $\frac{13}{7}$ & $\frac{4}{7}$ & $\frac{108}{7}$& $\frac{108}{7}$ \\
$1248$ & $\frac{18}{7}$ & $\frac{5}{7}$ & $\frac{156}{7}$& $\frac{156}{7}$ \\
  \hline
\end{tabular}
\end{tabular}
\caption{Full dataset of parameters with $f_0$ nonnegative integral and $f_1$ and $f_2$ nonnegative, where parameters have denominators $5$ and $7$. Since we only used one-thousand Fourier coefficients to generate this data, some of these series could in fact fail to be integral, but certainly the integral list is a subset of ours.}
\label{f:full57}
\end{figure}

\begin{figure}[h!]
   \renewcommand{\arraystretch}{1.2}
  \begin{tabular}{ccc}
\begin{tabular}[t]{|c|c|c|c|c|}
  \hline
  $m$ & $h_1$ & $h_2$ & $c$ & $\widetilde{c}$ \\
  \hline
$0$ & $\frac{31}{16}$ & $\frac{3}{2}$ & $\frac{47}{2}$ &
$\frac{47}{2}$ \\
$1$ & $-\frac{7}{16}$ & $-\frac{1}{2}$ & $-\frac{23}{2}$
& $\frac{1}{2}$ \\
$1$ & $-\frac{3}{8}$ & $-\frac{1}{2}$ & $-11$ & $1$ \\
$1$ & $\frac{7}{16}$ & $-\frac{1}{16}$ & $-1$ &
$\frac{1}{2}$ \\
$2$ & $-\frac{1}{4}$ & $-\frac{1}{2}$ & $-10$ & $2$ \\
$2$ & $\frac{3}{8}$ & $-\frac{1}{8}$ & $-2$ & $1$ \\
$3$ & $\frac{5}{16}$ & $-\frac{3}{16}$ & $-3$ &
$\frac{3}{2}$ \\
$4$ & $\frac{1}{4}$ & $-\frac{1}{4}$ & $-4$ & $2$ \\
$5$ & $\frac{3}{16}$ & $-\frac{5}{16}$ & $-5$ &
$\frac{5}{2}$ \\
$6$ & $\frac{1}{8}$ & $-\frac{3}{8}$ & $-6$ & $3$ \\
$7$ & $\frac{1}{16}$ & $-\frac{7}{16}$ & $-7$ &
$\frac{7}{2}$ \\
$9$ & $-\frac{1}{16}$ & $-\frac{9}{16}$ & $-9$ &
$\frac{9}{2}$ \\
$10$ & $-\frac{1}{8}$ & $-\frac{5}{8}$ & $-10$ & $5$ \\
$11$ & $-\frac{3}{16}$ & $-\frac{11}{16}$ & $-11$ &
$\frac{11}{2}$ \\
$12$ & $-\frac{1}{4}$ & $-\frac{3}{4}$ & $-12$ & $6$ \\
$13$ & $-\frac{5}{16}$ & $-\frac{13}{16}$ & $-13$ &
$\frac{13}{2}$ \\
$14$ & $-\frac{3}{8}$ & $-\frac{7}{8}$ & $-14$ & $7$ \\
$15$ & $-\frac{7}{16}$ & $-\frac{15}{16}$ & $-15$ &
$\frac{15}{2}$ \\
$17$ & $-\frac{9}{16}$ & $-\frac{17}{16}$ & $-17$ &
$\frac{17}{2}$ \\
$18$ & $-\frac{5}{8}$ & $-\frac{9}{8}$ & $-18$ & $9$ \\
$19$ & $-\frac{11}{16}$ & $-\frac{19}{16}$ & $-19$ &
$\frac{19}{2}$ \\
$20$ & $-\frac{3}{4}$ & $-\frac{5}{4}$ & $-20$ & $10$ \\
$21$ & $-\frac{13}{16}$ & $-\frac{21}{16}$ & $-21$ &
$\frac{21}{2}$ \\
$22$ & $-\frac{7}{8}$ & $-\frac{11}{8}$ & $-22$ & $11$
\\
$23$ & $-\frac{15}{16}$ & $-\frac{23}{16}$ & $-23$ &
$\frac{23}{2}$ \\
$23$ & $\frac{15}{8}$ & $\frac{3}{2}$ & $23$ & $23$ \\
$25$ & $-\frac{17}{16}$ & $-\frac{25}{16}$ & $-25$ &
$\frac{25}{2}$ \\
$26$ & $-\frac{9}{8}$ & $-\frac{13}{8}$ & $-26$ & $13$
\\
$27$ & $-\frac{19}{16}$ & $-\frac{27}{16}$ & $-27$ &
$\frac{27}{2}$ \\
$28$ & $-\frac{5}{4}$ & $-\frac{7}{4}$ & $-28$ & $14$ \\

                     \hline
                   \end{tabular}&
\begin{tabular}[t]{|c|c|c|c|c|}
  \hline
  $m$ & $h_1$ & $h_2$ & $c$ & $\widetilde{c}$ \\
  \hline
$29$ & $-\frac{21}{16}$ & $-\frac{29}{16}$ & $-29$ &
$\frac{29}{2}$ \\
  $30$ & $-\frac{11}{8}$ & $-\frac{15}{8}$ & $-30$ & $15$
\\
$31$ & $-\frac{23}{16}$ & $-\frac{31}{16}$ & $-31$ &
$\frac{31}{2}$ \\
$33$ & $-\frac{25}{16}$ & $-\frac{33}{16}$ & $-33$ &
$\frac{33}{2}$ \\
$34$ & $-\frac{13}{8}$ & $-\frac{17}{8}$ & $-34$ & $17$
\\
$35$ & $-\frac{27}{16}$ & $-\frac{35}{16}$ & $-35$ &
$\frac{35}{2}$ \\
$36$ & $-\frac{7}{4}$ & $-\frac{9}{4}$ & $-36$ & $18$ \\
$37$ & $-\frac{29}{16}$ & $-\frac{37}{16}$ & $-37$ &
$\frac{37}{2}$ \\
$38$ & $-\frac{15}{8}$ & $-\frac{19}{8}$ & $-38$ & $19$
\\
$39$ & $-\frac{31}{16}$ & $-\frac{39}{16}$ & $-39$ &
$\frac{39}{2}$ \\
$45$ & $\frac{29}{16}$ & $\frac{3}{2}$ & $\frac{45}{2}$
& $\frac{45}{2}$ \\
$66$ & $\frac{7}{4}$ & $\frac{3}{2}$ & $22$ & $22$ \\
$86$ & $\frac{27}{16}$ & $\frac{3}{2}$ & $\frac{43}{2}$
& $\frac{43}{2}$ \\
$105$ & $\frac{13}{8}$ & $\frac{3}{2}$ & $21$ & $21$ \\
$123$ & $\frac{25}{16}$ & $\frac{3}{2}$ & $\frac{41}{2}$
& $\frac{41}{2}$ \\
$156$ & $\frac{3}{2}$ & $\frac{23}{16}$ & $\frac{39}{2}$
& $\frac{39}{2}$ \\
$171$ & $\frac{3}{2}$ & $-\frac{7}{16}$ & $\frac{9}{2}$
& $15$ \\
$171$ & $\frac{3}{2}$ & $\frac{11}{8}$ & $19$ & $19$ \\
$185$ & $\frac{3}{2}$ & $-\frac{3}{8}$ & $5$ & $14$ \\
$185$ & $\frac{3}{2}$ & $\frac{21}{16}$ & $\frac{37}{2}$
& $\frac{37}{2}$ \\
$198$ & $\frac{3}{2}$ & $-\frac{5}{16}$ & $\frac{11}{2}$
& $13$ \\
$198$ & $\frac{3}{2}$ & $\frac{5}{4}$ & $18$ & $18$ \\
$210$ & $\frac{3}{2}$ & $-\frac{1}{4}$ & $6$ & $12$ \\
$210$ & $\frac{3}{2}$ & $\frac{19}{16}$ & $\frac{35}{2}$
& $\frac{35}{2}$ \\
$221$ & $\frac{3}{2}$ & $-\frac{3}{16}$ & $\frac{13}{2}$
& $11$ \\
$221$ & $\frac{3}{2}$ & $\frac{9}{8}$ & $17$ & $17$ \\
  $231$ & $\frac{3}{2}$ & $-\frac{1}{8}$ & $7$ & $10$ \\
$231$ & $\frac{3}{2}$ & $\frac{17}{16}$ & $\frac{33}{2}$
                            & $\frac{33}{2}$ \\
  $240$ & $\frac{3}{2}$ & $-\frac{1}{16}$ & $\frac{15}{2}$
& $9$ \\
$248$ & $\frac{3}{2}$ & $\frac{15}{16}$ & $\frac{31}{2}$
& $\frac{31}{2}$ \\
  \hline
                   \end{tabular}&
\begin{tabular}[t]{|c|c|c|c|c|}
  \hline
  $m$ & $h_1$ & $h_2$ & $c$ & $\widetilde{c}$ \\
  \hline
$255$ & $\frac{3}{2}$ & $\frac{1}{16}$ & $\frac{17}{2}$
& $\frac{17}{2}$ \\
$255$ & $\frac{3}{2}$ & $\frac{7}{8}$ & $15$ & $15$ \\
$261$ & $\frac{3}{2}$ & $\frac{1}{8}$ & $9$ & $9$ \\
$261$ & $\frac{3}{2}$ & $\frac{13}{16}$ & $\frac{29}{2}$
& $\frac{29}{2}$ \\
$266$ & $\frac{3}{2}$ & $\frac{3}{16}$ & $\frac{19}{2}$
& $\frac{19}{2}$ \\
$266$ & $\frac{3}{2}$ & $\frac{3}{4}$ & $14$ & $14$ \\
$270$ & $\frac{3}{2}$ & $\frac{1}{4}$ & $10$ & $10$ \\
$270$ & $\frac{3}{2}$ & $\frac{11}{16}$ & $\frac{27}{2}$
& $\frac{27}{2}$ \\
$273$ & $\frac{3}{2}$ & $\frac{5}{16}$ & $\frac{21}{2}$
& $\frac{21}{2}$ \\
$273$ & $\frac{3}{2}$ & $\frac{5}{8}$ & $13$ & $13$ \\
$275$ & $\frac{3}{2}$ & $\frac{3}{8}$ & $11$ & $11$ \\
$275$ & $\frac{3}{2}$ & $\frac{9}{16}$ & $\frac{25}{2}$
& $\frac{25}{2}$ \\
$276$ & $\frac{3}{2}$ & $\frac{7}{16}$ & $\frac{23}{2}$
& $\frac{23}{2}$ \\
$496$ & $\frac{45}{16}$ & $-\frac{5}{16}$ & $16$ &
$\frac{47}{2}$ \\
$496$ & $\frac{23}{8}$ & $-\frac{3}{8}$ & $16$ & $25$ \\
$496$ & $\frac{47}{16}$ & $-\frac{7}{16}$ & $16$ &
$\frac{53}{2}$ \\
$496$ & $\frac{49}{16}$ & $-\frac{9}{16}$ & $16$ &
$\frac{59}{2}$ \\
$496$ & $\frac{25}{8}$ & $-\frac{5}{8}$ & $16$ & $31$ \\
$496$ & $\frac{51}{16}$ & $-\frac{11}{16}$ & $16$ &
$\frac{65}{2}$ \\
  $598$ & $\frac{5}{2}$ & $\frac{1}{4}$ & $18$ & $18$ \\
$1118$ & $\frac{7}{2}$ & $\frac{1}{4}$ & $26$ & $26$ \\
$1194$ & $\frac{9}{2}$ & $-\frac{1}{4}$ & $30$ & $36$ \\
  $1298$ & $\frac{5}{2}$ & $\frac{3}{4}$ & $22$ & $22$ \\
$1640$ & $\frac{5}{2}$ & $\frac{13}{16}$ & $\frac{45}{2}$
& $\frac{45}{2}$ \\
$2323$ & $\frac{5}{2}$ & $\frac{7}{8}$ & $23$ & $23$ \\
$2778$ & $\frac{7}{2}$ & $\frac{3}{4}$ & $30$ & $30$ \\
$3599$ & $\frac{7}{2}$ & $\frac{13}{16}$ & $\frac{61}{2}$
& $\frac{61}{2}$ \\
  $4371$ & $\frac{5}{2}$ & $\frac{15}{16}$ & $\frac{47}{2}$
& $\frac{47}{2}$ \\
$5239$ & $\frac{7}{2}$ & $\frac{7}{8}$ & $31$ & $31$ \\
                     \hline
                   \end{tabular}
\end{tabular}
\caption{Full dataset for imprimitive representations. This does not include the one infinite family that we shall treat separately in Section \ref{s:quadraticfamily}. Also, the same remark on integrality as in Figure \ref{f:full57} applies here.}
\label{f:full2}
\end{figure}

\section{Trimming down to Theorem \ref{thmmain}}
\label{s:sieve}
Most of the potential examples that are tabulated in Figures \ref{f:full57} and \ref{f:full2} do not in fact correspond to a VOA satisfying the conditions of Theorem \ref{thmmain}. In this Section we explain how to trim these lists down to the statement of Theorem \ref{thmmain}.

First we shall use the deep fact, which follows by Huang \cite{Huang}, that $\rho(S)$ must be a \emph{symmetric} matrix, with $\rho(T)$ diagonal. Since $\rho(T)$ has distinct eigenvalues (cf. Section \ref{s:monodromy}), and the only invertible matrices that commute with a diagonal matrix with distinct eigenvalues are the diagonal matrices, the only remaining freedom in changing the basis is in conjugating by diagonal matrices. Since we wish to keep the coordinate $f_0$ fixed, this conjugation amounts to rescaling $f_1$ and $f_2$. Said differently, there is \emph{at most one choice of scalars} $A_1$ and $A_2$ appearing in the definition of $f_1$ and $f_2$ such that $\rho(S)$ is symmetric. Since $A_1$ and $A_2$ must themselves be integers, we performed a numerical computation in all of the finitely many remaining cases (with $y \neq -1/2$) to symmetrize $\rho(S)$ and compute exact values for $A_1$ and $A_2$. The idea was to numerically evaluate $F(\tau)$ and $F(-1/\tau)$ at random points $\tau$ near $i$, and by comparing the results we obtained a numerical expression for $\rho(S)$ to high enough precision to determine when $A_1$ and $A_2$ were nonnegative integers. Note that the hypergeometric expression for $F(\tau)$ is very well-suited to this type of high precision computation. Also, since we used finite precision computation, we could only rule out exactly when $A_1$ and $A_2$ are \emph{not} integers. The exact values for $A_1$ and $A_2$ that we report here are then justified since in each case we produce examples of VOAs realizing them.

After computing values for $A_1$ and $A_2$, we were then able to test the integrality of the first thousand coefficients of all three coordinates of $F(\tau)$, whereas previously we had only been able to make use of the integrality of the first coordinate $f_0$. This cut our list of possible character vectors down very dramatically. We then checked the remaining cases to verify that the Verlinde formula holds for $\rho(S)$. After all of this work, we found the following exhaustive list of sets of data that could possibly correspond to a VOA as in Theorem \ref{thmmain}:
\begin{enumerate}
\label{examples1}
\item examples corresponding to solutions of \eqref{eq1} with $y = -1/2$;
\item $(m,c,h_1,h_2) = (0,-68/7,-2/7,-3/7)$ which is realized by $\Vir(c_{2,7})$;
\item $(m,c,h_1,h_2) = (24,4,2/5,3/5)$ which is realized by $A_{4,1}$;
\item $11$ exceptional cases with $y=1/2$, equivalently, $h_1 = 3/2$.
\end{enumerate}
\begin{dfn}
  \label{d:useries}
  The $11$ exceptional examples with $h_1 = 3/2$ comprise the \emph{$U$-series}.
\end{dfn}
We shall discuss the $U$-series in greater detail in Section \ref{SU}.

Figure \ref{f:full3} on page \pageref{f:full3} lists data for the $U$-series, and Figure \ref{f:fouriercoeffs} on page \pageref{f:fouriercoeffs} lists the first several Fourier coefficients of $f_0$, $f_1$ and $f_2$ for the examples in the $U$-series. For convenience we recall here the formulas for the character vector $F(\tau) = (f_0,f_1,f_2)^T$ in terms of the parameters $h_1 = x+1$, $h_2 = y+1$ and $c = 8(h_1+h_2)-4$:
\begin{align*}
  f_0 &= j^{\frac{c}{24}}{}_3F_2\left(-\tfrac{c}{24},\tfrac{8-c}{24},\tfrac{16-c}{24};1-h_1,1-h_2;\tfrac{1728}{j}\right),\\
  f_1 &= A_1j^{\frac{2h_2-4h_1-1}{6}}{}_3F_2\left(\tfrac{4h_1-2h_2+1}{6},\tfrac{4h_1-2h_2+3}{6},\tfrac{4h_1-2h_2+5}{6};h_1,h_1-h_2;\tfrac{1728}{j}\right),\\
  f_2 &= A_2j^{\frac{2h_1-4h_2-1}{6}}{}_3F_2\left(\tfrac{4h_2-2h_1+1}{6},\tfrac{4h_2-2h_1+3}{6},\tfrac{4h_2-2h_1+5}{6};h_2,h_2-h_1;\tfrac{1728}{j}\right).
\end{align*}
Further, for the examples in the $U$-series $\rho(T) = \exp\left(2\pi i \diagg(-\frac{c}{24}, \frac{4h_1-2h_2+1}{6},\frac{4h_2-2h_1+1}{6})\right)$ and 
\[
\rho(S) = \frac 12\left(\begin{matrix}
     1&1&\sqrt{2}\\
     1&1&-\sqrt{2}\\
     \sqrt{2}& -\sqrt{2}&0
   \end{matrix}\right).
\]

In Section \ref{s:quadraticfamily} we shall discuss the existence of VOAs for the infinite family of solutions to Equation \eqref{eq1} with $y = -1/2$, and in Section \ref{SU} we provide some more detail about the $U$-series.

\begin{figure}[h!]
   \renewcommand{\arraystretch}{2}
\begin{tabular}[t]{|c|c|c|c|c|c|c|c|}
  \hline
  $m$ & $h_1$ & $h_2$ & $c$ & $\widetilde{c}$& $V_1$ & $A_1$ & $A_2$\\
  \hline
  $0$ & $\frac 32$ &$\frac{31}{16}$ & $\frac{47}{2}$ &$\frac{47}{2}$& $0$ &$4371$&$96256$\\
  \hline
  $45$ &$\frac 32$ & $\frac{29}{16}$ & $\frac{45}{2}$& $\frac{45}{2}$& $D_5$ *&$4785$&$46080$\\
    \hline
  $86$ & $\frac 32$ &$\frac{27}{16}$ & $\frac{43}{2}$& $\frac{43}{2}$&$2B_4\oplus G_2$ * &$5031$&$22016$\\
    \hline
  $123$ & $\frac 32$ &$\frac{25}{16}$ & $\frac{41}{2}$& $\frac{41}{2}$&$A_1\oplus A_{10}$ * &$5125$&$10496$\\
    \hline
  $156$ & $\frac{3}{2}$ & $\frac{23}{16}$ & $\frac{39}{2}$& $\frac{39}{2}$& $2B_{6}$ *&$5083$& $4992$\\
    \hline
  $185$ & $\frac{3}{2}$ & $\frac{21}{16}$ & $\frac{37}{2}$& $\frac{37}{2}$&$E_7\oplus F_4$ * &$4921$&$2368$\\
    \hline
  $210$ & $\frac{3}{2}$ & $\frac{19}{16}$ & $\frac{35}{2}$& $\frac{35}{2}$ &$B_{10}$ *&$4655$&$1120$\\
    \hline
  $231$ & $\frac{3}{2}$ & $\frac{17}{16}$ & $\frac{33}{2}$& $\frac{33}{2}$& $D_{11}$ * &$4301$&$528$\\
    \hline
  $248$ & $\frac{3}{2}$ & $\frac{15}{16}$ & $\frac{31}{2}$& $\frac{31}{2}$&$E_8$ &$3875$&$248$\\
      &&&&&$A_1\oplus D_{11}\oplus G_2$ &&\\
    \hline
  $261$ & $\frac{3}{2}$ & $\frac{13}{16}$ & $\frac{29}{2}$& $\frac{29}{2}$&$A_2\oplus B_{11}$ *&$3393$&$116$ \\
    \hline
  $270$ & $\frac{3}{2}$ & $\frac{11}{16}$ & $\frac{27}{2}$& $\frac{27}{2}$& $A_2\oplus E_8\oplus G_2$ &$2871$&$54$\\
      &&&&&$\CC\oplus B_3\oplus E_8$&&\\
    &&&&&$\CC\oplus C_3\oplus E_8$&&\\
                    \hline
\end{tabular}
\caption{Data for the $U$-series. The expressions for $V_1$ are far from unique in general. In the cases where there are at most three possibilities for $V_1$, we have listed all of them. In all other cases there are several (in some cases thousands) of possibilities and we have only listed one of them, along with an asterisk. }
\label{f:full3}
\end{figure}

\begin{figure}[h!]
   \renewcommand{\arraystretch}{1.2}
\begin{tabular}[t]{|c|c|c|}
  \hline
  $h_1$ & $h_2$ & $a_0,a_1,a_2,\ldots$\\
  \hline
   $\frac 32$ &$\frac{31}{16}$ & $1, 0, 96256, 9646891, 366845011, 8223700027, 130416170627,\ldots$ \\
  &&$4371, 1143745, 64680601, 1829005611, 33950840617, 470887671187,\ldots$\\
  &&$96256, 10602496, 420831232, 9685952512, 156435924992,1958810851328,\ldots $\\
  \hline
  $\frac 32$ & $\frac{29}{16}$ & $1, 45, 90225, 7671525, 260868780, 5354634636, 78809509455,\ldots$\\
  &&$4785, 977184, 48445515, 1241925725, 21267996075, 275102618220,\ldots$\\
  &&$46080, 5161984, 199388160, 4423680000, 68709350400, 827293870080,\ldots$\\
    \hline
  $\frac 32$ &$\frac{27}{16}$ & $1, 86, 82775, 5989341, 182136390, 3421630228, 46706033862,\ldots$\\
    &&$5031, 819279, 35627220, 827820606, 13070793291, 157564970907,\ldots$\\
    &&$22016, 2515456, 94360576, 2013605376, 30017759232, 346922095616,\ldots$\\
    \hline
  $\frac 32$ &$\frac{25}{16}$ & $1, 123, 74374, 4586752, 124739876, 2143484264, 27115530974,\ldots $\\
   &&$5125, 673630, 25702490, 541136245, 7872255635, 88368399005, 816197168410,\ldots$\\
   &&$10496, 1227008, 44597504, 913172992, 13037354496, 144348958464,\ldots$\\
    \hline
  $\frac{3}{2}$ & $\frac{23}{16}$ & $1, 156, 65442, 3442179, 83713890, 1314851889, 15401260043, 145567687044,\ldots$\\
  &&$5083, 542685, 18172323, 346513193, 4640683320, 48464931804, 419554761418,\ldots$\\
  &&$4992, 599168, 21046272, 412414080, 5625756032, 59548105344, 520893998976,\ldots$\\
    \hline
  $\frac{3}{2}$ & $\frac{21}{16}$ & $1, 185, 56351, 2528691, 54987069, 788715865, 8545883340, 75369712213,\ldots$\\
  &&$4921, 427868, 12578261, 217080369, 2673896760, 25953557278, 210363766807\ldots$\\
  &&$2368, 292928, 9914816, 185395456, 2410143296, 24333700608, 203337098176\ldots$\\
    \hline
  $\frac{3}{2}$ & $\frac{19}{16}$ & $1, 210, 47425, 1816325, 35302155, 461945596, 4624903605, 38016539200,\ldots$\\
  &&$4655, 329707, 8512950, 132853700, 1503485200, 13547531620, 102694766167,\ldots$\\
  &&$1120, 143392, 4661440, 82908000, 1024273600, 9839831680, 78373048544,\ldots$\\
    \hline
  $\frac{3}{2}$ & $\frac{17}{16}$ & $1, 231, 38940, 1274086, 22116963, 263714253, 2436524530, 18642807645,\ldots$\\
  &&$4301, 247962, 5625708, 79296041, 823487514, 6879624345, 48709339624,\ldots$\\
  &&$528, 70288, 2186448, 36857568, 431399936, 3932664912, 29784812640, \ldots$\\
    \hline
  $\frac{3}{2}$ & $\frac{15}{16}$ & $1, 248, 31124, 871627, 13496501, 146447007, 1246840863, 8867414995, \ldots$\\
  &&$3875, 181753, 3623869, 46070247, 438436131, 3390992753, 22393107641,\ldots$\\
  &&$248, 34504, 1022752, 16275496, 179862248, 1551303736, 11142792024, \ldots$\\
    \hline
  $\frac{3}{2}$ & $\frac{13}{16}$ & $1, 261, 24157, 580609, 8004754, 78925762, 618182705, 4079878514, \ldots$\\
  &&$3393, 129688, 2270671, 25996789, 226351177, 1618088408, 9950251364, \ldots$\\
  &&$116, 16964, 476876, 7131680, 74132236, 602971480, 4095721620, \ldots$\\
    \hline
  $\frac{3}{2}$ & $\frac{11}{16}$ & $1, 270, 18171, 375741, 4602852, 41167332, 296065548, 1809970083, \ldots$\\
  &&$2871, 89991, 1380456, 14210922, 112987953, 745155153, 4259274975,\ldots$\\
  &&$54, 8354, 221508, 3097278, 30156048, 230475996, 1475743590, 8240806224,\ldots$\\
    \hline
\end{tabular}
\caption{Fourier coefficients for the characters of the $U$-series.}
\label{f:fouriercoeffs}
\end{figure}

\section{Solutions with $y = -1/2$}
\label{s:quadraticfamily}

We turn now to the solutions of \eqref{eq1} with $y=-1/2$. Recall that equation \eqref{eq1} specializes to
\[
  x(128x^2+248x-m+120) = 0,
\]
and we can ignore the solutions $(m,0,-1/2)$ since they correspond to reducible monodromy representations (cf. Section \ref{s:monodromy}). Therefore we now study the solutions
\[
  (m,x,y) = \left(\tfrac{s(s-1)}{2}, \tfrac{s}{16}-1,-\tfrac{1}{2}\right)
\]
where $s> 0$ is an integer that is not divisible by $8$. The restriction $s > 0$ arises from the fact that $m = \dim V_1$ must be a nonnegative integer, and the restriction that $8$ does not divide $s$ is due to the irreducibility of the monodromy cf. Section \ref{s:monodromy}. The main result of this section, whose proof occupies the remainder of the section, classifies exactly what VOAs satisfying the restrictions of Theorem \ref{thmmain} correspond to these examples:
\begin{thm}
  \label{thmquad1}
 Suppose that $(m,x, y)= \left(\tfrac{s(s-1)}{2}, \tfrac{s}{16}-1,-\tfrac{1}{2}\right)$ for an integer $s > 0$ not divisible by $8$. Then $V$ is isomorphic to one of the following:
 \[
   B_{\ell,1},~ A_{1,2},~ \Vir(c_{3,4}).
 \]
\end{thm}
Remembering that $c=8\left(h_1+h_2-\tfrac{1}{2}\right)=8\left(x+y+\tfrac{3}{2}\right)$ we have $c = \tilde c = s/2$. In particular we have $c{\in}\tfrac{1}{2}\mathbf{Z}$, so that Corollary \ref{corc/2} applies. Our approach to the proof of Theorem \ref{thmquad1} is to deal separately with each of the possibilities (a)-(c) of Corollary \ref{corc/2}, although the arguments are similar in each case. We try to determine the structure of the Lie algebra $V_1$, or else prove that there is no choice of $V_1$ that is compatible with the data. A basic property \cite{DMRational} is that $V_1$ is \emph{reductive} and its Lie rank is denoted by $\ell$ (cf.\ Subsection \ref{SScl}). As for case (a), we will prove
\begin{prop}
  \label{propa}
  For $(m,x, y)$ as in Theorem \ref{thmquad1}, Corollary \ref{corc/2}(a) \emph{cannot} hold.
\end{prop}
\begin{proof}
Until further notice we assume that the Proposition is false.

\begin{lem}
\label{lemineqs}
  We have
\[
2\ell^2+3\ell+1 \leq m.   
\]
\end{lem}
\begin{proof}
  Because we are assuming that (a) of Corollary \ref{corc/2} holds, then $\tilde{c} \geq \ell+1$. Therefore,
\[
m=\tfrac{s(s-1)}{2}=\tilde c(2\tilde{c}-1)\geq(\ell+1)(2\ell+1)=2\ell^2+3\ell+1.
\]
\end{proof}

As a reductive Lie algebra, $V_1$ has a direct sum decomposition
\begin{equation}
\label{Levi}
V_1=\bigoplus_{i\geq 0} \mathfrak{g}_i
\end{equation}
where $\mathfrak{g}_0$ is abelian and each $\mathfrak{g}_i$ is a nonabelian simple Lie algebra ($i \geq 1$), say of Lie rank $\ell_i$. Let $\ell_0\df\dim\mathfrak{g}_0$.Then the total Lie rank of $V_1$ is $\ell{=}\sum_{i\geq 0} \ell_i$.

The table of dimensions for simple Lie algebras compared with Lie rank is given in Table \ref{Table:DIML}.

 \begin{table}[htbp]
\begin{center}
\begin{tabular}{c|ccccccccccccc}
$\ell$&$1$&$2$&$3$&$4$&$5$&$6$&$7$&$8$&$9$&$10$&$\ell$\\
\hline
$A_\ell$&$3$&$8$&$15$&$24$&$35$&$48$&$63$&$80$&$99$&$120$&$\ell^2{+}2\ell$\\
$B_\ell$&$$&$10$&$21$&$36$&$55$&$78$&$105$&$136$&$171$&$220$&$2\ell^2{+}\ell$\\
$C_\ell$&$$&$$&$21$&$36$&$55$&$78$&$105$&$136$&$171$&$220$&$2\ell^2{+}\ell$\\
$D_\ell$&$$&$$&$$&$28$&$45$&$66$&$ 91$&$120$&$15$3&$190$&$2\ell^2{-}\ell$\\
\hline
$\CC$&$1$&&&&&&&&\\
$G_2$&&$14$&&&&&&&\\
$F_4$&&&&$52$&&&&&\\
$E_6$&&&&&&$78$&&&&\\
$E_7$&&&&&&&$133$&&&\\
$E_8$&&&&&&&&$248$&&\\
\end{tabular}
\end{center}
\caption{Ranks and Dimensions of simple Lie algebras}
\label{Table:DIML}
\end{table}

Now suppose that $V_1$ \emph{only} has components $\mathfrak{g}_i$ that are classical (type $A_{\ell},\ldots, D_{\ell}$) or of type $G_2$ or $E_6$. By Table \ref{Table:DIML} each of these satisfies $\dim\mathfrak{g}_i\leq 2\ell_i^2+3\ell_i$. Using Lemma \ref{lemineqs} we have
\[
  2\ell^2+3\ell+1\leq\dim V_1\leq\ell_0+\sum_{i\geq1}(2\ell_i^2+3\ell_i)\leq 3\ell+2\sum_{i\geq 1}\ell_i^2
\]
so that $\ell^2+\tfrac{1}{2}\leq\sum_{i\geq 1}\ell_i^2$, and this is impossible because each $\ell_i$ is a positive integer and $\ell$ is their sum. This shows, with a rather naive use of inequalities, that $V_1$ must have some component that is exceptional of type $F_4$, $E_7$ or $E_8$.

We will rework this argument. So essentially we backtrack because the inequalities can be improved as we gain more restrictions on the $\mathfrak{g}_i$. For any exceptional simple component $\mathfrak{g}_i$ of type $F_4$, $E_7$ or $E_8$ we write $\dim\mathfrak{g}_i=2\ell_i^2+\ell_i+e_i$ and let $1 \leq i \leq e$ index such components. Note that $e_i \leq 14 \ell_i$, with equality being met only if $\mathfrak{g}_1=E_8$. Then we have
\begin{align*}
2\ell^2+3\ell+1\leq&\ell_0+\sum_{i=1}^e (2\ell_i^2+\ell_i+e_i)+\sum_{i>e} \dim\mathfrak{g}_i \\
  \leq&\ell_0+\sum_{i=1}^e (2\ell_i^2+\ell_i+e_i)+\sum_{i>e}(2\ell_i^2+\ell_i)\\
  =&\ell+\sum_{i=1}^e e_i+2\sum_{i\geq1}\ell_i^2
\end{align*}
It follows that
\begin{eqnarray}
  \label{ineq6}
&&2\ell^2+2\ell+1 \leq \sum_{i=1}^e e_i+2\sum_{i\geq 1}\ell_i^2\\
\Rightarrow&&2\ell_0^2+4\sum_{0\leq i<j} \ell_i\ell_j+2\sum_{i\geq 0}\ell_i+1\leq\sum_{i=1}^e e_i\leq 14\sum_{i=1}^e\ell_i \notag\\
\Rightarrow&&(\ell_0^2+\ell_0)+2\sum_{0\leq i<j} \ell_i\ell_j<6\sum_{i=1}^e\ell_i\notag 
\end{eqnarray}

Because the possible exceptional components are $F_4$, $E_7$ and $E_8$, and because there is at least one of them, the \emph{minimum} of the $\ell_i\ (1 \leq i \leq e)$ is \emph{at least} $4$ and \emph{at most} $8$. Then the previous inequality implies that
\[
  (\ell_0^2+9\ell_0)+2\ell_1\sum_{2\leq j} \ell_j \ \   {+} 2\sum_{0\leq i < j, i\neq 1} \ell_i\ell_j   <6\ell_1+6\sum_{j=2}^e \ell_j
\]
and so
\begin{equation}
  \label{ineq5}
(\ell_0^2+9\ell_0)+(2\ell_1-6)\left(\sum_{j=2}^e \ell_j\right) + 2\ell_1\left(\sum_{j> e} \ell_j\right) + 2\sum_{0\leq i < j, i\neq 1} \ell_i\ell_j   <6\ell_1.
\end{equation}

From this we can deduce that
\[
\sum_{j>e}\ell_j  \leq2
\]
If the last displayed inequality is an equality then also
\[
(2\ell_1-6)\sum_{j=2}^e \ell_j<2\ell_1.
\]
In this case we claim that  $\sum_{j=2}^e \ell_j=0$. To see this, denote the sum by $\Sigma$. Then $\ell_1\left(\Sigma -1\right)<3\Sigma$ so that $\left(\Sigma-1\right)<\tfrac{3}{\ell_1}\Sigma$. But this is impossible if $\Sigma >0$ because $\ell_1\geq4$ and $\Sigma\geq 4$.

By a very similar argument, suppose that $\sum_{j>e}\ell_j =1$. This means that there is a unique component $A_1$ apart from those of types $F_4$, $E_7$, $E_8$. Moreover $\left(\Sigma-2\right)<\tfrac{3}{\ell_1}\Sigma \leq \tfrac{3}{4}\Sigma$, whence $\Sigma<8$. And if $\ell_1 \geq 7$ then $\tfrac{4}{7}\Sigma<2$, which is once again again impossible unless $\Sigma =0$. The conclusion is that if we have a  component of type $A_1$ and two exceptional components then we must have $\ell_1=4$. Since $\ell_1$ could have been chosen to be any of the exceptional Lie ranks, all of the exceptional components must be $F_4$.

We can argue similarly if all components are exceptional. In this case the main inequality (\ref{ineq5}) reads
\begin{equation}
\label{ineq8}
\tfrac{1}{2}(\ell_0^2+9\ell_0)+(\ell_1-3)\left(\sum_{j=2}^e \ell_j\right)+ \sum_{0\leq i < j, i\neq 1} \ell_i\ell_j   <3\ell_1
\end{equation}
and all of the $l_i$ are equal to $4$, $7$ or $8$. So if there are at least three components then $\sum_{j\geq 2} \ell_j\geq 8$ and we can deduce that $8(\ell_1-3)+16 <3\ell_1$, i.e., $5\ell_1<8$, a contradiction. Similarly if there are two components  and the second is \emph{not} $F_4$ we obtain $7(\ell_1-3)<3\ell_1$, whence $\ell_1\leq 5$ and the first component is $F_4$. So either way one of the two components must be $F_4$.

To summarize so far, we've shown that one of the following must hold for the semisimple part, that is the \emph{Levi factor} $L$ of $V_1$:
\begin{itemize}
\item $L=\mathfrak{g}_1$
\item $L=F_4\oplus\mathfrak{g}_1$
\item $L=A_1 \oplus \mathfrak{g}_1$
\item $L=A_1 \oplus F_4\oplus F_4$
\item $L=A_1\oplus A_1 \oplus \mathfrak{g}_1$
\item $L=A_2 \oplus \mathfrak{g}_1$
\end{itemize}
and in all cases $\mathfrak{g}_1$ is one of $F_4$, $E_7$ or $E_8$.

Now let's assume that there is no exceptional component of type $E_8$. Then $e_i\leq 4 \ell_i$ and $\sum_{i=1}^e\ell_i \leq 11$. Going back to (\ref{ineq6}) we obtain
\begin{eqnarray*}
&&2\ell_0^2+4\sum_{0\leq i<j} \ell_i\ell_j+2\sum_{i\geq 0}\ell_i+1\leq\sum_{i=1}^e e_i\leq 4\sum_{i=1}^e\ell_i \\
\Rightarrow&&(1+\ell_0^2+\ell_0)+2\sum_{0\leq i<j} \ell_i\ell_j +\sum_{i>e}\ell_i<\sum_{i=1}^e\ell_i\leq 11\\
\Rightarrow&&(\ell_0^2+\ell_0)+2\ell_0\sum_{1\leq i} \ell_i +2\ell_1\left(\sum_{i\geq 2} \ell_i\right) +\sum_{i>e}\ell_i<10.
\end{eqnarray*}
Therefore $\ell_1\sum_{i\geq 2} \ell_i\leq 4$, which can only happen if $\ell_1=4$ and $\sum_{i\geq 2}\ell_i = 1$, or if $\sum_{i\geq 2}\ell_i=0$. The latter equation means that $L$ is  simple. The first conditions mean that the first exceptional component is $F_4$, it is the only exceptional component, and if there are nonexceptional components they must comprise a single $A_1$. 

In the simple case (not $E_8$) we have $(\ell_0^2+\ell_0)+2\ell_0\ell_1 <10$ and because $\ell_1\geq 4$ then $\ell_0=0$. Observe, too, that if $V_1=E_8$ then $\ell =\tilde{c}=c$, an impossibility because we are assuming that $\tilde{c}>\ell$. If $E_8$ is the only component then the very first inequality $2\ell^2+3\ell+1\leq \dim V_1=\ell_0+248$ together with $\ell =\ell_0+8$ readily implies that $\ell_0\leq2$.

This allows us to refine the list of possibilities for $V_1$:
\begin{itemize}
\item $V_1=F_4$
\item $V_1=E_7$
\item $V_1=\CC^k\oplus E_8$, ($1\leq k \leq 2$)
\item $L=F_4\oplus E_8$
\item $L=A_1 \oplus F_4$
\item $L=A_1 \oplus E_8$
\item $L=A_1\oplus A_1\oplus E_8$
\item $L=A_2 \oplus E_8$
\end{itemize}

Here's another trick. We have $\dim V_1 = \tfrac{s(s-1)}{2}$ for a positive integer $s$. This eliminates all possibilities when $L$ is simple. Now we are obliged to look more closely at $\ell_0$. In the absence of an $E_8$ component the possibilities are $L=A_1 \oplus F_4$, so $\ell_1=4$, $\ell_2=1$ and the last displayed inequality implies that $(\ell_0^2+\ell_0)+10\ell_0+8+1 <10$, in which case $\ell_0=0$ and $\dim V_1=3+52=\tfrac{s(s-1)}{2}$ with $s=11$. Then $\tilde{c} = \tfrac{11}{2}$, $\ell=5$. But we are assuming that $\tilde{c}-\ell \geq 1$, a contradiction. Now we're reduced to the following possibilities with an $E_8$ component:
\begin{itemize}
\item $L=F_4\oplus E_8$
\item $L=A_1\oplus E_8$
\item $L=A_1\oplus A_1\oplus E_8$
\item $L=A_2\oplus E_8$.
\end{itemize}

In the first case we may apply (\ref{ineq8}). We have $\ell_1=4$, $\ell_2=8$, so
$\tfrac{1}{2}(\ell_0^2+9\ell_0)+8+8\ell_0<12$ is enough to force $\ell_0=0$. Therefore $V_1=F_4 \oplus E_8$ has dimension $52+248=300$, so $s=25$ and $c=\tilde{c} = \tfrac{25}{2}$, $\ell=12$. Once again this is outside of the scope of the case under consideration, so this case does not occur.

In the second case we utilize (\ref{ineq6}) together with $\ell_1 = 8$, $\ell_2=1$, $e_1=112$ to find that $2\ell_0^2+38\ell_0\leq 61$ and thus $\ell_0\leq 1$.
Then $\dim V_1 =251$ or $252$ and neither integer has the required form $\tfrac{s(s-1)}{2}$. So this case does not occur.

The fourth case is similar, except that $\ell_2=2$. Just as before this leads to $\ell_0=0$, so $\dim V_1=256$, which does not conform to $\tfrac{s(s{-}1)}{2}$, so this case does not occur.

For the third and final case we proceed similarly, but now with $\ell_1=8$, $\ell_2=\ell_3=1$. As before this leads to $\ell_0=0$, $\dim V_1=254$ which once again is not of the form $\tfrac{s(s-1)}{2}$. This completes the proof of Proposition \ref{propa}.
\end{proof}

Our next goal is the proof of
\begin{prop}
  \label{propc}
  Let $(m,x,y)$ be as in Theorem \ref{thmquad1}. Then Corollary \ref{corc/2}(c) cannot hold.
\end{prop}

\begin{proof}
  We are assuming here that $c=\tilde{c}=\ell=\tfrac{s}{2}$ so by a Theorem of Dong-Mason (cf. Theorem \ref{thmctildec}), we know that $V \cong V_L$ is a lattice theory for some even lattice $L$. Let the root system of $L$ be denoted by $L_2$. The $q$-character of $V$ is then the quotient of modular forms
\[
f_0(\tau)=\frac{\theta_L(\tau)}{\eta(\tau)^{\ell}}=q^{-\ell/24}(1+(\ell+\abs{L_2})q+\cdots ).
\]

We have  $m=\tfrac{s(s-1)}{2}=2\ell^2-\ell$. Therefore $\abs{L_2}=2\ell^2-2\ell$. Because $L$ is an even lattice then its root system is the direct sum of simple root systems of types ADE. Let $\mathfrak{g}_i$ ($1 \leq i \leq N$) be the nonabelian simple Lie algebra components of $V_1$, and let $\Phi_i$ be the root system of $\mathfrak{g}_i$, say of rank $\ell_i$. Then we have
 \begin{align*}
2\left(\sum_{r=1}^N \ell_r\right)^2-2\sum_{r=1}^N\ell_r=&2\ell^2-2\ell=\sum_{i=1}^N \abs{\Phi_i}=\sum_{i} \abs{\Phi_i}+\sum_{j} \abs{\Phi_j}+\sum_{k} \abs{\Phi_k}\\
=&\sum_{i} (\ell_i^2+\ell_i)+\sum_{j} (2\ell_j^2-2\ell_j)+\sum_{k} (2\ell_k^2-2\ell_k+f_k)\\
=&\sum_{i} (-\ell_i^2+3\ell_i) +\sum_{k} f_k-2\sum_{r=1}^N \ell_r +2\sum_{r=1}^N \ell_r^2
\end{align*}
where $\abs{\Phi_i}=\ell_i^2+\ell_i$, $2\ell_i^2-2\ell_i$, $72$, $126$, $240$ for $\mathfrak{g}_i$ of type $A_{\ell_i}$, $D_{\ell_i}$, $E_6$, $E_7$, $E_8$ respectively, and where we use $i, j, k$ to index the occurring root systems of type $A, D, E$ respectively. We also have $f_k\df 18, 42,128$ for $E_6, E_7, E_8$ respectively. Note that $f_k<16\ell_k$.

This begins to look like what we faced in the course of the proof of Proposition \ref{propa}, where we first made a relatively naive estimate, then backtracked. The previous displayed equality yields
\[
4\sum_{1\leq r< s\leq N} \ell_r\ell_s=-2\ell+\sum_{i} (-\ell_i^2+3\ell_i) +\sum_{k} f_k
\]
Therefore, \emph{if there is} a component of type $E$ then for some $\ell_k$, say with $k =N$, we have  $\ell_N \geq 6$ and
\begin{align*}
&4\ell_N\sum_{r<N} \ell_r <\sum_{i} (-\ell_i^2+3\ell_i) +\sum_{k} f_k\\
\Rightarrow&(4\ell_N-3)\sum_{i} \ell_i+\sum_{k<N} (4\ell_N\ell_k-f_k)< f_N-\sum_{i} \ell_i^2.
\end{align*}
Because the sum over $k <N$ is nonnegative we then have
\[
(4\ell_N-3)\sum_{i} \ell_i+\sum_i \ell_i^2< f_N
\]
and so
\[
  19\sum_{i} \ell_i+\left(\sum_i \ell_i\right)^2< f_N.
\]
Now we find that if $f_N=128$, then $\sum_i\ell_i{\leq}5$; if $f_N=42$ then $\sum_i\ell_i \leq 1$; and if $f_N=18$ then $\sum_i\ell_i=0$.

If $\sum_i\ell_i=0$ then there are no type $A$ components, and we then have
\[
4\sum_{1\leq r< s\leq N} \ell_r\ell_s<\sum_{k} f_k \quad \Rightarrow \quad 4f_N\sum_{1\leq r< N} \ell_r<\sum_{k} f_k,
\]
from which it follows easily that there is at most one nonzero type $E$ component. And if there are any of type $D$, then $4f_N\sum_j \ell_j< f_N$, contradiction. So we are reduced to the possibility that there is a single component, of type $E$. Then $240=2\ell^2-2\ell$, an impossibility. This shows that some $\ell_i>0$.

We have therefore shown that \emph{if} there is a type $E$ component, then there must also be at least one type $A$ component. Suppose there is a unique type $E$ component. Then
\[
4\ell_N\sum_{i} \ell_i<\sum_{i} (-\ell_i^2+3\ell_i) +f_N,
\]
and so $\sum_{i} (\ell_i^2+21\ell_i)< f_N$, forcing $f_N=42$ or $128$. If $f_N=42$ then necessarily $\{\ell_i\}=\{1\}$, i.e., there is a unique type $A$ component and it is $A_1$. Then $L_2=A_1\oplus E_7$ and $\abs{L_2}=128 \neq 2\ell^2-2\ell$. Suppose that $f_N = 128$. Then $\sum_i \ell_i \leq 4$ and $\abs{L_2}\in\{240, 242, 244, 246, 248, 250, 254, 260\}$, none of which are $2\ell^2{-}2\ell$. This shows that there are at least two type $E$ components. In this case we have
\[
  4\ell_N\sum_{i} \ell_i+4\sum_{k<k'}  \ell_k\ell_{k'}< \sum_{i} (-\ell_i^2+3\ell_i) +\sum_k f_k
\]
and so
\[
  \sum_{i} (\ell_i^2+21\ell_i)+4\sum_{k<k'} \ell_k\ell_{k'} < \sum_k f_k,
\]
and since each $\ell_k \geq 6$, and $4\ell_k\ell_{k'}>f_k$, then there can be no more than two type $E$ components. Moreover they are both of type $E_8$, whence
$\sum_{i} (\ell_i^2+21\ell_i)=0$. Hence $L_2=E_8 \oplus E_8$, $\abs{L_2}=480$, and $\ell =16$. Now $L = E_8 \oplus E_8$, in which case $V{=}V_L$ is \emph{holomorphic}, a contradiction.

We have finally shown that $V_1$ has  \emph{no} components of type $E$. So we have
 \begin{align*}
   2\left(\sum_{r=1}^N \ell_r\right)^2-2\sum_{r=1}^N\ell_r=&2\ell^2-2\ell\\
   =&\sum_{i} \abs{\Phi_i}+\sum_{j} \abs{\Phi_j}\\
=&\sum_{i} (\ell_i^2+\ell_i)+\sum_{j} (2\ell_j^2-2\ell_j)\\
=&\sum_{i} (-\ell_i^2+3\ell_i) -2\sum_{r=1}^N \ell_r +2\sum_{r=1}^N \ell_r^2,
\end{align*}
so that
\[
2\left(\sum_{r=1}^N \ell_r\right)^2 = \sum_{i} (-\ell_i^2+3\ell_i) +2\sum_{r=1}^N \ell_r^2
\]
and $4\sum_{1\leq r<s} \ell_r\ell_s=\sum_{i} (-\ell_i^2+3\ell_i)$. If there are no type $A$ component then the right hand side of this inequality vanishes, whence so does the left hand side, meaning that there is a unique component, and it has type $D$. Here, then, we have $V \cong D_{\ell,1}$. However, this VOA has $4$ simple modules if $l \geq 5$ and $2$ if $l = 4$. Thus this example does not occur. 
Suppose there are some type $A$ components. Then the last displayed inequality implies that such a type $A$ component is unique, call it $\mathfrak{g}_1$. Then
\[
0=4\ell_1\sum_{2\leq r} \ell_r=-\ell_1^2+3\ell_1,
\]
so $\ell_1=3$ and $V=V_L\cong A_{3,1}$. Once again, this VOA has $4$ simple modules so it does not occur. This completes the proof of the Proposition.
\end{proof}

The final case is:
\begin{prop}
  \label{propb}
  Assume that $(m,x, y)$ is as in Theorem \ref{thmquad1} and that  Corollary \ref{corc/2}(b) holds. Then $V \cong B_{\ell,1}$, $A_{1,2}$ or $\Vir(c_{3,4})$.
\end{prop}
\begin{proof}
 In this case we have $c=\tilde{c}=\tfrac{s}{2}$, $\ell = \tilde{c}-\tfrac{1}{2}=\tfrac{s{-}1}{2}$, $m=\tfrac{s(s{-}1)}{2}=2\ell^2+\ell$. 

Now we have seen that the $q$-character of $V$ (and that of its simple modules, too) is uniquely determined by this data. It follows that the $q$-character of $V$ is equal to that of one of the VOAs in the statement of the Proposition.

Suppose first that $\ell=0$. Then $s=1, \tilde{c}=c=\tfrac{1}{2}$, and by \cite{MasonLattice} Theorem 8, it follows that $V$ contains the Virasoro VOA $\Vir(c_{3,4})$ as a subVOA. However from the last paragraph $V$ has the same $q$-character as this Virasoro VOA and therefore they are equal. This proves the Proposition if $\ell=0$. Thus from now on we may, and shall, assume that $V_1\neq 0$. We would like to then show that $V_1$ is isomorphic to $B_{\ell}$, or $A_1$.

Suppose that  $\ell = 1$.  Then $m=3$ and $V_1 \cong A_1$. By \cite{DMIntegrability} the subVOA $U\df \langle V_1\rangle$ generated by $V_1$ is isomorphic to an affine algebra $A_{1, k}$ of some positive integral level $k$. Now we can use the \emph{majorizing Theorem} in Appendix \ref{a:affinealgebras} to see that because the $q$-character of $V$ is the same as that for $A_{1,2}$ by the first paragraph, then  $k \leq 2$, and if $k=2$ then  $U=V\cong A_{1,2}$. Suppose that $k=1$. Then the commutant $C$ of $U$ has central charge $\tfrac{1}{2}$. Now consider $U\otimes C$: it is a subVOA of $V$ and from what we have said it majorizes $A_{1, 1}\otimes \Vir(c_{3, 4})$ or is equal to it. But this latter VOA itself majorizes $A_{1,2}$ as one sees by a direct check of $q$-expansions, and this shows that the case $k=1$ does \emph{not} occur.

Now suppose that $\ell \geq 2$. By the first paragraph $V$ has the same $q$-character as $B_{\ell,1}$. If we can show that $V_1\cong B_{\ell}$ then the same arguments used in the previous paragraph show that $V\cong B_{\ell,1}$, and the Proposition will be proved.

 We can attack this much as we did in the proofs of Propositions \ref{propa} and \ref{propc}. Let $V_1$ have Levi decomposition \eqref{Levi}. Then $\ell=\ell_0+\sum_i \ell_i$. Let $\Phi_i$ be the root system of $\mathfrak{g}_i$. Then 
 \begin{eqnarray}
   \label{ineqx}
&&2\ell^2+\ell=2\left(\sum_{i\geq 0} \ell_i\right)^2+\left(\sum_{i\geq 0} \ell_i\right)=\dim V_1=\ell_0+\sum_{i\geq 1}(\ell_i+\abs{\Phi_i})\notag\\
\Rightarrow&&2\left(\sum_{i\geq 0} \ell_i\right)^2=\sum_{i\geq 1}\abs{\Phi_i}\notag\\
\Rightarrow&&2\ell_0^2+4\ell_0\sum_{i\geq 1}\ell_i +4\sum_{1\leq i < j}\ell_i\ell_j=\sum_{i\geq1} (\abs{\Phi_i}-2\ell_i^2).
\end{eqnarray}
Now $\abs{\Phi_i}-2\ell_i^2=\ell_i-\ell_i^2$; $0$; $-2\ell_i$; $6$; $20$; $6$; $35$; $120$ for types  $A_{\ell_i}$;  $B_{\ell_i}$ or $C_{\ell_i}$; $D_{\ell_i}$; $G_2$; $F_4$; $E_6$, $E_7$, $E_8$ respectively.

Suppose that the left hand side of \eqref{ineqx} is $0$. Then $\ell_0=0$, $V_1$ has a unique component, and it has type $B$ or $C$. If the type is $B$ then $V\cong B_{\ell,1}$ as we have already explained, so we are done in this case. If the type is \emph{not} $B$ then $V_1\cong C_{\ell}$ with $\ell \geq 3$. By Theorem 1.1 of \cite{DMIntegrability} it follows that the subVOA $U \df \langle V_1\rangle$ generated by $V_1$ is isomorphic to $C_{\ell,k}$ for some positive integral level $k$. Since $U$ is generated by weight $1$ states, a consideration of the conformal subVOA  $U \otimes C$, where $C$ is the commutant of $U$, shows that  $\dim U_2\leq \dim V_2$ and hence that $\dim (C_{l,k})_2 \leq \dim(B_{\ell,1})_2$. However this contradicts Theorem \ref{thmBC} in Appendix \ref{a:affinealgebras}. This proves Proposition \ref{propb} if the left hand side of \eqref{ineqx} is $0$.

This reduces us to consideration of the case that the left hand side of \eqref{ineqx} is \emph{positive}, so the right side is too. So there must be at least one  \emph{exceptional} component

Suppose there are $k$ components of type $E_8$, and $r$ exceptional components \emph{not} of type $E_8$. The right side of \eqref{ineqx} is at most $120k+35r$, whereas the left side of \eqref{ineqx} is at least $4(64{k\choose 2}+16kr+4{r\choose 2})$. Therefore
\begin{eqnarray*}
&& 16(k^2-k)+8kr+(r^2-r) \leq 15k+\tfrac{35}{8}r\\
\Rightarrow&&16k^2-31k + 8kr + (r^2-r) \leq \tfrac{35}{8}r
\end{eqnarray*}
It follows easily that $k \leq 1$, and if $k=1$ then $r^2-15 \leq -\tfrac{29}{8}r \Rightarrow r \leq 2$. Again with $k=1$ we can argue more precisely that if
the exceptional components are $\mathfrak{g}_1=E_8$, $\mathfrak{g}_2$, $\mathfrak{g}_3$ then
\[
4(8\ell_2+8\ell_3+\ell_2\ell_3) \leq 120 + (\abs{\Phi_2}-2\ell_2^2) +(\abs{\Phi_3}-2\ell_3^2)
\]
and the two terms on the right hand side are among $\{6, 20, 6, 35 \}$, and $\ell_2$, $\ell_3$ are each one of $\{2, 4, 6, 7\}$. We see that this can never hold.

This shows that $k = 0$, i.e., there are \emph{no} components of type $E_8$. Repeating the argument if there are $k'$ components of type $E_7$ and $r'$ other exceptional components, then 
\begin{eqnarray*}
&&4\left(49{k'\choose 2}+4{r'\choose 2}+14k'r'\right)\leq 35k'+20r'\\
\Rightarrow&&98k'(k'-1)+8r'(r'-1)+56k'r' \leq 35k'+20r'\\
\Rightarrow&&\tfrac{49}{2}k'^2+4r'^2+14k'r' \leq \tfrac{133}{4}k'+7r'\\
\Rightarrow&&\tfrac{49}{2}(k'-\tfrac{19}{4\cdot 7})^2 +4(r'-\tfrac{7}{8})^2+14k'r' \leq \tfrac{1}{2}\left(\tfrac{19}{4}\right)^2+\tfrac{49}{16}<\tfrac{361}{32}+3\tfrac{1}{16}<15.
\end{eqnarray*}
We readily deduce that at least one of $k'$ or $r'$ is $0$. Thus if there are any exceptional components then either there is an $E_7$ and no other exceptional component, or else there are no exceptional components of type $E_8$ or $E_7$. In the former case, if $V_1=E_7$ then the right side of \eqref{ineqx} is odd, while the left side is even, a contradiction. If there are no $E_8$, $E_7$ components, then as before we have in case there are $t$ exceptional components that $8t(t-1)\leq20t$, which implies $(t^2-4t)\leq 0$, so $t \leq 2$. But if $t=2$ we get equality, meaning two $F_4$ components and $l\ell_1 =\ell_2=4$, impossible. 

Thus $t = 1$, i.e., there is a unique exceptional component, and $2\ell_0^2+8\ell_0\leq 20$. Then $\ell_0 \leq 1$ and $\dim V_1=2\ell^2+\ell=14(15), 52(53), 78 (79)$ (parentheses denotes the case $\ell_0=1$), which can only occur when $\ell=6$, $\ell_0=0$ and $V_1=E_6$. Furthermore $c=\tilde{c}=\tfrac{13}{2}$ and
the commutant of $U\df \langle V_1\rangle$ is isomorphic to $\Vir_{c_{3, 4}}$. Note that $U \cong E_{6, k}$ for some positive integral $k$ by \cite{DMIntegrability}. But it now follows that $U$ has central charge $6$. Since $E_{6, k}$ has $c=\tfrac{78k}{k+12}$ we must have $k=1$.

Now $m=78=s(s-1)/2$, so $s=13$. Therefore $h_1= x+1=(s/16 - 1)+1 =13/16$ and $h_2= y+1 = 1/2$ and the conformal weights of $V$ are $\{0, 1/2, 13/16\}$. Now the conformal weights for the simple $E_{6, 1}$-modules are $\{0, 2/3\}$ while those for $\Vir(c_{3, 4})$ are $\{0, 1/2, 1/16\}$. Since $E_{6,1}\otimes \Vir(c_{3, 4})$ is a conformal subVOA of $V$, it is impossible to reconcile the conformal weights of the tensor product with those for $V$. Thus this case cannot occur. This finally completes the proof of Proposition \ref{propb}.
\end{proof}

With these Propositions in hand we have completed the proof of Theorem \ref{thmquad1}. There remain two outstanding cases, enumerated as (2) and (3) on page \pageref{examples1}. As noted, there are examples of VOAs with the relevant numerical data in both cases, namely $\Vir(c_{2, 7})$ and $A_{4, 1}$. In the first case we have $m=0$ and here we may appeal to the main result of \cite{ANS1} to immediately conclude that indeed $V\cong \Vir(c_{2, 7})$. 

For the sake of brevity we sketch how to prove nonexistence in example (3). First note that inasmuch as the data determines the character vector of $V$ it follows in particular that the character of $V$ coincides with that of $A_{4, 1}$. Now we may proceed much as in the proofs of the three Propositions, although here it is much easier because we already know that $m=24$. We have $c=\tilde{c}=4$ and $\ell\leq c$. If $\ell=c$ then $V$ is a lattice theory and as before we find that $V_1=A_4$ and then that $V\cong A_{4, 1}$. However this VOA has more than three simple modules so it cannot occur. If $\ell<c$ we obtain a contradiction as in the Propositions. Alternatively, we first identify the Lie algebra $V_1=A_4$ then conclude that  $\langle V_1\rangle \cong A_{4, k}$ for some integral level $k$. Now apply the majorization argument and knowledge of the character to get $V\cong A_{4, 1}$, and hence a contradiction as before.

This finally completes our proof of Theorem 1.

\section{The $U$-series}
\label{SU}

By their very definition, potential VOAs that belong to the $U$-series have three simple modules and survive all of the numerical tests that we have so far applied. From an arithmetic perspective they are exquisitely balanced.

In this Section we discuss further properties of these VOAs, especially the question of whether they actually \emph{exist}. We shall present some results that render it very likely that there are $15$ VOAs in the $U$-series. See Remark \ref{r:useries}. Two of these examples are well-known in the literature, namely $E_{8,2}$ and the Gerald H\"{o}hn's Baby Monster VOA $\VB^{\natural}_{(0)}$ \cite{Hoehn1}. The remaining examples come about by an application of the results of Gaberdiel, Hampapura and Mukhi \cite{GHM} and Lin \cite{Lin}. These works are applicable on the basis of an apparent and surprising connection between VOAs in the $U$-series and VOAs $X$ on the Schellekens list \cite{Schellekens} of holomorphic VOAs of central charge $c=24$. Indeed, we propose Hypothesis $S$ below, which is a natural assumption about glueing VOAs and which leads to the identification of the $U$-series VOAs with certain commutants of subalgebras for various choices of $X$.

\subsection{Connections with the Schellekens list}
Let us record some of the properties of a VOA $V$ that lies in the $U$-series:
\begin{enumerate}
\item[(i)] $V$ is strongly regular and has just $3$ simple modules $M_0=V$, $M_1$, $M_2$.
\item[(ii)] The $q$-characters $f_j(\tau)$ of the $M_j$ are each congruence modular functions of weight $0$ with nonnegative integral Fourier coefficients described explicitly in Figure \ref{f:fouriercoeffs}.
\item[(iii)] The character vector $F=(f_0, f_1, f_2)^T$  is a vector-valued modular form whose associated MLDE is monic with irreducible monodromy $\rho$.
\item[(iv)] There is an integer $p$ in the range $5 \leq p \leq 15$ such that the central charge $c$, the dimension $m$ of the Lie algebra on $V_1$, and the \emph{conformal weights} $h_j$ of the $M_j$ are as follows (cf.\ Figure \ref{f:full3}):
\begin{align*}
c&=p+\tfrac{17}{2},& m=&(15-p)(2p+17),& h_0=&0,& h_1=&\tfrac{3}{2},& h_2=&\tfrac{2p+1}{16}.
\end{align*}
The formula for $m$ derives from that for the elliptic surface (\ref{eq1}).
\item[(v)]  The $S$-matrix is
\begin{equation}
    \label{STmatrix}
\rho(S)=  \frac{1}{2}\left(\begin{matrix} 1 & 1 & \sqrt{2} \\1 & 1 & -\sqrt{2} \\ \sqrt{2} & -\sqrt{2} & 0 \end{matrix}\right) 
\end{equation}
 with lexicographic ordering. In particular the fusion rules for $V$ are the same as the Ising model $\Vir(c_{3, 4})$. Especially, it follows that $M_1$ has quantum dimension $1$ and is a \emph{simple current}.
\end{enumerate}

Now let $k$ be a nonnegative integer. We define a family of VOAs $V^{(k)}$ as follows:
\[
V^{(k)}\df \begin{cases}
\Vir(c_{3, 4}) & k=0,\\
A_{1,2} & k=1,\\
B_{k,1}&k \geq 2.
\end{cases}
\]
 As a reminder, from Table \ref{Table1} we see that, like VOAs in the $U$-series,  $V^{(k)}$ is a simple VOA with just three simple modules. Denote these
 by $V^{(k)}$, $M_1'$, $M_2'$, say with conformal weights $0$, $h_1'= \tfrac{1}{2}$ and $h_2'=\tfrac{2k{+}1}{16}$ respectively. The central charge of $V^{(k)}$ is equal to $c_k \df \tfrac{2k+1}{2}$.

Now choose any VOA in the $U$-series with parameter $p$ as before, and denote this VOA by $W^{(p)}$ and choose $k\df 15-p$, so that $0 \leq k \leq 10$. For this choice of $k$ the tensor product VOA
\begin{eqnarray*}
T^{k}\df W^{(p)} \otimes V^{(k)}
\end{eqnarray*}
is a simple VOA with central charge equal to $p + \tfrac{17}{2} + \tfrac{2k+1}{2} = 24$. Let us consider the $T^k$-module
\begin{equation}
\label{tridecomp}
  X\df (W^{(p)}\otimes V^{(k)}) \oplus (M_1\otimes M_1')\oplus (M_2\otimes M_2').
\end{equation}
Gerald H\"{o}hn calls this procedure \emph{glueing} $W^{(p)}$ and $V^{(k)}$. Each
$M_j \otimes M_j'$ is a simple module for $T^{k}$, $j=1, 2$. The next result is very useful.  
\begin{lem}
  \label{lemconfwts}
  The conformal weights of $M_j \otimes M_j'$ for $j=1$, $2$ are both equal to $2$. In particular the conformal grading on $X$ is \emph{integral}.
\end{lem}
\begin{proof}
  We have $h_1+h_1'=\tfrac{3}{2}+\tfrac{1}{2}=2$ and $h_2+h_2'=\tfrac{2p+1}{16}{+}\tfrac{2k+1}{16}{=}2$. The Lemma follows.
\end{proof}

\begin{cor}
\label{corX1}
The conformal weight $1$ piece $X_1$ of $X$ satisfies
\[
X_1=T^k_1= W^{(p)}_1\oplus V^{(k)}_1.
\]
$\hfill\Box$
\end{cor}

Let $\chi \df \chi_X=\Tr_Xq^{L(0)-1}$ be the $q$-character of $X$. It follows from Lemma \ref{lemconfwts} that
\begin{equation}
  \label{qchar}
\chi \in  q^{-1}\ZZ[[q]].
\end{equation}

\begin{lem}
  \label{lem48}
  $\chi$ is the modular function of level $1$ and weight $0$ given by
\[
\chi = J(q)+48k.
\]
where $J(q) \df q^{-1}+196884q+\cdots$ is the absolute modular invariant with constant term $0$.
\end{lem}
\begin{proof}
  After \eqref{qchar}, $\chi$ is invariant under the $T$-action $\tau\mapsto \tau{+}1$. So to prove that $\chi$ is modular of level $1$ it suffices to establish invariance under the action of $S$. This will follow directly by a formal calculation based on the nature of $\rho(S)$ \eqref{STmatrix}. For the VOAs $W^{(p)}$ and $V^{(k)}$ have identical $S$-matrices. Therefore if we formally let $\{e_j\}$, $\{f_j\}$  ($j=1, 2, 3$) index bases with respect to which the two $S$-matrices are written then
\begin{align*}
  S \otimes S: \sum_i e_i\otimes f_i \mapsto& \tfrac{1}{4}\left\{2(e_1\otimes f_1+e_2\otimes f_2+2e_3\otimes f_3)+2(e_1 \otimes f_1+e_2 \otimes f_2) \right\}\\
  =&\sum_i e_i \otimes f_i,
\end{align*}
which is the required $S$-invariance.

It is well-known \cite{Zhu} that the $q$-characters of simple modules for strongly regular VOAs are holomorphic in the complex upper-half plane. Therefore $\chi$ is modular of level $1$ with a simple pole at $\infty$ and no other poles, and leading coefficient $1$. It follows that $\chi=J(\tau){+}\kappa$ for a constant $\kappa$.

To compute the constant $\kappa$, which is equal to $\dim X_1$, use Corollary \ref{corX1} to see that
\begin{align*}
  \dim X_1=&\dim W^{(p)}_1{+}\dim V^{(k)}_1\\
          =&m+\dim  B_k\\
  =&(15-p)(2p+17)+(2k^2+k)\\
  =&48k.
\end{align*}
This completes the proof of the Lemma.
\end{proof}

Lemma \ref{lem48} naturally suggests

\medskip\noindent
\textbf{Hypothesis S}: $X$ carries the structure of a holomorphic VOA containing $T^k$ as a subVOA. $X$ is therefore a holomorphic VOA of central charge $24$, that is, it is on the Schellekens list.
\medskip

Hypothesis S is completely analogous to H\"{o}hn's Vermutung 3.2.1 in \cite{Hoehn1}. It suggests where we should look to find VOAs in the $U$-series. We consider  this option in the next Subsection.

\subsection{Existence of $U$-series VOAs}

Throughout this Subsection, and for the sake of comparison, we generally use notation similar to that of the previous Subsection. In particular, we now fix $X$ to be a VOA on the Schellekens list. For a recent survey on the status of the VOAs in the Schellekens list, we refer the reader to \cite{LS}. In particular, the Schellekens list VOAs intervening in Table 3 exist and they are unique. Let $V^{(k)}{\subseteq}X$ be a subVOA isomorphic to an affine algebra as in the previous Subsection such that the weight 1 piece $V^{(k)}_1$ is a simple Lie algebra component of $X_1$ isomorphic to either $A_{1, 2}$ or $B_{k, 1}$.

This assumption involves some \emph{exclusions}. First, the case $k = 0$ and $V^k=\Vir(c_{3, 4})$ does not occur. This case is somewhat exceptional and was, in any case, handled by H\"{o}hn. Secondly, the cases $k=7, 9, 10$ do not occur either, but for a different reason. Namely because there is no $X$ with such a subalgebra (cf. Table \ref{TableS}).

\begin{lem}
  We have $V^{(k)}=C(C(V^{(k)}))$, i.e., $V^{(k)}$ coincides with its double commutant in $X$.
\end{lem}
\begin{proof}
  Let $D\df C(C(V^{(k)}))$ be the double commutant in question. It is a subVOA of $X$ that contains $V^{(k)}$, and we have $D_1 = V^{(k)}_1$. Furthermore $D$ and $V^{(k)}$ share the \emph{same Virasoro element}. On the other hand, of the three simple modules for $V^{(k)}$, the adjoint module is the only one that has integral conformal grading. Now the equality $D = V^{(k)}$ follows.
\end{proof}

Continuing earlier notation, we set
\begin{eqnarray}
  \label{Udef}
W^{(p)}\df C(V^{(k)}).
\end{eqnarray}
Note the difference, however. Our earlier $W^{(p)}$ was the hypothesized $U$-series VOA, whereas now there is no question about its existence. What \emph{is} in doubt is whether $W^{(p)}$ as defined in \eqref{Udef} is in the $U$-series. Basically, this comes down to the question, does $W^{(p)}$ have exactly three simple modules? In the next few paragraphs we will state and prove what we know about this. We will show on the basis of the results of Gaberdiel, Hampapura and Mukhi \cite{GHM} that $W^{(p)}$ is indeed in the $U$-series. This work, while undoubtedly correct, was not developed on an axiomatic basis.

We will need to use the following standard conjecture concerning commutants in a strongly regular VOA. In the present context it says

\medskip\noindent
\textbf{Hypothesis C}: $W^{(p)}$ is a strongly regular VOA.
\medskip

Proceeding on the basis of this Hypothesis C (actually, we only need $W^{(p)}$ to be rational and $C_2$-cofinite) we first prove 

\begin{lem}
  The decomposition \eqref{tridecomp} holds, where $M_1'$, $M_2'$ are the nonadjoint simple modules for $V^{(k)}$, labelled according to the $S$-matrix \eqref{STmatrix} and $M_1$, $M_2$ are simple modules for $W^{(p)}$.
\end{lem}
\begin{proof} Because we are assuming Hypothesis C, this follows from results of Xingjun Lin \cite{Lin}, cf. (1.1) of that paper.
\end{proof}

Furthermore we have

\begin{prop}
  The following hold:
\begin{enumerate}
\item[(a)] $M_1$ and $M_1'$ are simple currents for $W^{(p)}$ and $V^{(k)}$ respectively.
\item[(b)] $V^{(k)}\oplus M_1'$ and $W^{(p)} \oplus M_1$ are both rational super VOAs.
\item[(c)] $(W^{(p)} \otimes V^{(k)})\oplus (M_1 \otimes M_1')$ is a conformal subVOA of $X$.
\end{enumerate}
\end{prop}
\begin{proof}
  That $M_1'$ is a simple current for $V^{(k)}$ was already pointed out following \eqref{STmatrix}. Indeed, these simple currents for affine algebras are well-known. That $V^{(k)} \oplus M_1'$ is a rational super VOA is proved in \cite{DLMCurrents}, Examples 5.11 and 5.12. As for $M_1$, that it is a simple current for $W^{(p)}$ follows from the existence of $M_1'$ and the duality between module subcategories proved by Lin \cite{Lin}. Now it follows that the subspace $(W^{(p)}\otimes V^{(k)})\oplus (M_1\otimes M_1')\subseteq X$ is closed with respect to products, hence it is a subVOA. Therefore $V^{(k)}\oplus M_1'$ is itself a super VOA because we have already seen that $W^{(p)}\oplus M_1$ is. This completes the proof of the Lemma.
\end{proof}

Our main result is the following
\begin{thm}
  $W^{(p)}$ is a VOA in the $U$-series.
\end{thm}
\begin{proof}
  The main point in the proof is the work of Gaberdiel, Hampapura, and Mukhi \cite{GHM}. These authors consider the properties of commutants of affine algebras such as $V^{(k)}$ from the perspective of MLDEs. They are able to show, in the framework that we are working, that the commutant $W^{(p)}$, while not necessarily having exactly $3$ simple modules (which is what we need), at least satisfies $\dim \ch_{W^{(p)}} =3$. That is, the space of $q$-characters for the simple $W^{(p)}$-modules is $3$-dimensional. Note that we know the $q$-characters of the simple $W^{(p)}$-modules that are contained in $X$ and that they furnish an irreducible representation $\rho$ of $\Gamma$. As a check, \cite{GHM} describes the MLDE satisfied by these characters and one can check from their Tables that the conformal weights of these modules are precisely those of the $U$-series that we have already calculated from a completely different perspective.

From these comments it follows that we can organize the simple modules for $W^{(p)}$ into sets of $3$ so that the corresponding matrix representation of $\Gamma$ on the $q$-characters looks like
\begin{eqnarray*}
\left(\begin{array}{ccc}\rho & 0 & 0 \\0 & \ddots & 0 \\0 & 0 & \rho\end{array}\right)
\end{eqnarray*}
In particular, the $S$-matrix has a similar block diagonal decomposition. However, at least in its action on the $1$-point functions for $W^{(p)}$ (genus $1$ conformal block), the $S$-matrix has first row with only nonzero entries. It follows immediately from the displayed block diagonal matrix that this can only happen if it is a $1 \times 1$ block matrix. That is, there are only $3$ simple modules for $W^{(p)}$. This completes the proof of the Theorem.
\end{proof}

\begin{ex}
  $k=8$. We see from Table \ref{TableS} that in this case $W^{(p)}_1=E_8$ and  $W^{(p)}=E_{8,2}$. This equality holds because $E_{8,2}$ itself has only three simple modules. Thus this $U$-series VOA is well-known, as is its simple current $M_1$ and the super VOA $E_{8,2}\oplus M_1$ (cf. \cite{DLMCurrents}). This case was first handled by Gerald H\"{o}hn \cite{Hoehn1}.

We obtain $14$  different VOAs in the $U$-series, including the know $E_{8, 2}$, as we see from Table \ref{TableS}. The other $13$ are probably new.

In the exceptional case when $k=0$ we have $X_1=0$, so it was not considered in \cite{GHM}. In any case H\"{o}hn already proved, under the natural assumption that $X=V^{\natural}$ is the Frenkel-Lepowsky-Meurman Moonshine module, that the commutant of $V^{(0)}$ is the Baby Monster VOA $\VB^{\natural}_{(0)}$. So this example also falls into the $U$-series.
\end{ex}

\begin{table}
\begin{tabular}{|c|c|c|}   \hline
$\#$ & $X_1$&k \\ \hline
5&$(A_{1,2}^{15})\oplus A_{1,2}$&$1$\\
7&$(A_{3,4}^3)A_{1,2}$&$1$\\
8&$(A_{5,6}C_{2,3})\oplus A_{1,2}$&$1$\\
10&$(D_{5, 8}) \oplus A_{1,2}$&$1$    \\ \hline
25&$(D_{4, 2}^2B_{2, 1}^3)\oplus B_{2, 1}$&$2$\\
26&$(A_{5, 2}^2A_{2,1}^2)\oplus B_{2,1}$& $2$\\
28&$(E_{6,4}A_{2,1})\oplus B_{2,1}$&$2$\\ \hline
39&$(D_{6,2}C_{4,1}B_{3,1})\oplus B_{3,1}$&$3$\\
40&$(A_{9,2}A_{4,1})\oplus B_{3,1}$&$3$\\ \hline
47&$(D_{8,2}B_{4,1})\oplus B_{4,1}$&$4$\\
48&$(C_{6,1}^2)\oplus B_{4,1}$&$4$\\\hline
53&$(E_{7,2}F_{4,1})\oplus B_{5,1}$&$5$\\ \hline
56&$(C_{10,1})\oplus B_{6,1}$&$6$\\ \hline
62&$(E_{8,2})\oplus B_{8,1}$&$8$\\ \hline
\end{tabular}
\caption{VOAs on the Schellekens list with $X_1$ having a summand $B_{k, 1}$ or $A_{1, 2}$. Columns give the Schellekens list number, the structure of the Lie algebra $X_1$ with levels, and the $k$-value.}\label{TableS}
\end{table}

\appendix
\section{Primes in progressions}
\label{a:bertrand}

In Section \ref{s:fibers} we cut down the possible character vectors for VOAs occuring in Theorem \ref{thmmain} by making use of hypergeometric formulas for the character vector. To prove nonintegrality of the vectors \emph{not} contributing to Theorem \ref{thmmain}, we produced nontrivial denominators in all but finitely many cases. Our argument relies on the existence of primes in progressions that lie in specific intervals. In this short appendix we explain how to use effective versions of the prime number theorem for primes in arithmetic progressions to prove what we need. These sorts of results, which go back to Bertrand's postulate that there is always a prime between $x$ and $2x$, are well-known to analytic number theorists. A recent paper \cite{BMOR}, which establishes effective versions of prime number theorems for arithmetic progressions, enables us to get the precise results necessary for our application to the problem of classifying VOAs as in Theorem \ref{thmmain}. To treat solutions $(m,x,y)$ to equation \eqref{eq1} with $y = a/5$ where $a$ is an integer coprime to $5$, we make use in Section \ref{s:fibers} of the following result:
\begin{thm}
  \label{t:bertrand5}
If $X > 6496$ then the interval $[X,\frac{28}{27}X]$ contains at least one prime from each congruence class $a\pmod{30}$ with $\gcd(a,30) = 1$.
\end{thm}
\begin{proof}
  Let $\pi(X;q,a)$ denote the prime counting function for primes $p \equiv a\pmod{q}$. By Theorem 1.3 of \cite{BMOR}
  \[
  \abs{\pi(X;q,a) - \frac{\Li(X)}{\phi(q)}}<  c_\pi(q) \frac{X}{(\log X)^2}
\]
for all $X \geq x_{\pi}(q)$ for explicit constants $c_\pi(q)$ and $x_{\pi}(q)$ that are independent of $a$. We are interested in the function
\[
  F(X) = \pi(28X/27;q,a)-\pi(X;q,a).
\]
We must prove that there exists an $N$ such that $F(X) \geq 1$ for all $x > N$. Notice that if $X \geq x_\pi(q)$,
\begin{align*}
  \pi(28X/27;q,a) & > \frac{\Li(28X/27)}{\phi(q)}-\frac{28c_\pi(q)}{27} \frac{X}{(\log X+\log(28/27))^2}\\
-\pi(X;q,a) &> -\frac{\Li(X)}{\phi(q)} - c_\pi(q) \frac{X}{(\log X)^2}
\end{align*}
Therefore for $X \geq x_{\pi(q)}$,
\[
\abs{F(X)} > \frac{\Li(28X/27)-\Li(X)}{\phi(q)}-c_\pi(q)X\left(\frac{28}{27} \frac{1}{(\log X+\log(28/27))^2} -\frac{1}{(\log X)^2}\right)
\]
Taking $q = 30$, the paper \cite{BMOR} gives $x_\pi(30) = 789693271$ and $c_\pi(30) = 0.0005661$. One sees that for $X \geq x_{\pi}(30)$, $\abs{F(X)}$ is much larger than $1$. To prove the theorem for $6496 < X < x_{\pi}(30)$ we used a computer to verify it in the remaining cases. 
\end{proof}

Other solutions to equation \eqref{eq1} can be treated in a similar manner, where the relevant moduli are $6\cdot 7 = 42$ (primitive fibers) and $6\cdot 16 = 96$ (imprimitive fibers with $y \neq -1/2$); both of these moduli are treated in \cite{BMOR}.

\section{Affine algebras}
\label{a:affinealgebras}

Let $\mathcal{G}$ be a finite-dimensional simple Lie algebra of type $A$, $B$, $C$, $D$, $E$, $F$ or $G$ and Lie rank $\ell$ (dimension of a Cartan subalgebra). In this Appendix we discuss some properties of the universal vertex algebra $V(\mathcal{G}, k)$ of level $k$ and its simple quotient VOA $\mathcal{G}_{\ell,k}$, which is often called a WZW model when $k$ is a positive integer. For convenience, the constructions of these VOAs will be recalled below. For additional background, see \cite{Kac}, \cite{LL}.

\subsection{Statement of the main results}

There are two main results that we  intend to prove in this Appendix, both having to do with the conformal grading of WZW models. The first one we call the \emph{majorization Theorem}. As a referee has pointed out, this result may not be new, but we are unaware of a good reference:
\begin{thm}[Majorization]
  \label{thmmajor}
  Fix the type $\mathcal{G}$ and  Lie rank $\ell$. Regarding $\mathcal{G}_{\ell,k}$ as a  linear space equipped with its conformal $\ZZ$-grading, there are \emph{surjective} $\ZZ$-graded morphisms 
\[
  \mathcal{G}_{\ell,k'} \longrightarrow   \mathcal{G}_{\ell,k}
\]
for all positive integral $k'\geq k$.
\end{thm}

The second result is more specialized:
\begin{thm}
  \label{thmBC}
  For all positive integers $k, \ell$, we have 
\[
\dim (C_{\ell, k})_2 \geq \dim (B_{\ell,1})_2
\]
with equality only if $\ell=2$.
\end{thm}

\begin{rmk}
  \label{rmkBC}
  The proof will show that
\[
\dim (C_{\ell, k})_2 -\dim (B_{\ell,1})_2 \geq  2\ell-4.
\]
\end{rmk}

\subsection{The universal affine VOA $V(\mathcal{G},k)$}\label{SSAppB2}
Let $\mathcal{G}$ be a finite-dimensional nonabelian simple Lie algebra with Killing form $\langle ,\rangle$. The affine algebra associated to $\mathcal{G}$ is the Lie algebra defined by 
\[
\mathcal{G} \otimes \CC[t, t^{-1}] \oplus \CC K
\]
where $K$ is a central element and the  nontrivial brackets are
\begin{eqnarray*}
[a\otimes t^m, b\otimes t^n]\df [a, b]\otimes t^{m+n}+m\delta_{m+n, 0}\langle a, b \rangle K
\end{eqnarray*}
for $a,b \in \mathcal{G}$. There is a natural triangular decomposition
\[
\hat{\mathcal{G}} \df \hat{\mathcal{G}}^+ \oplus \hat{\mathcal{G}}_0 \oplus \hat{\mathcal{G}}^- 
\]
with
\begin{align*}
 \hat{\mathcal{G}}^{\pm} \df& \mathcal{G} \otimes t^{\pm 1}\CC[t^{\pm 1}]\\
 \hat{\mathcal{G}}_0 \df& \mathcal{G} \otimes t^0 \oplus \CC K \cong \mathcal{G} \oplus \CC.
\end{align*}
$\hat{\mathcal{G}}$ is also naturally $\ZZ$-graded by:
\[
\hat{\mathcal{G}}=   \bigoplus_{n\in\mathbf{Z}} \hat{\mathcal{G}}_n,\quad \hat{\mathcal{G}}_n \df \mathcal{G} \otimes t^{-n} \quad (n \neq 0), 
\]
so that $[\hat{\mathcal{G}}_m, \hat{\mathcal{G}}_n] \subseteq \hat{\mathcal{G}}_{m+n}$.

Choose any scalar (the \emph{level}) $k{\in}\CC$, and let $\CC_k$ denote the $1$-dimensional $(\mathcal{G}^{+}\oplus\hat{\mathcal{G}}_0)$-module defined as follows: $\mathcal{G}^{+}$ acts as $0$; $\mathcal{G}=\mathcal{G}{\otimes}t^0$
acts as $0$; $K$ acts as multiplication by the level $k$. The corresponding \emph{Verma-module} is the induced module
\[
V=  V(\mathcal{G}, k)\df \mathcal{U}(\hat{\mathcal{G}})\otimes_{\mathcal{U}(\hat{\mathcal{G}}^{+}{\oplus}\hat{\mathcal{G}}_0)} \CC_k.
\]
where, here and below, $\mathcal{U}$ denotes universal enveloping algebra. Using the PBW theorem and the triangular decomposition for $\hat{\mathcal{G}}$, one sees that $V$ is linearly isomorphic to the symmetric algebra $S(\hat{\mathcal{G}}^{-})$. The \emph{conformal grading} on the symmetric algebra is related to the grading on $\hat{\mathcal{G}}$ in which $a\otimes t^{-n}$ ($n \geq 1$, $a\in \mathcal{G}$) has weight (i.e., degree) $n$ and the \emph{vacuum element} $\mathbf{1}\df 1\otimes1$ has weight $0$.
\begin{equation}
  \label{symm}
V=V(\mathcal{G},k)\cong S(\hat{\mathcal{G}}^{-})=\oplus_{n\geq 0} S(\hat{\mathcal{G}}^{-})_n
\end{equation}
where
\begin{align*}
S(\hat{\mathcal{G}}^{-})_0=&\CC\mathbf{1}, & S(\hat{\mathcal{G}}^{-})_1=& \mathcal{G}\otimes t^{-1}.
\end{align*}

As long as $k$ is a positive integer (the only case that we care about) then $V$ carries the structure of a VOA, and the grading on $V$ induced by the $L(0)$-operator of the Virasoro element is the conformal grading we just described.
An obvious -- though important -- point is that this is \emph{independent of the level $k$}.

\subsection{The quotient VOA $L(\mathcal{G}, k)$ and the Majorization Theorem}
We continue to discuss the VOAs $V\df V(\mathcal{G}, k)$, always with $k$ a positive integer. Up to scalars, $V$ admits a unique nonzero, invariant, bilinear form $b_V$ by a Theorem of Li \cite{LiSymmetric}, however $b_V$ is always degenerate for the values of $k$ under consideration. The \emph{radical} of $b_V$ is the unique maximal $2$-sided ideal in $V$, and we denote the simple quotient VOA by
\[
L(\mathcal{G}, k)\df V(\mathcal{G}, k)/\textrm{Rad}(b_V).
\]
If $\mathcal{G}$ has type $A, B, \ldots, F, G$ and Lie rank $\ell$ we will often denote this VOA by $\mathcal{G}_{\ell,k}$.

A fundamental Theorem for us is  the determination of $\textrm{Rad}b_V$ by Kac \cite{Kac}. See also \cite{LL}, Proposition 6.6.17. To state the result concisely, we need some notation. Let $\Phi$ be the root system of $\mathcal{G}$ and let $\theta \in \Phi$ be the (unique) positive root of maximal height. Let $S_{\theta}\subseteq \mathcal{G}$ be a fundamental $\mathfrak{sl}_2$-subalgebra determined by $\theta$ having a Chevalley basis $\{e_{\theta}, f_{\theta}, h_{\theta}\}$.
\begin{prop}[Kac]\label{propR}
  We have 
 \[
 \Rad(b_V)=\mathcal{U}(\hat{\mathcal{G}})e_{\theta}(-1)^{k+1}\mathbf{1}.
 \]
 $\hfill\Box$
\end{prop}

We are now ready for:

\begin{proof}[Proof of Majorization Theorem \ref{thmmajor}] We have seen that, considered as just a $\ZZ$-graded linear space, $V(\mathcal{G}, k)$ coincides with the graded symmetric algebra \eqref{symm} which does not depend on $k$. From Kac's Theorem it is clear that the radical ideals $R_{\ell,k}\df b_{V(\mathcal{G}, k)}$ are graded subspaces that satisfy
\[
R_{\ell, k'}\subseteq R_{\ell, k} 
\]
for $k' \geq k$. These containments induce surjections of graded linear spaces 
\[
L(\mathcal{G}, k') \longrightarrow L(\mathcal{G}, k)
\]
and this is the statement of Theorem \ref{thmmajor}.
\end{proof}

\subsection{Proof of Theorem \ref{thmBC}}
In order to prove Theorem \ref{thmBC} we may assume that $\ell \geq 3$, and we shall do this. Furthermore, by applying Theorem \ref{thmmajor} we are reduced to proving Theorem \ref{thmBC} in the case $k = 1$, and we shall from now on also assume that this is the case.

Thus we must compare the dimensions of the weight $2$ pieces of the VOAs $C_{\ell,1}$ and $B_{\ell, 1}$. According to Proposition \ref{propR} these are given by the weight $2$ pieces of the graded quotients
\begin{align*}
  C_{\ell,1}=&V(C_{\ell},1)/\mathcal{U}(\hat{\mathcal{C_{\ell}}})e_{\theta C}(-1)^{2}\mathbf{1},\\
  B_{\ell,1}=&V(B_{\ell},1)/\mathcal{U}(\hat{\mathcal{B_{\ell}}})e_{\theta B}(-1)^{2}\mathbf{1},
\end{align*}
where $\theta B$ and $\theta C$ are the highest roots for the  root systems of type $B_{\ell}$ and $C_{\ell}$ respectively.

From the description of the underlying  $\ZZ$-graded space of $V(\mathcal{G},1)$ as a graded symmetric algebra presented in Subsection \ref{SSAppB2}, and because $B_{\ell}$, $C_{\ell}$ are Lie algebras of equal dimension, it follows that the degree $2$ pieces of $V(B_{\ell},1)$ and $V(C_{\ell},1)$ are also equal. Therefore, in order to prove Theorem \ref{thmBC} we must compare the dimensions of the degree $2$ pieces $(\mathcal{U}(\hat{\mathcal{B_{\ell}}})e_{\theta B}(-1)^{2}\mathbf{1})_2$ and $(\mathcal{U}(\hat{\mathcal{C_{\ell}}})e_{\theta C}(-1)^{2}\mathbf{1})_2$. Indeed, we shall prove the next result (and Remark \ref{rmkBC} also follows from this):
\begin{lem}
  \label{lemdimwt2}
 We have
\[
\dim (\mathcal{U}(\hat{\mathcal{B_{\ell}}})e_{\theta B}(-1)^{2}\mathbf{1})_2- \dim(\mathcal{U}(\hat{\mathcal{C_{\ell}}})e_{\theta C}(-1)^{2}\mathbf{1})_2 =2\ell-4.
\]
\end{lem}

The remainder of this Appendix proceeds with the proof of this Lemma. It amounts to a fairly elaborate computation of the dimensions of the $2$ graded spaces in question.

We begin with any simple Lie algebra $\mathcal{G}$. $\mathcal{U}(\hat{\mathcal{G}})$ is spanned by elements of the form
\begin{eqnarray*}
\{a^1(-m_1)\cdots a^r(-m_r) b^1(0)\cdots b^s(0)c^1(n_1)\cdots c^t(n_t) \mid m_i, n_i\geq 1 \}
\end{eqnarray*}
where the Lie algebra elements $a^i, b^i, c^i$ span $\mathcal{G}$, the $m_i$ and $n_i$ are decreasing sequences of integers, $c^i(n_i)$ is the operator induced by $c^i\otimes t^{n_i}$, etc. Now it is well-known that the radical spaces  $(\mathcal{U}(\hat{\mathcal{G}})e_{\theta }(-1)^{2}\mathbf{1})$ contain \emph{no} nonzero elements of degree less than $2$. Thus the weight $2$ piece is the lowest nonzero part. Because the operators $c^i(n_i)$ are \emph{lowering} operators for
$n_i>0$ they must annihilate $e_{\theta}(-1)^2\mathbf{1}$ (a result that can be checked directly). Similarly, the $b^i(0)$ are weight $0$ operators and the $a^i(-m_i)$ are \emph{raising} operators for $m_i>0$. The upshot is that we have
\[
(\mathcal{U}(\hat{\mathcal{G}})e_{\theta}(-1)^{2}\mathbf{1})_2=\hat{\mathcal{G}}_0e_{\theta}(-1)^2\mathbf{1}.
\]

For $b\in\mathcal{G}$ we also have
\[
b(0)e_{\theta}(-1)^2\mathbf{1}{=}2 [b, e_{\theta}](-1)e_{\theta}
\]
and as $b$ ranges over $\mathcal{G}$ we generate in this way $e_{\theta}(-1)^2$ as well as $e_{\gamma}(-1)e_{\theta}$ for positive roots $\alpha$, $\gamma$ such that $\gamma+\alpha = \theta$. Let the number of such positive roots $\gamma$ be denoted by $N=N_{\mathcal{G}}$. This argument shows that
\[
 \dim (\mathcal{U}(\hat{\mathcal{G}})e_{\theta}(-1)^{2}\mathbf{1})_2=1+N_{\mathcal{G}}.
\]

There is a representation-theoretic meaning of the integer $N$. Recall the $\mathfrak{sl}_2$ Lie algebra $\mathcal{S}\df \mathcal{S}_{\theta}=\langle e_{\theta}, f_{\theta}, h_{\theta} \rangle \subseteq \mathcal{G}$, and decompose the adjoint representation as a direct sum of irreducible $\mathcal{S}$-modules. We assert that
\begin{equation}
  \label{decomp}
\mathcal{G}=C_{\mathcal{G}}(\mathcal{S})\oplus\mathcal{S}\oplus_{i=1}^{N} V_i
\end{equation}
where $C_{\mathcal{G}}(\mathcal{S})$ is the \emph{centralizer} of $\mathcal{S}$ and where each $V_i$ is $2$-dimensional. Indeed, the $1$-dimensional summands are all contained in the centralizer and there is at least one $3$-dimensional summand, namely, $\mathcal{S}$ itself. Note that a Cartan subalgebra $\mathcal{H}$ is contained in the sum of these two modules. Let $V_i \subseteq \mathcal{G}$ be any other nonzero irreducible $\mathcal{S}$-submodule. On one hand $V_i$ is spanned by root vectors because $\mathcal{G}$ is, and on the other hand it contains a unique highest weight vector for $e_{\theta}$. Because $\theta$ is the highest root for $\mathcal{G}$ then \emph{every} root vector $v_{\gamma}$ ($\gamma\in\Phi^+$) is annihilated by $e_{\theta}$, and this means that $V_i$ contains \emph{exactly one} positive root vector, call it $v_{\gamma}$, and exactly one negative root vector, which must be $v_{\gamma-\theta}$. Setting $\beta\df \theta-\gamma \in \Phi^+$ we have $\alpha+\beta=\theta$.

This argument shows that $\dim V_i=2$, thereby confirming the decomposition \eqref{decomp}. Note that we obtain such a $V_i$ whenever $\theta = \alpha+\beta$ is decomposed into a sum of two positive roots, so that the number of $2$-dimensional summands in \eqref{decomp} is indeed equal to $N$.

We now find that
\begin{equation}
  \label{halfform}
N_{\mathcal{G}} = \tfrac{1}{2}(\dim\mathcal{G}-\dim C_{\mathcal{G}}(\mathcal{S})-3).
\end{equation}

Finally, 
\begin{lem}
  \label{lemdimBC}
We have
  \begin{enumerate}
\item If $\mathcal{G}=B_{\ell}$ then $C_{\mathcal{G}}(\mathcal{S})\cong A_1\oplus B_{\ell-2}$;
\item If $\mathcal{G}=C_{\ell}$ then $C_{\mathcal{G}}(\mathcal{S})\cong C_{\ell-1}$.
\end{enumerate}
\end{lem}
\begin{proof}
  We tackle the case $C_{\ell}$ first. In standard notation (cf. \cite{H}, Section 12) we choose an orthonormal basis $\{e_i\}$ in Euclidean space $\RR^{\ell}$. A root system of type $C_{\ell}$ may then be chosen to consist of $\{\pm e_i\pm e_j \mid i\neq j\}\cup\{\pm 2e_i\}$, and we have $\theta=2e_1$. Then all roots with indices $i$, $j$ greater than $1$ correspond to elements of $C_{\mathcal{G}}(\mathcal{S})$, and these form a root system of type $C_{\ell-1}$. The assertion of the Lemma in this case follows immediately.

Similarly, a root system of type $B_{\ell}$ may be taken to be $\{\pm e_i\pm e_j\mid i \neq j\}\cup\{\pm e_i\}$ and in this case $\theta = e_1+e_2$. Here, all roots with indices greater than $2$ together with $e_1-e_2$ correspond to elements in the centralizer, and the conclusion is that $C_{\mathcal{G}}(\mathcal{S})\cong A_1\oplus B_{\ell-2}$.
\end{proof}

At last we can compute the needed dimensions using Lemma \ref{lemdimBC} and \eqref{halfform}. We find that
\begin{align*}
N_{B_{\ell}}=& \tfrac{1}{2}((2\ell^2+\ell)-(2(\ell-2)^2+(\ell-2)+3)-3)=4\ell-6 \\
N_{C_{\ell}}=& \tfrac{1}{2}((2\ell^2+\ell)-(2(\ell-1)^2+(\ell-1))-3)=  (-(-4\ell+2+(-1))-3)=2\ell-2
\end{align*}
Therefore $N_{B_{\ell}}-N_{C_{\ell}}=2\ell -4$, and this completes the proof of Lemma \ref{lemdimwt2} and thereby that of Theorem \ref{thmBC} also.
\bibliography{refs}{}
\bibliographystyle{plain}
\end{document}